\documentclass{amsart}

\usepackage{amsfonts,amsmath,amsthm,amssymb,amscd,latexsym,enumerate,etoolbox,mathrsfs,comment}

\usepackage[all,cmtip,pdf]{xy}
\usepackage{thmtools, thm-restate}
\usepackage{tikz-cd}
\usepackage{tikz}
\usetikzlibrary{calc}

\usepackage[all]{xy}

\usepackage[pagebackref,colorlinks,linkcolor=blue,citecolor=blue,urlcolor=blue,hypertexnames=true]{hyperref}

\newcommand{\R}{\mathbb{R}}
\newcommand{\C}{\mathbb{C}} 
\newcommand{\N}{\mathbb{N}}

\newcommand{\Z}{{\mathbb Z}}
\newcommand{\Q}{{\mathbb Q}}

\renewcommand{\phi}{\varphi}

\newcommand{\gpd}{\mathcal{G}}
\newcommand{\hpd}{\mathcal{H}}

\theoremstyle{plain}
    \newtheorem{theorem}{Theorem}[section]
    \newtheorem{lemma}[theorem]{Lemma}
    \newtheorem{corollary}[theorem]{Corollary}
    \newtheorem{proposition}[theorem]{Proposition}
        \newtheorem*{proposition*}{Proposition}
    \newtheorem{conjecture}[theorem]{Conjecture}
\theoremstyle{definition}
    \newtheorem{definition}[theorem]{Definition}
    \newtheorem{example}[theorem]{Example}

    \newtheorem{remark}[theorem]{Remark}

    \newtheorem{question}[theorem]{Question}
\theoremstyle{remark}

\newcommand{\Manoa}{M\=anoa}
\newcommand{\Hawaii}{Hawai\kern.05em`\kern.05em\relax i}

\begin{document}

\title[The homology of groupoid models]{A trace pairing and Elliott invariant for groupoid homology}
\author{Robin J. Deeley}
\address{Robin J. Deeley,   Department of Mathematics,
University of Colorado Boulder
Campus Box 395,
Boulder, CO 80309-0395, USA }
\email{robin.deeley@colorado.edu}
\author{Rufus Willett}
\address{Rufus Willett, Department of Mathematics,
University of Hawaii Manoa
2565 McCarthy Mall, Keller 401A
Honolulu, HI 96822, USA}
\email{rufus@math.hawaii.edu}
\thanks{RJD was partially supported by NSF Grant DMS 2247424 and Simons Foundation Gift MP-TSM-00002896. RW was partially supported by NSF Grant DMS 2247968 and Simons Foundation Gift MP-TSM-00002363}

\begin{abstract}
For an \'{e}tale groupoid, we define a pairing between the Crainic-Moerdijk groupoid homology and the simplex of invariant Borel probability measures on the base space.  The main novelty here is that the groupoid need not have totally disconnected base space, and thus the pairing can give more refined information than the measures of clopen subsets of the base space.

Our principal motivation is $C^*$-algebra theory.  The Elliott invariant of a $C^*$-algebra is defined in terms of $K$-theory and traces; it is fundamental in the long-running program to classify simple $C^*$-algebras (satisfying additional necessary conditions).  We use our pairing to define a groupoid Elliott invariant, and show that for many interesting groupoids it agrees with the $C^*$-algebraic Elliott invariant of the groupoid $C^*$-algebra: this includes irrational rotation algebras and the $C^*$-algebras arising from orbit breaking constructions studied by the first listed author, Putnam, and Strung.  These results can be thought of as establishing a refinement of Matui's HK conjecture for the relevant groupoids.
\end{abstract}

\maketitle

\tableofcontents

\section*{Introduction}
The construction of a $C^*$-algebra from a topological groupoid is now classical \cite{Renault:1980fk}. It has become an active area of research to study which $C^*$-algebras can be obtained from groupoids. The groupoid often has more structure, particularly dynamical structure, that allows one to understand $C^*$-algebraic properties: for example, one can relate simplicity of a groupoid $C^*$-algebra to minimality of the underlying groupoid.

One central question in this area is the relationship between groupoid homology in the sense of Crainic and Moerdijk \cite{Crainic:2000aa} and $C^*$-algebraic K-theory. The HK-conjecture \cite{Matui2016aa, Matui:2017hq} predicts a strong relationship between these two invariants. The precise statement of the HK-conjecture is as follows.
\begin{conjecture}
Suppose that $\mathcal{G}$ is a locally compact, Hausdorff, second countable, \'etale, essentially principal, minimal, ample groupoid. Then
\[ K_*(C^*_r(\mathcal{G})) \cong H_{**}(\mathcal{G}).\footnote{Here and throughout, the subscript ``$_{**}$'' refers to the $\Z/2$-graded homology theory associated to a $\Z$-graded homology theory defined by summing all the even groups, and all the odd groups separately.} \] 
\end{conjecture} 
The HK-conjecture is however false in general \cite{Scarparo:2020aa}, and is even false in the principal case \cite{Deeley2023aa,Chaiser:2025aa}. 

Based on these counterexamples, we reframe the conclusion of the HK-conjecture as a desirable property for a groupoid to have; this desirable property is interesting to investigate for a much wider class of groupoids than that covered by the original statement of the HK-conjecture above. This idea has been explored (for example) in \cite{Bonicke:2021aa, Bruce:2024aa, Miller:2024aa, ProiettiYasashita2023aa}. 

In the present paper, we refine this previous work by including the relevant pairings between traces and K-theory on the one hand, and between the space of invariant probability measures and groupoid homology on the other.  The first basic result of the present paper is as follows; it was previously known in the ample case, i.e.\ when the base space is zero-dimensional.

\begin{proposition}\label{intro pair}
Suppose that $\mathcal{G}$ is a locally compact, Hausdorff, \'etale groupoid with compact base space.  There is a canonical pairing 
$$
\rho_H:T(\gpd)\to \text{Hom}(H_0(\gpd),\R)
$$
between the simplex $T(\gpd)$ of invariant Borel probability measures on the base space, and the zeroth Crainic-Moerdijk homology group $H_0(\gpd,\R)$ of the groupoid. 
\end{proposition}

This allows us to define the \emph{Elliott invariant} of a locally compact, Hausdorff, \'etale groupoid with compact base space to be the quadruple $(H_{**}(\gpd), T(\gpd),\rho_H,[1])$, where $[1]\in H_0(\gpd)$ is the class of the unit.  This is modeled after the Elliott invariant of a unital $C^*$-algebra $A$, which we take in the form\footnote{This is not the original definition, but it is equivalent in the most important cases: see Definition \ref{c* ell} below and the following discussion.} $(A,T(A),\rho_K,[1])$, where $T(A)$ is the tracial state space of $A$, $\rho_K:T(A)\to \text{Hom}(K_0(A),\R)$ is the canonical pairing, and $[1]\in K_0(A)$ is the class of the unit.  We then say that a groupoid $\gpd$ is \emph{HK-good} if there is an isomorphism between these Elliott invariants (see Definition \ref{hk good} for a precise version).  Thus a groupoid being HK-good is a refinement of it satisfying the property in the HK-conjecture: one asks that $K$-theory and homology are isomorphic as in the HK conjecture, and also that there is a homeomorphism between the space of invariant measures and the tracial state space, and that these isomorphisms are compatible with the pairings and the class of the unit.  Note that the question of whether a given groupoid is HK-good makes sense for groupoids that are not minimal, ample, essentially principal or second countable, and indeed there are interesting examples that do not have these properties.  

These constructions and definitions are carried out in Section \ref{Sec:pair}.  In Section \ref{Sec:Irrational}, we compute the pairing directly from the definitions for the transformation groupoids arising from irrational rotation actions on the circle; it gives the answer one expects by analogy with the $C^*$-algebra case.  This computation is actually a special case of more general theorems, proved using more sophisticated machinery later in the paper, but we hope that giving a direct computation builds intuition.

\begin{remark}
(This remark was inspired by comments of Ian Putnam).  If $\gpd$ is a \emph{smooth} \'{e}tale groupoid, one could also define a pairing of $H_*(\gpd)$ with invariant measures on $\gpd^{(0)}$ as follows.  First, use that invariant measures define classes in the zero dimensional cyclic cohomology groups $HC^0(C_c^\infty(\gpd))$ (i.e.\ traces) of the smooth convolution algebra $C_c^\infty(\gpd)$.  Next use the isomorphism of groupoid homology with (periodic) cyclic homology of the smooth convolution algebra  $HP_*(C_c^\infty(\gpd))$ established by Crainic and Moerdijk \cite[Proposition 6.10]{Crainic:2000aa}, and the usual pairing between cyclic homology and cyclic cohomology.  It seems interesting to compare this with our pairing, and also to consider higher-dimensional cyclic cocycles, and maybe and other smooth subalgebras, from this point of view.
\end{remark}

Having established the basic machinery above, the rest of the paper is principally motivated by the following question.

\begin{question}\label{main q}
Which $C^*$-algebras have HK-good groupoid models?
\end{question} 

Of course, if the given $C^*$-algebra does not admit any \'etale groupoid model, then it cannot admit an HK-good model.  $C^*$-algebras with no \'{e}tale groupoid model do exist, see \cite[Section 4]{Li:2017aa}.  Li has shown that all classifiable $C^*$-algebras (i.e.\ those classified by their Elliott invariant as discussed above) admit a twisted groupoid model \cite{Li:2018yv}. However, it is not known if the twist can be removed. We optimistically conjecture that the twist can be removed and / or that a suitable variant of groupoid homology can be defined in the twisted case, and that the groupoid model can be chosen to be HK-good.

\begin{conjecture}\label{main con}
Any unital classifiable $C^*$-algebra admits an HK-good groupoid model.
\end{conjecture} 

Although not formulated this way, there are already partial positive results in the literature. For example, \cite[Theorem 4.10]{Matui:2017hq} essentially establishes Conjecture \ref{main con} for AF groupoids, and \cite[Theorem 4.14]{Matui:2017hq} essentially establishes Conjecture \ref{main con} for $C^*$-algebras associated to shifts of finite type.  Another class of examples comes from \cite[Corollary 4.20]{Bonicke:2021aa}: this shows that if $\gpd$ is an ample, second countable, principal \'{e}tale groupoid with compact base space and dynamic asymptotic dimension at most one, then it is an HK-good model for its groupoid $C^*$-algebra\footnote{\label{bad bbgw}The statement of \cite[Corollary 4.20]{Bonicke:2021aa} also claims that the isomorphism between groupoid homology and $K$-theory ``clearly'' restricts to a bijection on positive elements.  This is not clear to us, or to the (other) authors of \cite{Bonicke:2021aa}: the statement about positive cones seems likely to be true, but needs additional arguments.}. As an application of this result, Reardon \cite{Reardon:PhDthesis} has shown that the groupoids constructed by Putnam \cite{Putnam:2018aa} have dynamic asymptotic dimension one and hence are HK-good. Furthermore, any results showing that the HK-conjecture holds for a groupoid whose reduced $C^*$-algebra is purely infinite implies that the given groupoid is HK-good (there are no traces/measures/pairings to consider in this case) up to checking that the isomorphism between $K$-theory and homology preserve the class of the unit, see for example \cite{Ortega:2020aa}.  Finally, we note that the results of Guo-Proietti-Wang \cite{Guo:2024ab} on mapping tori associate an Elliott invariant to actions of $\Z^d$ on spaces in a way that is closely related to ours, although their groupoid Elliott invariant is defined in more index-theoretic terms.

While our broad theme is Question \ref{main q} above, the immediate goal of the present paper is to study specific examples coming from actions of the integers, and from orbit breaking constructions applied to actions of the integers.  Here are some sample theorems.

\begin{theorem}
Let $X$ be a compact Hausdorff space equipped with an action of the integers.  Assume that $X$ is a $d$-sphere, a $d$-torus, or has covering dimension at most three.  Then the transformation groupoid $\Z\ltimes X$ is HK-good.
\end{theorem}

See Sections \ref{sec:HomZAct}, \ref{Sec:int pair}, and \ref{Sec:crossed} for more details: we actually go a bit further than the spaces above, in order to cover the crossed products constructed by the first listed author, Putnam, and Strung in \cite{Deeley:2024aa}.  

We also cover the orbit-breaking systems from \cite{Deeley:2024aa} as in the next result: see Sections \ref{Sec:point}, and \ref{Sec:Cantor} for details.

\begin{theorem}
The groupoids obtained by orbit breaking from point-like systems and from Floyd-type systems (both as in \cite{Deeley:2024aa}) are HK-good, as long as the subspace $Y$ one breaks along has covering dimension at most three.
\end{theorem}

To obtain these results, we also explore the relationship between the pairings, ensure that groupoid homology works as expected for groupoids with infinite dimensional base space, and develop machinery for computing homology via long exact sequences, spectral sequences, and a relation to group hyperhomology. These results should be useful in other applications and build on previous work in \cite{Crainic:2000aa, Deaconu:2021aa, Matui:2017hq, Miller:2025aa, Miller:2024aa, Ortega:2020aa, Pimsner:1985aa}.  

More specifically, for actions by the integers, we proceed via a general construction.  First, in Section \ref{sec:HomZAct} we give a general computation of the groupoid homology; this is a close analogue of the Pimsner-Voiculescu sequence from $K$-theory (compare with \cite{Deaconu:2021aa, Ortega:2020aa}). In Section \ref{Sec:int pair} we show moreover that groupoid homology for an integer action is isomorphic to the cohomology of the mapping torus; the corresponding result for $K$-theory on the other hand is well-known, so we are able to construct a `Chern character' from the $K$-theory of the crossed product to the (rational) homology of the groupoid using the classical Chern character for the mapping torus.  We show that this Chern character is compatible with the pairings with invariant measures using ideas of Connes \cite{Connes:1981ab}.  We show moreover that it is an integral isomorphism in many cases (but not always), roughly corresponding to actions on spaces where the classical Chern character is an integral isomorphism.  This establishes that many such groupoids are HK-good: we discuss explicit examples in Section \ref{Sec:crossed}.

For the orbit breaking examples, one must use spaces of low dimension and the fact that the Chern character is an isomorphism from the $K$-theory of the space to its cohomology. In this case, the relevant spaces of low dimension are where the orbit breaking is done. The key tool is a long exact sequence in homology (which is based on work of Matui in the ample case \cite{Matui:2022aa}) that is analogous to a long exact sequence in $K$-theory due to Putnam, \cite{Putnam:1989hi}, \cite[ Example 2.6]{Putnam:1998aa}, and \cite{Putnam:2021aa}.  The extension of Matui's work to the non-ample case is carried out in Section \ref{incl sec}.  This is then specialized to orbit breaking examples in Section \ref{obsec}, and further specialized to point-like systems and Cantor-like systems in Sections \ref{Sec:point} and \ref{Sec:Cantor} respectively.

As mentioned above, we have dropped the requirement that the groupoid is ample. As such, there are many examples of groupoids that are not HK-good; this occurs essentially whenever one replaces the low dimensional space in our constructions with one where the Chern character fails to be an integral isomorphism. This phenomenon is discussed in Example \ref{Ex:CrossNotGood} and Remarks \ref{rem:HKbad} and \ref{Rem:NotHKgoodCantor}.   

On the other hand, our results indicate that for a given $C^*$-algebra there can be many different HK-good groupoid models. As an explicit example, we discuss three HK-good models for the irrational rotation $C^*$-algebra (it will admit other HK-good models as well): see Example \ref{Ex:ManyHKgood} for further details.  Slightly surprisingly, we show these three HK-good models for the irrational rotation algebra have different groupoid homology: the groupoid homology only becomes the same after reducing from a $\Z$-graded theory to a $\Z/2$-graded theory in order to match $K$-theory.

Finally, it is also worth mentioning that a given $C^*$-algebra can have both HK-good models and others that are not HK-good. The reader can see Example \ref{Ex:CrossNotGood} and \ref{Ex:CrossedGoodSpecialCase} for an example where this situation occurs in the context of minimal integer actions. 


Let us make some brief comments on the methods we use and the background needed to read this paper.  Crainic-Moerdijk homology for non-ample groupoids necessitates some amount of algebraic topology and homological algebra.  Indeed, the basic definitions are in terms of double complexes built from sections of equivariant sheaves; thus we use sheaf theory, and the presence of double complexes naturally leads us to use spectral sequences.  On the other hand, we avoid triangulated and derived categories: in general the methods in this paper do not really require any substantial machinery from sheaf theory or homological algebra that was developed after the 1950s.  We have tried to summarize the relevant background in Section \ref{secPrelim} below, and have also included some more material on homological algebra, group hyperhomology, and Chern characters in three appendices.

Let us finally remark that we believe our construction of a Chern character for integer actions is a shadow of something much more general.  Indeed, we have constructed a Chern character for a large class of transformation groupoids in \cite{Deeley:2024aa}, and very interesting recent results of Proietti-Yamashita \cite{Proietti:2025aa} construct a Chern character for many ample groupoids.  The `best' way to establish that a groupoid is HK-good should be to construct an appropriate Chern character, and show it is compatible with the relevant pairings.  We hope to see further progress on these issues: a general Chern character could open up the computation of the Elliott invariant of groupoid $C^*$-algebras to homological methods, and these are typically more powerful than proceeding directly through $K$-theory. 

\section*{Acknowledgments}
Both authors thank the University of Colorado Boulder and the University of \Hawaii\ at \Manoa\ for facilitating this collaboration. The first listed author thanks Alistair Miller and Karen Strung for some discussions that in part led to these results. These discussions took place while the first listed author was visiting the Fields Institute for the Thematic Program on Operator Algebras and Applications in the fall of 2023.  Both authors thank Valerio Proietti and Ian Putnam for useful comments on a preliminary draft.  We also thank Mikael R\o{}rdam for pointing out the issue in footnote \ref{bad bbgw} to the second author.



\section{Notation and Preliminaries} \label{secPrelim}

We start by introducing notational conventions for \'{e}tale groupoids: see for example \cite{Sims:2017aa} for background on \'{e}tale groupoids and their $C^*$-algebras. For a groupoid $\gpd$ we write $\gpd^{(0)}$ for the base space (also called the unit space) and $r,s:\gpd\to \gpd^{(0)}$ for the range and source maps.  An ordered pair $(g, h) \in\gpd\times \gpd$ is \emph{composable} if $s(g) = r(h)$ and its composition is denoted by $gh$.  The inverse of $g \in \gpd$ is denoted $g^{-1}$.  All groupoids considered in the present paper will be topological groupoids: we always assume the topologies are locally compact and Hausdorff.  A \emph{bisection} is a subset $B$ of $\gpd$ such that $r|_B:B\to r(B)$ and $s|_B:B\to s(B)$ are homeomorphisms.  We will also always assume that $\gpd$ \emph{\'etale}, meaning that there is a basis for the topology consisting of open bisections. Under these assumptions, $\gpd^{(0)}$ is a closed and open subset of $\mathcal{G}$ and there is a canonical Haar system given by counting measures. 

To a groupoid $\mathcal{G}$ satisfying the assumptions above one can associate its \emph{reduced groupoid $C^*$-algebra} using the method in \cite{Renault:1980fk}; see also \cite{Sims:2017aa} for an exposition focused on the (much simpler) special case of \'{e}tale groupoids.  The resulting $C^*$-algebra is denoted by $C^*_r(\mathcal{G})$.  

A groupoid $\mathcal{G}$ is \emph{ample} if its base space is totally disconnected (e.g., the Cantor set).  It  is \emph{principal} if for each $x \in \mathcal{G}^{(0)}$, $s^{-1}(x)\cap r^{-1}(x)=\{x\}$.  A subset $F$ of $\gpd^{(0)}$ is \emph{invariant} if $s(r^{-1}(F))=F$; in this case, we write $\gpd|_F$ for $r^{-1}(F)$, which is a subgroupoid of $\gpd$ with the induced operations.  Note that $\gpd|_F$ need not be locally compact or \'{e}tale when equipped with the subspace topology; it does, however, have these properties if $F$ is open.

A \emph{$\gpd$-space} is a locally compact Hausdorff topological space $X$ equipped with a continuous open map $\rho:X\to G^{(0)}$ called the \emph{anchor map} and a continuous \emph{action map}
$$
\gpd_s\!\times_\rho X\to X
$$
denoted $(g,x)\mapsto gx$ that satisfies $(gh)x=g(hx)$ where defined; here $\gpd_s\!\times_\rho X$ denotes the fibred product $\{(g,x)\in \gpd\times X\mid s(g)=\rho(x)\}$ equipped with the topology it inherits from $\gpd\times X$.

Many of the groupoids considered in this paper are constructed from group actions. Let $X$ be a locally compact Hausdorff space equipped with an action of a discrete group $G$ ($G$ will almost always be $\Z$ in this paper) by homeomorphisms. The associated \emph{transformation groupoid} is denoted $G\ltimes X$.  It is defined to be $G \times X $ as a topological space, with associated algebraic operations determined by:
\begin{enumerate}[(i)]
\item the base space is $\{e\}\times X\subseteq \gpd$ where $e$ is the identity element of $\gpd$ (usually, we just identify the base space with $X$);
\item the range and source of $(\gamma,x)$ are $\gamma x$ and $x$ respectively;
\item the composition $(\gamma, \alpha x) \cdot (\alpha, y)$ is equal to $(\gamma \alpha, y)$;
\item the inverse of $(\gamma, x)$ is $(\gamma^{-1}, \gamma x)$.
\end{enumerate}

We now recall the \emph{Crainic-Moerdijk homology} of a groupoid $\mathcal{G}$ following \cite[Section 3]{Crainic:2000aa}.  We will need to work in terms of sheaves: see for example \cite{Bredon:1997aa} or \cite{Kashiwara:1990aa} for general background on sheaves and sheaf cohomology.  We recall that a sheaf $\mathcal{S}$ over a space $X$ can be specified by either: 
\begin{enumerate}[(i)]
\item for each open set, one specifies an abelian group of sections $\Gamma(U;\mathcal{S})$ satisfying appropriate conditions as the open sets vary (see for example \cite[pages 6-7]{Bredon:1997aa} or \cite[Definition 2.2.1]{Kashiwara:1990aa});
\item an \emph{\'{e}tale space}, which is a (possibly non-Hausdorff) topological space $\mathcal{S}$ equipped with a local homeomorphism $\pi:\mathcal{S}\to X$ satisfying appropriate conditions (see for example \cite[Definition 1.2]{Bredon:1997aa}).
\end{enumerate}
In particular, for us a sheaf will always mean a sheaf of abelian groups.  For a sheaf $\mathcal{S}$ on a locally compact Hausdorff space $X$ and an open subset $U$ of $X$, we write $\Gamma_c(U;\mathcal{S})$ for the compactly supported sections of $\mathcal{S}$ over $U$: see for example \cite[pages 21-22]{Bredon:1997aa} or \cite[Line (2.5.2) on page 103]{Kashiwara:1990aa}.  

The most important example of a sheaf for us will be the sheaf $\mathcal{Z}$ whose sections $\Gamma(U;\mathcal{Z})$ are locally constant functions from $U$ to $\Z$.  Another useful example to bear in mind is the sheaf $\mathcal{R}$ whose sections $\Gamma(U;\mathcal{R})$ of continuous real-valued functions on $X$.

We recall some standard definitions.

\begin{definition}\label{csoft}(Compare for example \cite[Section II.9]{Bredon:1997aa} or \cite[Section 2.5]{Kashiwara:1990aa}.)
A sheaf $\mathcal{S}$ over a locally compact Hausdorff space $X$ is \emph{$c$-soft} if for any compact $K\subseteq X$, open set $U\supseteq K$, and section $s\in \Gamma(U;\mathcal{S})$ there is an open set $V$ with $K\subseteq V\subseteq U$, and a section $t\in \Gamma(X;\mathcal{S})$ such that the restrictions of $s$ and $t$ to $V$ are the same. 
\end{definition}

For example, the sheaf $\mathcal{R}$ discussed above is $c$-soft.  Indeed, given $K\subseteq U$ and $s\in \Gamma(U;\mathcal{R})$, choose open sets $V,W$ with $K\subseteq V\subseteq \overline{V}\subseteq W$ and $\overline{W}$ compact.  Use Urysohn's lemma to choose a continuous function $g:X\to [0,1]$ that is equal to one on $V$, and zero outside $W$ (we use here that $\overline{W}$ is normal, even though $X$ might not be).  Then define $t$ to be equal to $gs$ on $U$, and zero outside $U$.  On the other hand, the sheaf $\mathcal{Z}$ is $c$-soft if $X$ is totally disconnected, but usually not for more general spaces.





We now specialize to equivariant sheaves on groupoids.  

\begin{definition}\label{gsheaf}(Compare \cite[2.1]{Crainic:2000aa}).
Let $\gpd$ be a locally compact, Hausdorff, \'{e}tale groupoid.  
A \emph{$\mathcal{G}$-sheaf} is a sheaf $\mathcal{S}$ on $\mathcal{G}^{(0)}$ such that the \'{e}tale space of $\mathcal{S}$ is a $\mathcal{G}$-space in such a way that the the anchor map $\rho:\mathcal{S}\to \mathcal{G}^{(0)}$ and the \'{e}tale space structure map $\pi:\mathcal{S}\to \gpd^{(0)}$ are the same.
\end{definition}

For example, the sheaves $\mathcal{Z}$ and $\mathcal{R}$ defined above are always $\gpd$-sheaves on $\gpd^{(0)}$.  Note also that if $\gpd=G\ltimes X$ is a transformation groupoid associated to an action of a group $G$ on a locally compact space $X$, then a $\gpd$-sheaf on $\gpd^{(0)}=X$ is the same thing as a \emph{$G$-sheaf} (also called an \emph{equivariant sheaf}) on $X$ in the sense of \cite[Section 5]{Grothendieck:1957aa}.

Following \cite[1.7]{Crainic:2000aa}, for each $n\in \N$ write $\mathcal{G}^{(n)}$ for the collection of `composable strings'
$$
x_0 \stackrel{g_1}{\leftarrow} x_1  \stackrel{g_2}{\leftarrow} \cdots  \stackrel{g_n}{\leftarrow} x_n
$$
where $x_i\in \mathcal{G}^{(0)}$ and each $g_i\in \mathcal{G}$ satisfies $r(g_i)=x_{i-1}$ and $s(g_i)=x_i$.  Note that $\mathcal{G}^{(0)}$ has the usual meaning, and $\gpd^{(1)}=\gpd$.  We equip $\gpd^{(n)}$ with the topology it inherits as a subspace of $\mathcal{G}^n$ for $n\geq 1$, and the usual topology if $n=0$.    For each $n$, define 
\begin{equation}\label{taun map}
\tau_n:\mathcal{G}^{(n)}\to \mathcal{G}^{(0)},\quad (x_0 \stackrel{g_1}{\leftarrow} \cdots  \stackrel{g_n}{\leftarrow} x_n)\mapsto x_0
\end{equation}
and for a $\gpd$-sheaf $\mathcal{S}$ on $\mathcal{G}^{(0)}$, define 
\begin{equation}\label{s sub n}
\mathcal{S}_n:=\tau_n^*\mathcal{S}
\end{equation} 
to be the pullback sheaf (also called the inverse image sheaf: see for example \cite[page 12]{Bredon:1997aa} or \cite[Definition 2.3.1]{Kashiwara:1990aa}). 


For the next definition, we work in terms of double complexes: for definiteness, we adopt the conventions on double complexes and associated total complexes in \cite[pages 7-9]{Weibel:1995ty}.

\begin{definition}(Crainic and Moerdijk, \cite[3.4]{Crainic:2000aa}) \label{gpd hom}
Let $\mathcal{A}$ be a $\mathcal{G}$-sheaf on $\mathcal{G}^{(0)}$.  Following Crainic and Moerdijk \cite[3.3]{Crainic:2000aa} (or use Lemma \ref{inj ginj}) there exists a resolution
\begin{equation}\label{reso}
0\to \mathcal{A}\to \mathcal{S}^0\to \mathcal{S}^1\to \cdots 
\end{equation}
of $\mathcal{A}$ by $c$-soft $\mathcal{G}$-sheaves; in other words, the sequence in line \eqref{reso} is exact in the category of $\gpd$-sheaves, and each $\mathcal{S}^i$ is $c$-soft.  

The \emph{groupoid homology of $\mathcal{G}$ with coefficients in $\mathcal{A}$}, denoted $H^*(\gpd;\mathcal{A})$, is the homology of the (direct sum total complex associated to the) double complex 
\begin{equation}\label{tot com}
\xymatrix{ \Gamma_c(\mathcal{G}^{(0)};\mathcal{S}^0_0) \ar[d] & \Gamma_c(\mathcal{G}^{(0)};\mathcal{S}^0_1) \ar[d] \ar[l]^\delta & \Gamma_c(\mathcal{G}^{(2)};\mathcal{S}^0_2) \ar[d] \ar[l]^\delta & \cdots \ar[l]^-\delta \\
\Gamma_c(\mathcal{G}^{(0)};\mathcal{S}^1_0) \ar[d] & \Gamma_c(\mathcal{G}^{(0)};\mathcal{S}^1_1) \ar[d] \ar[l]^\delta & \Gamma_c(\mathcal{G}^{(2)};\mathcal{S}^1_2) \ar[d] \ar[l]^\delta & \cdots \ar[l]^-\delta \\
\vdots & \vdots & \vdots & & \\
}
\end{equation}
where the vertical arrows are functorially induced from the resolution in line \eqref{reso}, and the horizontal arrows are as in \cite[1.7]{Crainic:2000aa}\footnote{We will not need the general definition, but discuss some special cases in Remark \ref{horiz arrows} below.}.
\end{definition}

Let us give some remarks on the definition.

\begin{remark}
Crainic and Moerdijk defined groupoid homology $H_*(\gpd;\mathcal{A})$ under a finite-dimensionality assumption on the base space that guarantees the existence of a finite length resolution of $\mathcal{A}$ by $c$-soft $\gpd$-sheaves as in line \eqref{reso}; they then showed that it does not depend on the choice of finite-length resolution.  Note that under their finite-dimensionality assumption there is no difference between the direct sum and direct product total complex, so Crainic and Moerdijk do not make explicit which they are using.  We show that groupoid homology does not depend on the choice of resolution as in line \eqref{reso} for locally compact, Hausdorff, \'{e}tale groupoids with possibly infinite-dimensional base space in Proposition \ref{cm id wd}.  
\end{remark}

\begin{remark}
To be specific about conventions: the complex in line \eqref{tot com} is a fourth quadrant complex with the columns indexed by $0,1,2,...$ and rows indexed by $0,-1,-2,...$.  In particular, for $n\in \Z$, the $n^{\text{th}}$ chain group of the total complex is given by 
$$
\bigoplus_{p-q=n} \Gamma_c(\gpd^{(p)};\mathcal{S}_p^q)
$$
with total differential defined as the sum of the vertical and horizontal differentials as on \cite[page 8]{Weibel:1995ty}.  Note that $H_n(\gpd;\mathcal{A})$ will typically be non-zero for both positive and negative values of $n$. 
\end{remark}

\begin{remark}\label{horiz arrows}
We will not need to describe the horizontal arrows in line \eqref{tot com} in full generality, but let us mention a special case for relatively simple sheaves.  Assume for simplicity that $\mathcal{S}$ is a sheaf of $\C$-valued functions (maybe satisfying some additional properties) on $\gpd^{(0)}$, whence each $\mathcal{S}^i$ also has sections given by $\C$-valued functions.  For $n\geq 1$ and $i\in \{0,...,n\}$, following \cite[1.7]{Crainic:2000aa}, we define $d_i:\mathcal{G}^{(n)}\to \mathcal{G}^{(n-1)}$ by 
$$
d_i(x_0 \stackrel{g_1}{\leftarrow} \cdots  \stackrel{g_n}{\leftarrow} x_n):=\left\{\begin{array}{ll} x_1 \stackrel{g_2}{\leftarrow} \cdots  \stackrel{g_n}{\leftarrow} x_n~, & i=0 \\
x_0 \stackrel{g_1}{\leftarrow} \cdots x_{i-1} \stackrel{g_ig_{i+1}}{\longleftarrow} x_{i+1}\cdots   \stackrel{g_n}{\leftarrow} x_n~, & 1\leq i<n \\
x_0 \stackrel{g_1}{\leftarrow} \cdots    \stackrel{g_{n-1}}{\leftarrow} x_{n-1}~, & i=n\end{array}\right..
$$
For $t\in \Gamma_c(\gpd^{(n)};\mathcal{S}_n^i)$ we define 
$$
(\delta_i t)(\underline{h}):=\sum_{\underline{g}\in d_i^{-1}(\underline{h})} t(\underline{g})
$$
(this makes sense as $d_i$ is a local homeomorphism and $t$ is compactly supported), and $\delta:=\sum_{i=0}^n (-1)^i\delta_i$.  Note in particular that for $t\in \Gamma_c(\gpd^{(1)};\mathcal{S}_1^i)$ and $x\in \gpd^{(0)}$ we have
\begin{equation}\label{first bound}
(\delta t)(x)=\sum_{g\in s^{-1}(x)}t(g) - \sum_{g\in r^{-1}(x)}t(g).
\end{equation}
\end{remark}

Going back to generalities, we make the following notational conventions.

\begin{definition}\label{z coeff}
Define $H_*(\mathcal{G}):=H_*(\gpd;\mathcal{Z})$.  Write $H_{**}(\gpd)$ for the $\Z/2$-graded homology theory with even and odd groups defined by
\begin{equation}\label{z2 grade}
H_{ev}(\gpd):=\bigoplus_{m\in \Z} H_{2m}(\gpd) \quad \text{and}\quad H_{od}(\gpd):= \bigoplus_{m\in \Z} H_{2m+1}(\gpd).
\end{equation}

\end{definition}

We will also need to discuss compactly supported sheaf cohomology of a locally compact Hausdorff space $X$: see for example \cite{Bredon:1997aa} or \cite{Kashiwara:1990aa} for background.  The sheaf cohomology of $X$ can be defined by taking a resolution 
$$
0\to \mathcal{Z}\to \mathcal{S}^0\to \mathcal{S}^1\to \cdots 
$$ 
of the sheaf $\mathcal{Z}$ of locally constant $\Z$-valued functions on $X$ by $c$-soft $\gpd$-sheaves, and defining $H_c^*(X)$ to be the homology of the associated complex
$$
\Gamma_c(X;\mathcal{S}^0)\to\Gamma_c(X;\mathcal{S}^1)\to \cdots 
$$
of compactly supported sections.  We will also use the notation $H^{**}_c(X)=H_c^{ev}(X)\oplus H^{od}_c(X)$ for the corresponding $\Z/2$ graded theory and its even and odd parts.  If $X$ is compact, we will generally drop the subscript ``$_c$'' and just write $H^*(X)$.   For `reasonable' spaces $X$, sheaf cohomology agrees with any other standard definition of cohomology.  For more exotic spaces, sheaf cohomology generally behaves better than other standard theories for our purposes: the interested reader might compare the sheaf cohomology of the Cantor space $X$ (for which $H^0(X)$ identifies with the group of continuous $\Z$-valued functions $C(X,\Z)$) with its singular cohomology (for which $H^0(X)$ identifies with the group of all functions from $X$ to $\Z$).

We conclude this section with two basic examples of groupoid homology to help orient the reader.  In both cases, the claimed identifications are true essentially by definition.

\begin{example}\label{spex}
Let $X$ be a locally compact Hausdorff space considered as a groupoid.  Then for any $n\in \Z$, $H_n(X)\cong H^{-n}_c(X)$, where the left hand side is Crainic-Moerdijk groupoid homology, and the right hand side is sheaf cohomology.  Thus groupoid homology of spaces is the same as compactly supported sheaf cohomology, but `reflected' to negative degrees.  As the notation on the left hand side looks like the homology of the space, this might cause confusion in some contexts; however, in this paper we will have no need to consider homology of spaces, Crainic-Moerdijk homology of trivial groupoids, or groupoid cohomology at all, so this should not lead to confusion.  
\end{example}

\begin{example}\label{gpex}
Let $\gpd=G$ be a discrete group.  Then $H_*(\gpd)$ agrees with the usual group homology $H_*(G)$, and this time the grading degrees match: this is essentially true by definition, or see Corollary \ref{cm hh} below for a proof of something more general.  
\end{example}

Combining these Examples \ref{spex} and \ref{gpex}, it is useful for intuition to think of groupoid homology as being built from `cohomology of spaces in negative degrees, and cohomology of groups in positive degrees'.

Let us finally mention some notational conventions on abstract (co)chain complexes in an abelian category (typically, the category of modules over a ring, or of $\gpd$-sheaves for some groupoid).  We will typically write an abstract chain complex as $C_\bullet$, with notation like ``$C_{\bullet+1}$'' meaning the chain complex whose $i^\text{th}$ object is $C_{i+1}$, and similarly for ``$C_{-\bullet}$'' with $i^\text{th}$ object $C_{-i}$ and so on.  We will also write $C^\bullet$ for cochain complexes.  Our (co)chain complexes will always be indexed by $\Z$, but we will occasionally write something like ``$(C_q)_{q\geq 0}$'' for a chain complex: this should be taken to mean a chain complex indexed by $\Z$, where $C_q=0$ for $q<0$.  

A map $f:C_\bullet\to D_\bullet$ between chain complexes is a sequence of maps $f_q:C_q\to D_q$ that is compatible with the boundary maps.  Such a map $f$ induces a map on homology, and $f$ is a \emph{quasi-isomorphism} if the induced map on homology is an isomorphism.

\section{A trace pairing and Elliott invariant for groupoids}\label{Sec:pair}

Our main goal in this section is to define a pairing between the space of invariant probability measures on the base space of a groupoid and its zeroth homology group.  From there we define an `Elliott invariant' for groupoids modeled on the usual Elliott invariant for $C^*$-algebras.  

To motivate this, let us first recall a convenient form of the Elliott invariant for $C^*$-algebras.  This is purely motivation, and will not be used in the rest of the paper.

\begin{definition}\label{c* ell}
Let $A$ be a unital $C^*$-algebra.  Then its \emph{(weak) Elliot invariant} is the quadruple 
$$
(K_*(A),[1]_K,T(A),\rho_K)
$$
where $K_*(A)=K_0(A)\oplus K_1(A)$ is the $\Z/2$-graded $K$-group of $A$, $T(A)$ is its tracial state space equipped with the weak-$*$ topology and affine structure it inherits from the dual $A^*$, $\rho_K:T(A)\to \text{Hom}(K_0(A),\R)$ is the pairing between traces and $K$-theory, and $[1]_K\in K_0(A)$ is the class of the unit.
\end{definition}

This is essentially\footnote{It is not actually the same, as those authors use the space $\text{Aff}(T(A))$ of continuous affine real-valued functions on $T(A)$ in place of $T(A)$; either of the versions from Definition \ref{c* ell} above or \cite[Definition 2.3]{Carrion:2020aa} is recoverable from the other, however.} the same thing as the invariant denoted $KT_u$ in \cite[Definition 2.3]{Carrion:2020aa}.  It is not the classical Elliott invariant, which also takes the order on $K$-theory into account. However, it agrees with the classical version for the so-called `classifiable' $C^*$-algebras as defined below: see for example \cite[Discussion around Definitions 2.2 and 2.3]{Carrion:2020aa}. 

\begin{definition}\label{class class}
A unital $C^*$-algebra is \emph{classifiable} if it is simple, separable, nuclear, satisfies the UCT, and is Jiang-Su stable.
\end{definition}

The precise meanings of the list of adjectives appearing above is not important for this paper: readers who are not experts in $C^*$-algebra classification theory can just read ``classifiable'' as ``nice and well-studied''.  If one is only looking at classifiable $C^*$-algebras then it is reasonable to drop the adjective ``weak'' in ``(weak) Elliott invariant'', which is why we have included it in parentheses; we will typically drop it.

The name ``classifiable'' comes from the fact that these are exactly those $C^*$-algebras that are classified by their (weak) Elliott invariant, i.e.\ the $C^*$-algebras are isomorphic if and only if their Elliott invariants are isomorphic.  This is the culmination of the 40-year program outlined in \cite{Elliott:1995dq}.  It is difficult to summarize the huge amount of work that went into this, but let us say that the program was completed through a combination of \cite{Kirchberg-ICM,Phillips-documenta} (in the purely infinite case) and \cite{Elliott:2015fb,Tikuisis:2015kx} (in the stably finite case).    The preprint \cite{Carrion:2020aa} contains a historical summary, as well as a different approach in the stably finite case.

We now turn back to groupoids, and our main goal of pairing invariant measures with the zeroth homology group.  Throughout the rest of this section $\mathcal{G}$ is a locally compact, Hausdorff, \'{e}tale groupoid.

The following definition is based on \cite[Definition 2.3.8 and Exercise 2.3.9]{Renault:2009zr}.

\begin{definition}\label{inv meas}
For a Borel measure $\mu$ on $\mathcal{G}^{(0)}$, define a Borel measure $r^*\mu$ on $\mathcal{G}$ by setting 
$$
\int_{\gpd}  f d(r^*\mu):=\int_{\gpd^{(0)}} \sum_{g\in r^{-1}(x)}f(g)d\mu(x)
$$
for all $f\in C_c(\gpd)$.  Define $s^*\mu$ analogously.  A Borel measure $\mu$ on $\mathcal{G}^{(0)}$ is \emph{invariant} if $r^*\mu=s^*\mu$.
\end{definition}

Using that there is a basis for the topology consisting of open bisections, one checks that $\mu$ is invariant if and only if for any open bisection $B$, $\mu(s(B))=\mu(r(B))$.

\begin{definition}\label{tg}
Assume that $\mathcal{G}^{(0)}$ is compact.\footnote{In the non-compact case, it is probably more reasonable to consider measures with possibly infinite mass.  We do not address this in the current paper.}  Define $T(\mathcal{G})$ to be the space of invariant Borel probability measures on $\mathcal{G}^{(0)}$ equipped with the weak-$*$ topology and affine structure it inherits as a subset of the dual of $C(\mathcal{G}^{(0)})$.
\end{definition}

Let now $\mathcal{B}$ denote the $\mathcal{G}$-sheaf on $\mathcal{G}^{(0)}$ whose sections over an open set $U\subseteq \mathcal{G}^{(0)}$ consist of all Borel functions $f:U\to \C$ whose restriction to every compact subset of $U$ is bounded\footnote{It might seem more natural to define sections to be bounded Borel functions on $U$, but this does not define a sheaf.}.  

\begin{definition}\label{top res}
A resolution 
\begin{equation}\label{top res seq}
0\to \mathcal{Z}\to \mathcal{S}^0\to \mathcal{S}^1 \to \cdots 
\end{equation}
of the sheaf $\mathcal{Z}$ of locally constant $\Z$-valued functions on $\mathcal{G}^{(0)}$ by $c$-soft $\gpd$-sheaves is said to be \emph{Borel} if $\mathcal{S}^0$ is a subsheaf of $\mathcal{B}$ that contains $\mathcal{Z}$, and the map $\mathcal{Z}\to \mathcal{S}^0$ appearing in the resolution is the inclusion.  
\end{definition}

Given a Borel resolution of $\mathcal{Z}$ as above, let 
$$
C_0(\mathcal{G}):=\bigoplus_{p=0}^\infty \Gamma_c(\mathcal{G}^{(p)};\mathcal{S}_p^p)
$$
denote the zeroth chain group arising from the double complex as in line \eqref{tot com} above.  We define a pairing 
\begin{equation}\label{pre pair}
T(\mathcal{G})\times C_0(\mathcal{G})\to \C,\quad (\mu,(a_p)_{p=0}^\infty )\mapsto \int_{\mathcal{G}^{(0)}} a_0d\mu.
\end{equation}

\begin{proposition}\label{pair good}
Let $\mathcal{G}$ be a locally compact, Hausdorff, \'{e}tale groupoid with compact base space.  Then the pairing defined above descends to a well-defined pairing 
$$
T(\mathcal{G})\times H_0(\mathcal{G})\to \C.
$$
(This includes the statement that Borel resolutions always exist, and that the pairing does not depend on the choice of Borel resolution).
\end{proposition}

\begin{proof}
We first show that given a Borel resolution as in line \eqref{top res seq}, the map in line \eqref{pre pair} descends to $H_0(\gpd)$.  Indeed, it suffices for this to show that if $t\in \Gamma_c(\gpd^{(1)};\mathcal{S}_1^0)$ is a section and $\delta:\Gamma_c(\gpd^{(1)};\mathcal{S}_1^0)\to \Gamma_c(\gpd^{(0)};\mathcal{S}_0^0)$ is the relevant horizontal boundary map from line \eqref{tot com}, then for any $\mu \in T(\gpd)$ we have $\mu(\delta t)=0$.  From the formula for $\delta$ in line \eqref{first bound}, we have 
$$
(\delta t)(x)=\sum_{g\in s^{-1}(x)}t(g) - \sum_{h\in r^{-1}(x)}t(h).
$$
Hence with notation as in Definition \ref{inv meas}
\begin{align*}
\mu(\delta t) & =\int_{\gpd^{(0)}} (\delta t)(x)d\mu(x) \\ & =\int_{\gpd^{(0)}}\sum_{r\in s^{-1}(x)}t(g)d\mu(x) - \int_{\gpd^{(0)}}\sum_{h\in r^{-1}(x)}t(h)d\mu(x) \\ & =\int_{\gpd} t d(s^*\mu)-\int_{\gpd}td(r^*\mu).
\end{align*}
As $\mu$ is invariant, $r^*\mu=s^*\mu$, so this is zero.  

We now show that Borel resolutions exist, and in fact construct a canonical Borel resolution.  Let $0\to \mathcal{Z}\to \mathcal{B}$ be the canonical inclusion.  Lemma \ref{inj ginj} gives an exact sequence of $\gpd$-sheaves 
\begin{equation}\label{can top res}
0\to \mathcal{Z}\to \mathcal{B}\to \mathcal{I}^1\to \mathcal{I}^2\to \cdots 
\end{equation}
where each $\mathcal{I}^i$ is injective as a $\gpd$-sheaf, and also $c$-soft.  Note also that $\mathcal{B}$ is $c$-soft, whence the resolution in line \eqref{can top res} is a Borel resolution as claimed.

Finally, consider any Borel resolution as in line \eqref{top res seq}.  Then we have a commutative diagram
$$
\xymatrix{ 0\ar[r] &  \mathcal{Z} \ar@{=}[d] \ar[r] & \mathcal{S}^0 \ar[r] \ar[d] & \mathcal{S}^1 \ar[r] \ar@{-->}[d] & \mathcal{S}^1 \ar[r] \ar@{-->}[d] & \cdots  \\
0\ar[r]& \mathcal{Z}\ar[r]& \mathcal{B}\ar[r]& \mathcal{I}^1\ar[r]&  \mathcal{I}^2 \ar[r] &\cdots  }
$$
where the first vertical arrow $\mathcal{S}^0\to \mathcal{B}$ is the inclusion from the definition of ``Borel resolution'' and the dashed vertical arrows are filled in using injectivity of each $\mathcal{I}^i$ as a $\gpd$-sheaf \footnote{This is a standard argument from homological algebra: compare for example \cite[Comparison Theorems 2.2.6 and 2.3.7]{Weibel:1995ty}.}.  As both sequences are exact, these arrows constitute a quasi-isomorphism of complexes, so induce an isomorphism between the corresponding groupoid homologies by Proposition \ref{cm id wd}.  This quasi-isomorphism is compatible with the pairings with $T(\gpd)$ (from the formula for the pairing, and as $\mathcal{S}^0\to \mathcal{B}$ is just inclusion), and so the independence of the pairing from the choice of Borel resolution follows.
\end{proof}

\begin{definition}\label{gpd ell}
Let $\gpd$ be a locally compact, Hausdorff, \'{e}tale groupoid with compact base space.  The \emph{groupoid Elliott invariant} consists of the quadruple
$$
(H_{**}(\gpd),[1]_H,T(\gpd),\rho_H)
$$
where: $H_{**}(\gpd)=H_{ev}(\gpd)\oplus H_{od}(\gpd)$ is the $\Z/2$-graded Crainic-Moerdijk homology as in line \eqref{z2 grade}; $T(\gpd)$ is the space of invariant probability measures as in Definition \ref{tg}; $\rho_H:T(\gpd)\to \text{Hom}(H_{**}(\gpd),\R)$ is the pairing between $T(\gpd)$ and $H_{**}(\gpd)$ that agrees with the pairing from Proposition \ref{pair good} on $H_0(\gpd)$ and is zero on the other summands; and $[1]_H\in H_{even}(\gpd)$ is the class of the constant function with value one in $H_0(\gpd)$.
\end{definition}

We will want to compare the groupoid Elliott invariant above to the Elliott invariant of the reduced groupoid $C^*$-algebra.   The next lemma looks at the spaces of invariant measures and traces; it is well-known.  For the statement, let $E:C^*_r(\gpd)\to C_0(\gpd^{(0)})$ denote the canonical conditional expectation as in \cite[Propoisition 4.2.6]{Sims:2017aa}.

\begin{lemma}\label{meas to trace}
Let $\gpd$ be a locally compact, Hausdorff, \'{e}tale groupoid with compact base space.  For $\mu \in T(\gpd)$, define a map 
$$
\tau_\mu:C^*_r(\mathcal{G})\to \C
$$
as the composition of the canonical conditional expectation $E:C^*_r(\mathcal{G})\to C(\mathcal{G}^{(0)})$ and integration against $\mu$.  Then the map
$$
T(\mathcal{G})\to T(C^*_r(\mathcal{G})),\quad \mu \mapsto \tau_\mu
$$
is well-defined, injective, affine, and continuous.  Moreover, if $\mathcal{G}$ is principal then $\tau$ is an affine homeomorphism.
\end{lemma}

\begin{proof}
Direct checks based on the properties of $E$ show that the map $\mu\mapsto \tau_\mu$ is continuous, injective, and affine map and takes image in the state space $C^*_r(\gpd)$; we leave this to the reader.  To check that the image consists of traces, it suffices to check that if $\mu$ is invariant and $a,b\in C_c(\gpd)$ are supported on bisections then $\tau_\mu(ab)=\tau_\mu(ba)$; this follows from invariance, and we again leave the direct computation to the reader.  If $\gpd$ is principal, then $\tau$ is well-known to be surjective: see for example \cite[Lemma 4.3]{Li:2017aa}; in this case it is thus a homeomorphism as $T(\gpd)$ is compact and $T(C^*_r(\gpd))$ is Hausdorff.
\end{proof}

We will typically elide the difference between an invariant measure on $\gpd^{(0)}$ and the trace it defines on $C^*_r(\gpd))$, and write $\tau$ for both.

\begin{definition}\label{hk good}
Let $\gpd$ be a groupoid satisfying the conditions of Definition \ref{gpd ell}.  Then $\gpd$ is \emph{HK-good} if the Elliott invariants
$$
(K_*(C^*_r(\gpd)),[1]_K,T(C^*_r(\gpd)),\rho_K) \quad \text{and}\quad (H_{**}(\gpd),[1]_H,T(\gpd),\rho_H)
$$
are isomorphic in the following sense:
\begin{enumerate}[(i)]
\item there are algebraic isomorphisms $\phi_0:K_0(C^*_r(\gpd))\to H_{ev}(\gpd)$ and  $\phi_1:K_1(C^*_r(\gpd))\to H_{od}(\gpd)$;
\item $\phi_0([1_K])=[1_H]$;
\item the map $\tau:T(\gpd)\to T(C^*_r(\gpd))$ of Lemma \ref{meas to trace} is an affine homeomorphism such that the diagram
$$
\xymatrix{ T(\gpd) \ar[d]^{\tau} \ar[r]^-{\rho_H} & \text{Hom}(H_{ev}(\gpd),\R) \ar[d]^{\text{Hom}(\phi_0,\R)} \\
T(C^*_r(\gpd)) \ar[r]^-{\rho_K} & \text{Hom}(K_{0}(C^*_r(\gpd)),\R) }
$$
commutes.
\end{enumerate}
\end{definition}

We make the following conjecture.

\begin{conjecture}
Any unital classifiable $C^*$-algebra admits an HK-good groupoid model.
\end{conjecture}

The conjecture is known to hold for many classifiable $C^*$-algebras that arise as the $C^*$-algebra of an \emph{ample} groupoid.  For example, \cite[Theorem 4.10]{Matui:2017hq} essentially establishes it for AF groupoids, and \cite[Theorem 4.14]{Matui:2017hq} essentially establishes it for $C^*$-algebras associated to shifts of finite type.  Another class of examples partly generalizing those above comes from \cite[Corollary 4.20]{Bonicke:2021aa}: this shows that if $\gpd$ is an ample, second countable, principal \'{e}tale groupoid with compact base space and dynamic asymptotic dimension at most one, then it is an HK-good model for its groupoid $C^*$-algebra\footnote{See footnote \ref{bad bbgw}.}. As an application of the result of the previous sentence, Reardon \cite{Reardon:PhDthesis} has shown that the groupoids constructed by Putnam \cite{Putnam:2018aa} have dynamic asymptotic dimension one and hence are HK-good.    

It also seems interesting to study HK-good groupoids that give rise to non-simple $C^*$-algebras.  Again, examples are known in the ample case, such as some of the coarse groupoids studied in \cite[Example 5.8]{Bonicke:2021aa}.

The main goal of the rest of this paper is to give some examples of HK-good groupoids that go beyond the ample case.

\section{An example: irrational rotation groupoids} \label{Sec:Irrational}

In this section, we give our first non-trivial computation of a pairing between groupoid homology and invariant measures: the groupoids $\Z\ltimes S^1$ coming from an irrational rotation action on $S^1$, and the (unique) invariant measure coming from Lebesgue measure.  It will follow from general machinery we develop below (see Example \ref{tori ex}) that these examples are HK-good, but we thought it was worth doing one explicit computation from first principles to give intuition for the pairing, and also because of the importance of irrational rotation groupoids in $C^*$-algebra theory and noncommutative geometry.  (The computation can be done even more simply using group hyperhomology as in Appendix \ref{hyp hom sec}: we leave this to the interested reader).

Let $S^1$ be the circle, which we identify with the interval $[0,1]$ modulo the gluing $0\sim 1$.  Let $\theta\in [0,1)$ be a fixed irrational number, and let $\alpha:S^1\to S^1$ be the action defined by the addition of $\theta$ modulo $\Z$.  Let $\gpd:=\Z\ltimes S^1$ denote the corresponding transformation groupoid, which we call an \emph{irrational rotation groupoid}.  As $\theta$ is irrational, there is a unique invariant probability measure on $S^1$ coming from the Lebesgue measure on $[0,1]$, which we denote $\tau$.

In order to compute the pairing of $H_0(\gpd)$ with $\tau$, it will be useful to have a geometrically natural resolution of the sheaf $\mathcal{Z}$ of locally constant $\Z$-valued functions on $S^1$.  Let $\mathcal{S}^0$ be the sheaf whose sections over an open set $U\subseteq S^1$ are those functions $s:U\to \Z$ such that:
\begin{enumerate}[(i)]
\item $s$ is continuous other than at a discrete set of points in $U$;
\item at each point $x\in U$, ${\displaystyle \lim_{y\to x^-}s(y)}$ exists, and ${\displaystyle \lim_{y\to x^+}s(y) }$ exists and equals $s(x)$.
\end{enumerate}
Note that $\mathcal{S}^0$ contains $\mathcal{Z}$ as a subsheaf, so we may define $\mathcal{S}^1:=\mathcal{S}^0/\mathcal{Z}$ to be the quotient sheaf with associated quotient map $\partial:\mathcal{S}^0\to \mathcal{S}^1$.  Concretely, one checks that sections of $\mathcal{S}^1$ over an open subset $U\subseteq S^1$ identify with functions $s:U\to \Z$ such that the set of values where $s$ is non-zero is discrete.  Moreover, one can check that for an open set $U\subseteq S^1$, the  map $\partial_U:\Gamma(U;\mathcal{S}^0)\to\Gamma(U;\mathcal{S}^1)$ induced by the quotient identifies with the map 
\begin{equation}\label{bound form}
s\mapsto \overline{s},\quad \overline{s}(x):=s(x) - \lim_{y\to x^-}s(y).
\end{equation}
Direct checks show that $\mathcal{S}^0$ and $\mathcal{S}^1$ are indeed $c$-soft $\gpd$-sheaves.  In summary then, we have a resolution 
\begin{equation}\label{res}
0\to \mathcal{Z}\to \mathcal{S}^0\stackrel{\partial}{\to} \mathcal{S}^1\to 0
\end{equation}
of $\mathcal{Z}$ by $c$-soft $\gpd$-sheaves.

The groupoid homology of $\gpd=\Z\ltimes S^1$ is then by definition the homology of the double complex
$$
\xymatrix{ \Gamma_c(\gpd^{(0)};\mathcal{S}^0_0) \ar[d]^\partial & \Gamma_c(\gpd^{(1)};\mathcal{S}^0_1) \ar[l]^{\delta} \ar[d]^\partial & \Gamma_c(\gpd^{(2)};\mathcal{S}^0_2) \ar[l]^{\delta}   \ar[d]^\partial & \cdots \ar[l]^-{\delta} \\
\Gamma_c(\gpd^{(0)};\mathcal{S}^1_0)  & \Gamma_c(\gpd^{(1)};\mathcal{S}^1_1) \ar[l]^{\delta} & \Gamma_c(\gpd^{(2)};\mathcal{S}^1_2) \ar[l]^{\delta} & \cdots \ar[l]^-{\delta} },
$$
as in line \eqref{tot com} on page \pageref{tot com}.  

Now, by definition of the homology of a double complex (and noting the conventions on row positions), a $0$-cycle representing a class in $H_0(\gpd)$ consists of a pair 
\begin{equation}\label{cyc}
(a,b) \in \Gamma_c(\gpd^{(0)};\mathcal{S}^0_0)\oplus  \Gamma_c(\gpd^{(1)};\mathcal{S}^1_1)
\end{equation}
such that $\partial(a)=-\delta(b)$.  We also use the resolution in line \eqref{res} to compute the cohomology groups of $S^1$, and then define maps from $H^*(S^1)$ into $H_0(\gpd)$ as follows.  

First, an element of $H^0(S^1)$ is represented by a constant function $a:S^1\to \Z$.  We define a map
\begin{equation}\label{h0 map}
\phi_0:H^0(S^1)\to H_0(\gpd),\quad a\mapsto (a,0);
\end{equation}
it is straightforward to see that this is well-defined.  On the other hand, for a cycle $b\in \Gamma_c(S^1;\mathcal{S}^1)$ representing a class in $H^1(S^1)$, we define an element of $\Gamma_c(\gpd^{(1)};\mathcal{S}^1_1)$ by 
$$
b_1(m,x):=\left\{\begin{array}{ll} b(x) & m=1 \\ 0 &  \text{otherwise}\end{array}\right.  
$$
Note that if $\alpha:\Gamma_c(S^1;\mathcal{S}^1)\to \Gamma_c(S^1;\mathcal{S}^1)$ is the map induced by the action, then 
$$
\delta(b_1)=\alpha(b)-b
$$
(compare line \eqref{first bound} on page \pageref{first bound} for the definition of $\delta$).  As the map $\alpha(b)-b$ represents the trivial class in $H^1(S^1)$ (the action is homotopic to the identity, so induces the identity map on cohomology) there must be $a_b\in \Gamma_c(S^1;\mathcal{S}^0)$ such that 
$$
\alpha(b)-b=-\partial(a_b).
$$
Adjusting $a_b$ by a constant, we assume for definiteness that $a_b(0)=b(0)$; $a_b$ is then uniquely determined by $\delta(b)$.  We may therefore define a map 
\begin{equation}\label{h1 map}
\phi_1:H^1(S^1)\to H_0(\gpd),\quad b\mapsto (a_b,b_1);
\end{equation}
one checks directly this is well-defined.  

\begin{proposition}\label{hom lem}
Let $\gpd=\Z\rtimes S^1$ be an irrational rotation groupoid with associated angle $\theta$, and $\tau$ the unique invariant probability measure coming from Lebesgue measure.  

Then the maps from lines \eqref{h0 map} and \eqref{h1 map} induce an isomorphism 
$$
\phi_0\oplus \phi_1:H^0(S^1)\oplus H^1(S^1)\to H_0(\gpd),
$$
so in particular $H_0(\gpd)\cong \Z\oplus \Z$.  Let moreover $a$ be the generator of $H^0(S^1)$ given by the constant function with value $1$, and let $b$ be the generator of $H^1(S^1)$ given by the function that takes value $1$ at $0\in S^1$, and zero elsewhere.  Then we have the formulas
$$
\tau(\phi_0(a))=1 \quad \text{and}\quad \tau(\phi_1(b))=\theta.
$$
\end{proposition}

The formulas above are consistent with the known $K$-theory of $C^*_r(\gpd)$, and its pairing with the trace: see for example \cite{Pimsner:1980aa} for the original reference, \cite[Theorem VI.5.2]{Davidson:1996jq} for a textbook exposition, and \cite[2.5]{Cuntz:1981ac} and \cite[Proposition 6] {Pimsner:1985aa} for elegant alternative arguments.

\begin{proof}[Proof of Proposition \ref{hom lem}]
Let $(a,b)\in \Gamma_c(\gpd^{(0)};\mathcal{S}^0_0)\oplus  \Gamma_c(\gpd^{(1)};\mathcal{S}^1_1)$ be a cycle for $H_0(\gpd)$.  We first claim that we can adjust $(a,b)$ by boundaries of the form $\delta(c)$ with $c\in \Gamma_c(\gpd^{(2)};\mathcal{S}^1_2)$ so that $b$ is supported on $\{1\}\times S^1$.

Indeed, we may regard elements of $\Gamma_c(\gpd^{(1)};\mathcal{S}^1_1)$ as functions from $\Z\times S^1$ to $\Z$ with finite support.  Let $m\in \Z$ be maximal with the property that the restriction of $b$ to $\{m\}\times S^1$ is non-zero.  We may regard elements of $\Gamma_c(\gpd^{(2)};\mathcal{S}^1_2)$ as functions from $ \Z\times \Z\times S^1$ to $\Z$ with finite support.  If $m\geq 1$, define $c:\Z\times \Z\times S^1\to \Z$ by stipulating $c$ is supported on $\{(-1,m)\}\times S^1$, and satisfies $c(-1,m,x)=b(m,x)$.  Then one checks that subtracting the boundary of $c$ from $b$ `cancels off' the part of $b$ supported on $\{m\}\times S^1$, and does not otherwise change $b$ outside of $(\{-1\}\times S^1)\cup (\{m-1\}\times S^1)$.  Continuing in this way, we may keep adjusting by boundaries so that $b$ is supported on $\{(n,x)\in \Z\times S^1\mid n\leq 0\}$.  Applying a similar procedure to the negative part of the support, we may assume that $b$ is supported on $(\{0\}\times S^1)\cup (\{1\}\times S^1)$.  Finally, subtracting off $\delta(c)$ where $c$ is supported on $\{(0,0)\}\times S^1$ and satisfies $c(0,0,x)=b(0,x)$, we may assume that $b$ is supported on $\{1\}\times S^1$ as claimed.

We next claim that further adjusting $(a,b)$ by classes of the form $\partial(e)$ with $e\in \Gamma_c(\gpd^{(1)};\mathcal{S}^0_1)$, we can assume that $b$ is supported on the singleton $\{(1,0)\}\subseteq  \Z\times S^1$ (at the price of also changing $a$).  Indeed, let $m:=\sum_{x\in S^1}b(1,x)\in \Z$ (this makes sense as $b$ is finitely supported).  Define 
$$
e:=\sum_{\{x\in S^1\mid b(1,x)\neq 0\}} b(1,x)\chi_{[0,x)}.
$$
Then $\partial(e)=m\chi_{\{(1,0)\}}-b$ (compare line \eqref{bound form}), establishing the claim.

Now, if $b(0,1)=m\in \Z$, we have by line \eqref{first bound} that 
$$
\delta(b)(x)=\left\{\begin{array}{ll} m & x=\theta \\ -m & x=0 \\ 0 & \text{otherwise}\end{array}\right..
$$
As $\delta(b)=-\partial(a)$, we necessarily have $a=n+m\chi_{[0,\theta)}$ for some uniquely determined $n\in \Z$.  One checks that with these conditions, a representative of this form is unique.  Using very similar (and simpler) arguments, one can represent any class for $H^0(S^1)$ by a uniquely determined constant function $S^1\to \Z$, and one can represent any class for $H^1(S^1)$ by a uniquely determined function supported on the singleton $\{0\}\subseteq S^1$.  These observations complete the computation of homology.

We now look at the parings.  The formula $\tau(\phi_0(a))=1$ is clear.   To compute $\phi_1(b)$, we note that $\alpha(b)-b$ is the function that is $-1$ at $\theta$, $1$ at $0$, and zero elsewhere.  This is equal to $-\partial(\chi_{[0,\theta)})$; noting that $\partial_{[0,\theta)}(0)=b(0)$, we therefore have that $\phi_1(b)=(\chi_{[0,\theta)},b_1)$.  Hence $\tau(\phi_1(b))=\int_0^1\chi_{[0,\theta)}=\theta$ as claimed.
\end{proof}

\section{Homology of crossed products by the integers} \label{sec:HomZAct}

In this section, we compute the groupoid homology of transformation groupoids $\Z\ltimes X$ associated to an action of the integers on a locally compact Hausdorff space.  We will write $G$ for $\Z$: this is mainly to avoid having to use the notation ``$\Z[\Z]$'' for the integral group ring of the integers, so we can instead write ``$\Z G$''.

For a $G$-module $M$, write $\alpha:M\to M$ for the action of the usual generator $1\in G=\Z$; in particular, we use this notation when $M=\Z G$ is the integral group ring acting on itself.  There is then a free resolution\footnote{Topologically, it comes from considering the circle with one $0$-cell and $1$-cell as a CW complex model for the classifying space $BG$: compare for example \cite[Section I.4]{Brow:1982rt}.} 
\begin{equation}\label{fm res}
0\leftarrow \Z \leftarrow \Z G \stackrel{\delta}{\leftarrow} \Z G \leftarrow 0
\end{equation}
of the trivial $\Z G$-module $\Z$ where the map $\delta$ is given by
\begin{equation}\label{delta def}
\delta:x\mapsto x-\alpha(x).
\end{equation}
Note moreover that if $M$ is any $G$-module, then we may define $\delta:M\to M$ by the same formula as in line \eqref{delta def}.  If we do this, we have canonical identifications 
\begin{equation}\label{free gp hom}
M^G=\text{Kernel}(\delta) \quad \text{and}\quad M_G=\text{Cokernel}(\delta)
\end{equation}
where $M^G$ and $M_G$ are respectively the invariant submodule, and coinvariant quotient of $M$ for the $G$-action (see for example \cite[pages 27 and 34]{Brow:1982rt}).

Here is the homology computation we need (see \cite[Theorem 3.9]{Deaconu:2021aa} and \cite[Lemma 1.3]{Ortega:2020aa} for related results in the case of ample base space).  

\begin{proposition}\label{free gp ses}
Let $X$ be a locally compact Hausdorff space equipped with an action of $\Z$. Then for each $n\in \Z$ there is a canonical short exact sequence
$$
0\to H^{-n}_c(X)_\Z\to H_{n}(\Z\ltimes X)\to H_c^{1-n}(X)^\Z \to 0
$$
(this includes the case $n=1$, where the sequence degenerates to an isomorphism $H_{1}(\Z\ltimes X)\cong H_c^{0}(X)^\Z$, and $n>1$, where it degenerates to $H_n(\Z\ltimes X)=0$).  
\end{proposition}

One could maybe do the computation a little more quickly using the hyperhomology spectral sequence as in \cite[Section 5.7]{Weibel:1995ty}, but we give a more direct approach (which in any case really just amounts to unpacking the spectral sequence argument).

\begin{proof}[Proof of Proposition \ref{free gp ses}]
Let 
\begin{equation}\label{sheaf res}
0\to \mathcal{Z}\to \mathcal{S}^0\to \mathcal{S}^1\to \cdots 
\end{equation}
be a resolution of the sheaf $\mathcal{Z}$ of locally constant $\Z$-valued functions on $X$ by $c$-soft $\gpd$-sheaves.  As in Corollary \ref{cm hh}, we may treat $H_{n}(G\ltimes X)$ as the group hyperhomology $\mathbb{H}_*(G;X)$ of Definitions \ref{hyper1} and \ref{hyper3}.  Doing this, $H_*(G\ltimes X)$ is then the homology of the double complex
$$
\xymatrix{ \Z G\otimes_{\Z G} \Gamma_c(X;\mathcal{S}^{0})  \ar[d] & \Z G\otimes_{\Z G} \Gamma_c(X;\mathcal{S}^{0}) \ar[d] \ar[l]^-{\delta} \\
 \Z G\otimes_{\Z G} \Gamma_c(X;\mathcal{S}^{1})  \ar[d] & \Z G \otimes_{\Z G} \Gamma_c(X;\mathcal{S}^{1}) \ar[d] \ar[l]^-{\delta} \\
 \Z G\otimes_{\Z G} \Gamma_c(X;\mathcal{S}^{2})  \ar[d] & \Z G \otimes_{\Z G} \Gamma_c(X;\mathcal{S}^{2}) \ar[d] \ar[l]^-{\delta} \\
&  }
$$
formed as the tensor product of the complex of line \eqref{fm res} and the complex $\Gamma_c(X;\mathcal{S}^{-\bullet})$ of compactly supported sections.  For notational simplicity, write $M_q$ for the $\Z G$-module $\Gamma_c(X;\mathcal{S}^{-q})$, so the above becomes  
$$
\xymatrix{ M_0 \ar[d] & M_0 \ar[d] \ar[l]^-{\delta} \\
M_{-1} \ar[d] & M_{-1} \ar[d] \ar[l]^-{\delta} \\
M_{-2} \ar[d] & M_{-2} \ar[d] \ar[l]^-{\delta} \\
&  }.
$$
Write $T_\bullet$ for the total complex associated to this, and note that the first column defines a subcomplex of $T_\bullet$ with quotient the second column, and thus we get a short exact sequence of chain complexes
$$
0\to M_\bullet \to T_\bullet \to M_{\bullet-1} \to 0
$$
(the degree shift on the second copy of $M_{\bullet}$ comes about due to the different ways the two columns appear as summands in the double complex: compare \cite[1.2.6]{Weibel:1995ty}).
This in turn gives rise to a long exact sequence of homology groups
$$
\xymatrix{ 
0 \ar[r] & H_1(T_\bullet) \ar[r] & H_0(M_\bullet) \ar[r]^-\delta & H_0(M_\bullet) \ar[r] & H_0(T_\bullet) \ar[dlll] \\
&  H_{-1}(M_\bullet) \ar[r]_-\delta & H_{-1}(M_\bullet) \ar[r] &  H_{-1}(T_\bullet) \ar[r] & \cdots }
$$
with connecting maps $H_i(M_\bullet) \to H_i(M_\bullet)$ given by $\delta$ (this can be computed for example from the proof of \cite[Theorem 1.3.1]{Weibel:1995ty}).  On the other hand, $H_*(M_\bullet)$ computes the compactly supported sheaf homology $H_c^{-*}(X)$, and therefore the long exact sequence above becomes
$$
\xymatrix{ 0\ar[r] &  H_1(G\ltimes X) \ar[r] &  H^0_c(X) \ar[r]^-\delta &  H^0_c(X) \ar[r] &  H_0(G\ltimes X) \ar[dlll] \\
&  H^{1}_c(X)\ar[r]_-\delta & H_c^1(X) \ar[r] &  H_{-1}(G\ltimes X) \ar[r] &  \cdots}.
$$
Splitting this up into short exact sequences along the kernels and cokernels of the maps $\delta$ and applying line \eqref{free gp hom} gives the result.
\end{proof}

\begin{remark}\label{h0 explicit rem}
As it will be useful later, let us give an explicit description of cycles in $H_0(\Z\ltimes X)$ based on the proof of Proposition \ref{free gp ses}.  Write $\gpd$ for $\Z\ltimes X$.

Let $\mathcal{Z}$ be the sheaf of locally constant $\Z$-valued functions on $X$, and consider a resolution by equivariant $c$-soft $\gpd$-sheaves that starts
\begin{equation}\label{exp res}
0\to \mathcal{Z}\stackrel{\partial}{\to} \mathcal{R}\stackrel{\partial}{\to} \mathcal{S}_B^1\stackrel{\partial}{\to} \cdots
\end{equation}
where $\mathcal{R}$ is the sheaf of continuous $\R$-valued functions on $X$, and $\mathcal{S}_B^1$ is the sheaf of Borel $S^1$-valued functions on $X$.  The map from $\mathcal{Z}$ to $\mathcal{R}$ is the canonical inclusion, and the map $\mathcal{R}\to \mathcal{S}_B^1$ is induced by the exponential map $x\mapsto e^{2\pi i x}$; it is not too difficult to see that this is exact, and that $\mathcal{R}$ and $\mathcal{S}_B^1$ are $c$-soft $\gpd$-sheaves; moreover, the image of $\mathcal{R}$ in $\mathcal{S}^1_B$ is the subsheaf $\mathcal{S}^1$ consisting of continuous $S^1$-valued functions on $X$.  

Now, a cycle for $H_0(\Z\ltimes X)$ is a pair $(f,v)$ where $f:X\to \R$ is a continuous compactly supported function, $v:X\to S^1$ is a compactly supported (i.e.\ equal to $1$ outside a compact set) bounded Borel function on $X$, and these functions satisfy $\partial(v)=0$ and $\delta(v)=-\partial(f)$.  The fact that $\partial(v)=0$ implies that $v$ is continuous (as it is locally in the image of the exponential map from $\mathcal{R}$).  On the other hand, the formula $\delta(v)=-\partial(f)$ says explicitly that $\exp(2\pi if)=\alpha(v)v^*$.
\end{remark}

\section{The pairing and Chern character for actions of the integers}\label{Sec:int pair}

Let $\alpha$ denote a homeomorphism of a locally compact space $X$, and also the associated actions of $\Z$ on $X$ and on $C_0(X)$; we fix this notation throughout the remainder of this section.  Our aim in this section is to construct a Chern character map 
\begin{equation}\label{chernc}
\text{ch}_\alpha:K_*(C_0(X)\rtimes_\alpha\Z)\to H_{**}(\Z\ltimes X;\Q)
\end{equation}
that becomes an isomorphism on tensoring the left hand side by $\Q$, is compatible with the pairings with invariant probability measures, and that moreover can be made sense of as an integral isomorphism in some cases: for example, this happens if $X$ has only low-dimensional homology as in Corollary \ref{int cor} below.  

We should note that the existence of a Chern character rational isomorphism follows from our earlier results in \cite{Deeley:2024aa}.  However, compatibility with the pairings and the fact that it is sometimes compatible with an integral isomorphism is not so obvious for the Chern character in \cite{Deeley:2024aa} (which is essentially due to Raven \cite{Raven:2004aa}), so we give a different approach here.

We first recall the approach to the Pimsner-Voiculescu sequence based on mapping tori: the original reference for this is \cite[pages 48-49]{Connes:1981ab}, or see for example \cite[Sections 10.3-10.4]{Blackadar:1998yq} for a textbook exposition; the original reference for the Pimsner-Voiculescu sequence is \cite{Pimsner:1980aa}.

\begin{definition}\label{map torus}
The \emph{mapping torus} of $\alpha$, denoted $M_\alpha$, is the maximal ideal space\footnote{We could also directly define $M_\alpha$ to be the quotient of $X\times [0,1]$ by the relation $(x,0)\sim (\alpha^{-1}(x),1)$, but it is more useful for us to have a description of the continuous functions on the space.} of the $C^*$-algebra
$$
C_0(M_\alpha):=\{f\in C_b(\R,C_0(X))\mid f(t+1)=\alpha(f(t)) \text{ for all } t\in \R\}.
$$
\end{definition}

Note that there is a short exact sequence 
\begin{equation}\label{mc ses}
0\to C_0(SX)\to C_0(M_\alpha) \to C_0(X)\to 0
\end{equation}
where $SX:=(0,1)\times X$ is the usual $C^*$-algebraic suspension of $X$, the left hand map is inclusion as $\Z$-equivariant functions from $\R$ to $C_0(X)$ that vanish at the integers, and the right hand map is evaluation at zero.  This gives rise to a six-term exact sequence
\begin{equation}\label{mc seq}
\xymatrix{ K^0(SX)\ar[r] & K^0(M_\alpha) \ar[r] & K^0(X) \ar[d] \\ K^1(X) \ar[u] & K^1(M_\alpha) \ar[l] & K^1(SX) \ar[l] }.
\end{equation}
Moreover, if we use the periodicity isomorphisms to make identifications $K^i(SX)\cong K^{i-1}(X)$, then the boundary maps in this sequence identify with $\text{id}-\alpha_*$ (see for example \cite[Proposition 10.4.1]{Blackadar:1998yq}).

The following result records the key $C^*$-algebraic facts we will need; this seems likely to be known to experts, but we provide proofs where we could not find appropriate references.   For $i\in \{0,1\}$ and an action $\beta$ of $\R$ on a $C^*$-algebra $B$, we let $\phi_{\beta}^i:K_i(B)\to K_{i+1}(B\rtimes_{\beta} \R)$ denote the Connes-Thom isomorphism as defined in \cite[Section II]{Connes:1981ab}. 

\begin{proposition}\label{mc basics}
Let $X$ be a locally compact Hausdorff space, and let $\alpha:X\to X$ be a homeomorphism inducing an action of $\Z$ on $X$.  Let $\beta$ denote the action of $\R$ on $C_0(X)\rtimes_\alpha\Z$\footnote{We can equivalently define $\beta$ to be the composition of the dual action of $\widehat{\Z}=\R/\Z$ on $C_0(X)\rtimes_\alpha\Z$, and the quotient map $\R\to \R/\Z$.} determined by
$$
\beta_t(f u^n)=e^{2\pi i tn} f u^n,
$$
where $f\in C_0(X)$, $u\in M(C_0(X)\rtimes \Z)$ is the canonical unitary multiplier, $n\in \Z$, and $t\in \R$.  Then the following hold:
\begin{enumerate}[(i)]
\item \label{mc me}There is a distinguished class of isomorphisms 
$$
m_\alpha: (C_0(X)\rtimes_\alpha \Z)\rtimes_{\beta}\R\to C_0(M_\alpha)\otimes \mathcal{K}(\ell^2(\Z)),
$$
all of which induce the same map on $K$-theory.
\item \label{ct me} the isomorphism 
$$
\psi_\alpha^i:=(m_\alpha)_*\circ \phi_\beta^i:K_i(C_0(X)\rtimes_\alpha \Z)\to K^{i+1}(M_\alpha)
$$
fits into a commutative diagram
$$
\xymatrix{ K^{1}(SX)  \ar[r] & K^{1}(M_\alpha) \\ 
K^0(X)\ar[r] \ar[u] & K_0(C_0(X)\rtimes_\alpha\Z) \ar[u]^-{\psi^0_\alpha}}
$$
where the left hand vertical map is Bott periodicity, and the horizontal maps are induced from the canonical inclusions $C_0(SX)\to C_0(M_\alpha)$ (see line \eqref{mc seq}) and $C_0(X)\to C_0(X)\rtimes_\alpha \Z$.
\item \label{pv ct} The six term exact sequence in line \eqref{mc seq}, the commutative diagram from part \eqref{ct me}, and the Pimsner-Voiculescu exact sequence as constructed in \cite[Sections 10.3-4]{Blackadar:1998yq} induce a commutative diagram of short exact sequences
$$
\xymatrix{ 0 \ar[r] & K^1(SX)_\Z  \ar[r] & K^1(M_\alpha)  \ar[r] & K^1(X)^\Z \ar@{=}[d] \ar[r] & 0 \\
0 \ar[r] & K^0(X)_\Z \ar[r] \ar[u] & K_0(C_0(X)\rtimes_\alpha \Z) \ar[u]^-{\psi_\alpha^0} \ar[r] & K^1(X)^\Z \ar[r] & 0 }
$$
with all vertical maps isomorphisms.
\end{enumerate}
\end{proposition}

\begin{proof}
For part \eqref{mc me}, let $\pi_X:C_0(X)\to \mathcal{B}(H_X)$ be a faithful nondegenerate representation.  Let
$$
\pi: C_0(X)\rtimes_\alpha \Z\rtimes_\beta \R\to \mathcal{B}(H_X\otimes \ell^2(\Z)\otimes L^2(\R))
$$
be the regular representation, i.e.\ if we identify $H_X\otimes \ell^2(\Z)\otimes L^2(\R)=L^2(\Z\times \R,H_X)$, then $\pi$ is induced by the covariant triple $(\widetilde{\pi_X},u,v)$ where 
$$
((\widetilde{\pi_X}(f))\xi)(n,t):=\pi_X(\alpha_{-n}(f))\xi(n,t)
$$
and 
$$
(u_m\xi)(n,t):=e^{2\pi i mt} \xi(n-m,t),\quad (v_s\xi)(n,t):=\xi(n,t-s).
$$

Let $F:L^2(\R) \to L^2(\R)$ denote the unitary isomorphism induced by the Fourier transform, and let $\lambda$ denote the translation action of $\Z$ on $\R$.  Then one computes that $1\otimes 1 \otimes F$ conjugates (the integrated form of) $(\widetilde{\pi_X},u,v)$ to the regular representation of $C_0(\R\times X)\rtimes_{\lambda\times \alpha}\Z$ on $H_X\otimes \ell^2(\Z)\otimes L^2(\R)=\ell^2(\Z,L^2(\R)\otimes H_X)$, and thus conjugates the image of $C_0(X)\rtimes_\alpha \Z\rtimes_\beta \R$ to $C_0(\R\times X)\rtimes_{\lambda\otimes \alpha}\Z$.  

Let $\sigma$ denote the automorphism of $\mathcal{K}(\ell^2(\Z))$ induced by conjugation by the bilateral shift.  Then more direct computations show that $C_0(\R\times X)\rtimes_{\lambda\otimes \alpha}\Z$ canonically identifies with  
$$
C_b(\R\times X,\mathcal{K}(\ell^2(\Z)))^\Z=\{f\in C_b(\R\times X,\mathcal{K}(\ell^2(\Z)))\mid f(t+1,\alpha(x))=\sigma(f(t,x))\}.
$$
Choose now any path $(\sigma_t)_{t\in [0,1]}$ of automorphisms of $\mathcal{K}(\ell^2(\Z))$ connecting $\sigma$ to the identity that is induced by a norm continuous part of unitaries connecting the bilateral shift to the identity.  Use this path to `untwist' $C_0(\R\times X)\rtimes_{\lambda\otimes \alpha}\Z$, showing that it is isomorphic to 
\begin{align*}
\{f\in C_b(\R\times X,\mathcal{K}(\ell^2(\Z)))\mid f(t+1,\alpha(x)) =f(t,x)\} & = C_0((\R\times X)/\Z,\mathcal{K}(\ell^2(\Z))) \\ & 
= C_0(M_\alpha)\otimes \mathcal{K}(\ell^2(\Z)).
\end{align*}
This isomorphism of $C_0(\R\times X)\rtimes_{\lambda\otimes \alpha}\Z$ with $C_0(M_\alpha)\otimes \mathcal{K}(\ell^2(\Z))$ depends on the choice of the path $(\sigma_t)$, but any two such paths are homotopic, so the map induced on $K$-theory is well-defined.

For part \eqref{ct me}, note that the canonical inclusion $C_0(X)\to C_0(X)\rtimes_\alpha \Z$ is $\R$-equivariant for the trivial $\R$ action $\text{tv}$ on $C_0(X)$, and the action $\beta$ on $C_0(X)\rtimes_\alpha \Z$.  Hence by naturality of the Connes-Thom isomorphism we get a commutative diagram
$$
\xymatrix{ K_1(C_0(X)\rtimes_{\text{tv}}\R)  \ar[r] & K^{1}(C_0(X)\rtimes_\alpha\Z\rtimes_\beta \R) \\ 
K^0(X)\ar[r] \ar[u]^-{\phi^0_{\text{tv}}} & K_0(C_0(X)\rtimes_\alpha\Z) \ar[u]^-{\phi^0_\beta}}.
$$
Up to applying a Fourier transform, $K_1(C_0(X)\rtimes_{\text{tr}}\R)$ identifies with $C_0(\R\times X)$ and the Connes-Thom map $\phi^0_{\text{tv}}$ identifies with the Bott periodicity isomorphism, so using also the identifications established in the proof of part \eqref{mc me} we get a commutative diagram
$$
\xymatrix{ K^1((0,1)\times X) \ar[r] & K^1(\R\times X)  \ar[r] & K^{1}(C_b(\R\times X,\mathcal{K}(\ell^2(\Z)))^\Z) \\ 
K^0(X)\ar@{=}[r] \ar[u] & K^0(X)\ar[r] \ar[u]^-{\phi^0_{\text{tv}}} & K_0(C_0(X)\rtimes \Z) \ar[u]^-{\phi^0_\beta}}.
$$
where the top-left horizontal map is induced by inclusion (and is an isomorphism) and the left hand vertical map is Bott periodicity.  Up to composing with the isomorphism $m_\alpha$, the outer rectangle is the diagram in the statement.

Part \eqref{pv ct} follows as the Pimsner-Voiculescu exact sequence in \cite[Section 10.3-4]{Blackadar:1998yq} is exactly defined by forcing the given diagram to commute.
\end{proof}

The next result we need is a `dual version' of a result of Connes: compare \cite[Section II, Theorem 3]{Connes:1981ab}.  It seems to be well-known, but we could not find the precise statement we wanted in the literature so sketch a proof.

\begin{proposition}\label{trace agree}
Let $X$ be a compact Hausdorff space equipped an action $\alpha$ of $\Z$ by homeomorphisms, and let  $\tau:C(X)\to \C$ be an invariant probability measure.  

Any class in $K^1(M_\alpha)$ can be represented by a continuous function $u:[0,1]\to \mathcal{U}(M_n(C(X)))$ that is continuously differentiable on $(0,1)$, and with the property that $u(1)=\alpha(u(0))$.   Having chosen such a representative for each class, the formula 
\begin{equation}\label{lt}
\Lambda_\tau:K^1(M^\alpha)\to \R,\quad [u]\mapsto \frac{1}{2\pi i} \int_0^1 \tau (u'(t)u(t)^*) dt,
\end{equation}
(where ``$u'$'' denotes the derivative of $u$ in the $t$-variable) gives a well-defined homomorphism that does not depend on the choice of representatives.

Moreover, with notation as in Proposition \ref{mc basics}, the diagram 
\begin{equation}\label{tlt}
\xymatrix{ K^1(M_\alpha) \ar[r]^-{\Lambda_\tau} & \R \ar@{=}[d] \\ K_0(C_0(X)\rtimes_\alpha\Z) \ar[u]^-{\psi_\alpha^0} \ar[r]^-\tau & \R}
\end{equation}
commutes.
\end{proposition}

\begin{proof}
For a function $u:[0,1]\to \mathcal{U}(M_n(C(X)))$, the real number $\Lambda_\tau[u]$ is the de la Harpe-Skandalis determinant of the path as in \cite[Section 1]{Harpe:1984aa} (the notation there would be $\widetilde{\Delta}_\tau(u)$).  The fact that $\Lambda_\tau$ is a well-defined group homomorphism (including that any $K_1$-class can be represented by a function that is differentiable in the $t$-variable) then follows directly from \cite[Lemme 1]{Harpe:1984aa}.

We need some more notation.  Let $B$ be a general unital $C^*$-algebra equipped with an action $\beta$ of $\R$, and let $\tau$ denote an $\R$-invariant tracial state on $B$.   Let $\widehat{\tau}$ denote the `dual' (unbounded) trace on $B\rtimes_\beta \R$ determined on $C_c(\R,B)$ by $\widehat{\tau}(f)=\tau(f(0))$.  Let $\delta_\beta$ denote the (densely defined) derivation on $B\rtimes_\beta \R$ determined as on \cite[page 32]{Connes:1981ab} by the formula ${\displaystyle \delta_\beta(a)=\lim_{t\to 0} t^{-1}(\beta_t(a)-a)}$.  As in \cite[Section II, Theorem 3]{Connes:1981ab} there is then a well-defined homomorphism 
$$
\Theta_{\tau}:K_1(B\rtimes_\beta \R)\to \R,\quad [u] \mapsto \frac{1}{2\pi i}\widehat{\tau}(\delta(u)u^*)
$$
(this includes the statement that any class in $K_1(B\rtimes_\beta \R)$ can be represented by a unitary such that the right hand side makes sense).

Now, let $p\in M_n(C(X)\rtimes_\alpha \Z)$ represent a class $[p]$ in $K_0(C(X)\rtimes_\alpha \Z)$.  According to \cite[Section II, Proposition 4]{Connes:1981ab},  there is an exterior equivalent action $\beta'$ of $\R$ on $C(X)\rtimes_\alpha \Z$ that fixes $p$.  If $(u_t)$ is the unitary one-cocycle implementing the exterior equivalence, then there is an isomorphism 
$$
\theta:C(X)\rtimes_\alpha \Z\rtimes_\beta \R\to C(X)\rtimes_\alpha \Z\rtimes_{\beta'} \R
$$
determined on $C_c(\R,C(X)\rtimes_\alpha \Z)$ by $(\theta b)(t):=u_tb(t)$.  Consider the diagram below
$$
\xymatrix{  K_0(C(X)\rtimes_\alpha\Z) \ar[r]^-{\phi_\beta^0} \ar[d]^-\tau &  K_1(C(X)\rtimes_\alpha \Z\rtimes _{\beta'} \R) \ar[d]^-{\Theta_\tau}  \ar[r]^-{(\theta^{-1})_*}  & K_1(C(X)\rtimes_\alpha \Z\rtimes _\beta \R) \ar[r]^-{(m_\alpha)_*} \ar[d]^-{\Theta_\tau} &   K^1(M_\alpha) \ar[d]^-{\Lambda_\tau}  \\
\R \ar@{=}[r] & \R \ar@{=}[r] & \R \ar@{=}[r] & \R }.
$$
We have that $\theta_*\circ \phi^0_{\beta'}=\phi_0^\beta$ (see \cite[Section 1, Proposition 3]{Connes:1981ab}), so it suffices to show that this diagram commutes on the class $[p]$.  We do this one square at a time.  

For the leftmost square, we can represent $\phi_\beta^0[p]$ as the Fourier transform of the function $u(t)=e^{2\pi i \chi(t)p}$ where $\chi:\R\to [0,1]$ is a smooth nondecreasing function such that $\chi(t)=0$ for $t\leq 0$ and $\chi(t)=1$ for $t\geq 1$ (compare \cite[pages 37-38]{Connes:1981ab}).  As the Fourier transform interchanges $\delta$ and differentiation, and as it interchanges evaluation at $0$ with integration, we thus have that 
$$
\Theta_\tau[u]=\frac{1}{2\pi i} \int_\R \tau(2\pi i\chi'(t)p uu^*)dt=\tau(p)\int_\R\chi'(t)dt=\tau(p). 
$$ 
Commutativity of the central square is again a direct computation using that $\widehat{\tau}$ is a trace and $\theta$ is induced by a unitary multiplier.  Commutativity of the rightmost square is another direct computation, based on the description of $m_\alpha$ in terms of the Fourier transform, and the fact that the Fourier transform interchanges evaluation at $0$ with integration, and that it interchanges the derivation $\delta$ with differentiation; the integral ends up being only over $[0,1]$ rather than $\R$ as we are implicitly also taking the standard trace on $\mathcal{K}(\ell^2(\Z))$, which has the effect of `adding up' the $\Z$ translates of this.
\end{proof}

We now turn to homology.  Let $X$ be a locally compact Hausdorff space equipped with an action $\alpha$ of $\Z$.  Recall from Remark \ref{h0 explicit rem} that cycles for $H_0(\Z\ltimes X)$ can be thought of as pairs $(f,v)$ where $f:X\to \R$ is continuous and compactly supported, and $v:X\to S^1$ is continuous and compactly supported (i.e.\ equal to one outside a compact set), and satisfies $\alpha(v)v^*=e^{2\pi i f}$.  Moreover, cycles for $H^1_c(M_\alpha)$ (respectively, $H^1_c(SX)$) can be thought of as continuous and compactly supported functions $u:M_\alpha\to S^1$ (respectively, $u:SX\to S^1$), and cycles for $H^0_c(X)$ can be thought of as locally constant and compactly supported functions $f:X\to \Z$.  

\begin{lemma}\label{exp low}
Let $X$ be a locally compact Hausdorff space equipped with an action $\alpha$ of $\Z$.  There is an isomorphism $H_0(\Z\ltimes X)\to H^1_c(M_\alpha)$ explicitly realized on the level of cycles by sending $(f,v)$ to the function $u:M_\alpha\to S^1$ defined by $u(t,x)=e^{2\pi i tf(x)}v(x)$.  This isomorphism fits moreover into a commutative diagram 
$$
\xymatrix{0 \ar[r] & H^1_c(SX)_\Z  \ar[r]   & H^1_c(M_\alpha) \ar[r]  & H^1_c(X)^\Z \ar@{=}[d] \ar[r] & 0 \\
0 \ar[r] & H^{0}_c(X)_\Z \ar@<1ex>[u]  \ar[r] & H_{0}(\Z\ltimes X) \ar@<1ex>[u]  \ar[r] & H^1_c(X)^\Z \ar[r] & 0 }
$$
where the bottom row is as in Proposition \ref{free gp ses}.  The same holds if we use rational coefficients.
\end{lemma}

\begin{proof}
The left vertical map is induced by the map sending the class $[f]\in H^0_c(X)$ of a compactly supported locally constant function $f:X\to \Z$ to the class of the function $u(t,x)=e^{2\pi i tf(x)}$ in $H^1_c(SX)$.  This is $\Z$-equivariant (for the trivial $\Z$-action in the suspension variable), and induces an isomorphism $H^0_c(X)\to H^1_c(SX)$, whence it also induces an isomorphism after taking coinvariants.  Clearly the diagram commutes, whence the central vertical map is an isomorphism by the five lemma.  The same argument works with rational coefficients (or use the universal coefficient theorem).
\end{proof}

Analogously to Lemma \ref{trace agree} we may define a map $\Lambda_\tau^H:H^1_c(M_\alpha)\to \R$ using the same formula as for $\Lambda_\tau:K^1(M_\alpha)\to \R$ (see line \eqref{lt} above - this involves first approximating a cycle $u:M_\alpha\to S^1$ by a cycle that is continuously differentiable in the $t$ variable).  We then have another compatibility lemma about pairings.

\begin{lemma}\label{trace agree 2}
With notation as in Lemma \ref{exp low}, the diagram 
$$
\xymatrix{ H^1_c(M_\alpha) \ar[r]^-{\Lambda^H_\tau} & \R \ar@{=}[d] \\ H_0(\Z\ltimes X) \ar@<1ex>[u]  \ar[r]^-\tau & \R}
$$
commutes.  The same holds with rational coefficients.
\end{lemma}

\begin{proof}
Let $(f,v)$ be a cycle for $H_0(\Z\ltimes X)$, so $\tau[f,v]=\int_X fd\tau$ by definition.  On the other hand, the isomorphism $H_0(\Z\ltimes X)\to H^1_c(M_\alpha)$ from Lemma \ref{exp low} sends $(f,v)$ to the function $u:M_\alpha\to S^1$ defined by $u(t,x)=e^{2\pi i tf(x)}v(x)$.  Computing,
$$
\Lambda_\tau[u]=\frac{1}{2\pi i } \int_0^1 \tau(u'(t)u(t)^*)dt = \frac{1}{2\pi i}\int_0^1 \tau(2\pi i fu(t)u(t)^*) dt=\tau(f),
$$
so we are done.
\end{proof}

We will also need to compare the mapping cone sequence in cohomology to the sequence in groupoid homology.  The following result is in some sense an (easier) analogue of $\Z$-equivariant Bott periodicity for groupoid homology.  It is also essentially a special case of \cite[Theorem B]{Proietti:2023aa}, but we provide a direct proof in the case we need for the sake of keeping the paper self-contained.

\begin{proposition}\label{gh bott}
Let $X$ be a locally compact Hausdorff space equipped with an action $\alpha$ of $\Z$.   Let $\R$ be equipped with the usual translation of $\Z$, and equip $\R\times X$ with the diagonal action.  Then there is a canonical isomorphism
$$
H_i(\Z\ltimes X)\cong H_{i-1}(\Z\ltimes (\R\times X))
$$
for each $i\in \Z$.
\end{proposition}

We need an ancillary lemma, which is standard algebraic topology: we provide a proof as we could not find an appropriate reference.  For the statement, recall that a sheaf $\mathcal{S}$ on a locally compact space $X$ is \emph{$c$-fine} if the homomorphism sheaf $\mathcal{H}om(\mathcal{S},\mathcal{S})$ is $c$-soft (see for example \cite[Section II.9]{Bredon:1997aa}).

\begin{lemma}\label{int fib lem}
Let $X$ be a locally compact Hausdorff space equipped with an action $\alpha$ of $\Z$.   Let $\R$ be equipped with the translation action of $\Z$, and equip $\R\times X$ with the diagonal action. Then there is an `integration along the fiber' isomorphism $\int:H_c^{*}(\R\times X)\to H_c^{*-1}(X)$ that is moreover equivariant for the induced $\Z$-actions.
\end{lemma}

\begin{proof}
We define essentially the same resolution $0\to \mathcal{Z}\to \mathcal{S}^0\to \mathcal{S}^1\to 0$ of the sheaf $\mathcal{Z}$ of locally constant $\Z$-valued functions on $\R$ that we used to compute the pairing for an irrational rotation action in Section \ref{Sec:Irrational}.  Precisely: $\mathcal{S}^0$ is the sheaf whose sections on an open set $U\subseteq \R$ are the functions $s:U\to \Z$ such that for all $x\in U$, ${\displaystyle \lim_{y\to x^-}s(y)}$ exists, ${\displaystyle\lim_{y\to x^+}s(y)}$ exists and equals $s(x)$, and $s$ has a discrete set of discontinuities in $U$; $\mathcal{S}^1$ is the sheaf whose sections on an open set $U$ are functions $s:U\to \Z$ with discrete support; the map $\mathcal{Z}\to \mathcal{S}^0$ is the canonical inclusion; and the map $\mathcal{S}^0\to \mathcal{S}^1$ is defined over an open set $U$ by sending a section $s$ to the section $t$ defined by ${\displaystyle t(x):=s(x)-\lim_{y\to x^-}s(y)}$.  One checks directly that each $\mathcal{S}^i$ is a $c$-fine $\Z$-equivariant sheaf, and that $\mathcal{S}^1$ identifies with $\mathcal{S}^0/\mathcal{Z}$.   

Define now a map $\int:\Gamma_c(\R;\mathcal{S}^1)\to \Z$ by sending a section $s$ to $\sum_{x\in \R} s(x)$ (this makes sense as $s$ is finitely supported).  The map $\int$ extends to a map of complexes $\int:\Gamma_c(\R;\mathcal{S}^\bullet)\to \mathcal{C}^\bullet$ where $\mathcal{C}^\bullet$ is $\Z$ in degree one, and zero in all other degrees.  As the homology groups of the complex $\Gamma_c(\R;\mathcal{S}^\bullet)$ (i.e.\ the cohomology $H_c^i(\R))$ are $\Z$ in degree one, and zero in all other degrees, this map is a quasi-isomorphism.

Let now $\mathcal{T}^\bullet$ be any resolution of the sheaf $\mathcal{Z}$ of locally constant $\Z$-valued functions on $X$ by $c$-fine $\Z$-sheaves, and define $(\mathcal{S}\otimes \mathcal{T})^i$ to be the sheaf $\mathcal{S}^0\otimes \mathcal{T}^i\oplus \mathcal{S}^1\otimes \mathcal{T}^{i-1}$.  As each $\mathcal{S}^i$ is torsion free, the tensor product sheaf $(\mathcal{S}\otimes \mathcal{T})^\bullet$ is a resolution of the sheaf of locally constant $\Z$-valued functions on $\R\times X$.  Moreover, it consists of $\Z$-equivariant sheaves, and each constituent sheaf in the resolution is $c$-fine by \cite[Corollary II.9.18]{Bredon:1997aa}.  

Now, as $\mathcal{S}^\bullet$ and $\mathcal{T}^\bullet$ are $c$-fine, \cite[Proposition II.5.1]{Bredon:1997aa} identifies $\Gamma_c(\R\times X;(\mathcal{S}\otimes \mathcal{T})^\bullet)$ with the tensor product complex $\Gamma_c(\R;\mathcal{S}^\bullet)\otimes \Gamma_c(X;\mathcal{T}^\bullet)$.  As $\int$ induces a quasi-isomorphism $\Gamma_c(\R;\mathcal{S}^\bullet)\to \mathcal{C}^\bullet$, $\int\otimes \text{id}$ induces a quasi-isomorphism 
\begin{equation}\label{last qi}
\smallint\otimes \text{id}:\Gamma_c(\R\times X;(\mathcal{S}\otimes \mathcal{T})^\bullet)\to \Gamma_c(X;\mathcal{T}^{\bullet-1})
\end{equation}
by the K\"{unneth} formula (the purely algebraic version for chain complexes - see for example \cite[Theorem 3.6.3]{Weibel:1995ty}).  As the left hand side in line \eqref{last qi} computes $H_c^{*}(\R\times X)$ and the right hand side computes $H_c^{*-1}(X)$, we are done.
\end{proof}

\begin{proof}[Proof of Proposition \ref{gh bott}]
As the map $\int:\Gamma_c(\R\times X;(\mathcal{S}\otimes \mathcal{T})^\bullet)\to \Gamma_c(X;\mathcal{T}^{\bullet-1})$ is $\Z$-equivariant, it induces a map $\int:\mathbb{H}_*(\Z;\R\times X)\to \mathbb{H}_*(\Z;X)$ on hyperhomology (see Definitions \ref{hyper1} and \ref{hyper3} for notation).  On the other hand, the map it induces on the $E_2$-page of the hyperhomology spectral sequences (see for example \cite[Proposition 5.7.6]{Weibel:1995ty}) from $E^2_{pq}(\R\times X)=H_p(\Z;H^q_c(\R\times X))$ to $E^2_{pq}(X):=H_{p}(\Z;H^q_c(X))$ is an isomorphism (with appropriate degree shift).  As the hyperhomology spectral sequence converges, $\int$ induces an isomorphism on hyperhomology $\int:\mathbb{H}_*(\Z;\R\times X)\to \mathbb{H}_*(\Z;X)$, and so we are done by the identification of hyperhomology and groupoid homology of Corollary \ref{cm hh}.
\end{proof}

\begin{lemma}\label{mc seq gpd}
Let $X$ be a locally compact Hausdorff space equipped with an action $\alpha$ of $\Z$.   For each $n$, the short exact sequences of Proposition \ref{free gp ses} form the bottom row in a commutative diagram of short exact sequences 
$$
\xymatrix{0 \ar[r] & H^n_c(SX)_\Z  \ar[r] \ar[d]  & H^n_c(M_\alpha) \ar[r] \ar[d] & H^n_c(X)^\Z \ar@{=}[d] \ar[r] & 0 \\
0 \ar[r] & H^{n-1}_c(X)_\Z   \ar[r] & H_{1-n}(\Z\ltimes X)   \ar[r] & H^n_c(X)^\Z \ar[r] & 0 }
$$
where the left hand vertical map is induced by integration along the fiber $\R$, the middle vertical map is induced again by the isomorphism from Lemma \ref{gh bott} plus the Morita equivalence $\Z\ltimes (\R\times X)\sim M_\alpha$, and all the vertical maps are isomorphisms.  The same holds if we use rational rather than integer coefficients.
\end{lemma}

\begin{proof}
Equip $\R$ with the translation action of $\Z$, and let $\int:H^n_c(\R\times X)\to H^{n-1}_c( X)$ be the integration along the fiber isomorphism if Lemma \ref{int fib lem}.  Moreover, Lemma \ref{gh bott} implies that external product with this generator also induces isomorphisms on groupoid homology $H_{n-1}(\Z\ltimes (\R\times X))\cong H_n(\Z\ltimes X)$.  We thus get a commutative diagram
$$
\xymatrix{ 0 \ar[r] & H^{n+1}_c(\R\times X)_\Z  \ar[r]  \ar[d] & H_{-n-1}(\Z \ltimes (\R\times X)) \ar[r] \ar[d] & H^{n+2}_c(\R \times X)^\Z  \ar[r] \ar[d] & 0  \\
0 \ar[r] & H^n_c(X)_\Z  \ar[r] & H_{-n}(\Z\ltimes X)  \ar[r] & H^{n+1}_c(X)^\Z \ar[r]  & 0 }.
$$
Morita invariance (see \cite[Corollary 4.6]{Crainic:2000aa}\footnote{This reference requires a finite-dimensionality assumption on $X$, and there are some technical difficulties with derived categories to remove it.  A simpler argument in the infinite-dimensional case can be deduced using hyperhomology $\mathbb{H}_*(\Z;\R\times X)$ and a Cartan-Leray spectral sequence argument as in \cite[page 173, particularly line (7.8)]{Brow:1982rt}}.) gives that the middle entry of the top row is isomorphic to $H^{n+1}_c(M_\alpha)$ and the five lemma implies that central vertical map is an isomorphism as claimed.  The same argument works with rational coefficients (or use the universal coefficient theorem).
\end{proof}

\begin{theorem}\label{main seq}
Let $X$ be a locally compact Hausdorff space, and let $\alpha:X\to X$ be a homeomorphism inducing an action, also denoted $\alpha$, of $\Z$ on $X$. 

There is a natural Chern character 
$$
\text{ch}_\alpha:K_*(C_0(X)\rtimes_\alpha \Z)\to H_{**}(\Z\ltimes X;\Q).
$$
that becomes an isomorphism on tensoring the left hand side by $\Q$, and that fits into a commutative diagram of short exact sequences
\begin{equation}\label{ch com}
\xymatrix{0 \ar[r] & K^*(X)_\Z  \ar[r] \ar[d] & K_*(C_0(X)\rtimes_\alpha \Z) \ar[r] \ar[d]^-{\text{ch}_\alpha}  \ar[r] & K^{*+1}(X)^\Z \ar[r]  \ar[d] & 0  \\
 0 \ar[r] & H^{**}_c(X;\Q)_\Z  \ar[r] & H_{**}(\Z\ltimes X;\Q) \ar[r] & H^{**+1}_c(X;\Q)^\Z \ar[r] & 0  }
\end{equation}
with the left and right vertical maps induced by the classical Chern character.  Moreover, one can replace $\text{ch}_\alpha$ by an integral isomorphism if there is an integral Chern isomorphism for $M_\alpha$ in the sense of Definition \ref{ch int}.

Finally, if $X$ is compact, then for any invariant probability measure $\tau$ on $X$ the diagram 
\begin{equation}\label{last tau}
\xymatrix{  K_0(C(X)\rtimes_\alpha \Z) \ar[d]^{\text{ch}_\alpha} \ar[r]^-\tau & \R \ar@{=}[d]\\ H_{ev}(\Z\ltimes X;\Q) \ar[r]  \ar[r]^-\tau  & \R }
\end{equation}
commutes\footnote{This can be usefully compared to material in \cite[Section 3]{Guo:2024ab}.}.
\end{theorem}

\begin{proof}
The Chern character $\text{ch}_\alpha$ is defined by composing the maps
$$
K_*(C(X)\rtimes_\alpha\Z)\to K^{*+1}(M_\alpha) \stackrel{\text{ch}}{\to} H^{**+1}(M_\alpha;\Q)\to H_{**}(\Z\ltimes X;\Q)
$$
where the first arrow is from Proposition \ref{mc seq}, the second is the classical Chern character, and the third is from Lemma \ref{exp low} on the summand $H^1(M_\alpha)$, and from Lemma \ref{mc seq gpd} on the summands in all other degrees.  It is a rational isomorphism as the central map is a rational isomorphism, and as the left and right maps are isomorphisms.

For commutativity of the diagram in line \eqref{ch com}, we consider the large diagram
$$
\xymatrix{0 \ar[r] & K^*(X)_\Z  \ar[r] \ar[d] & K_*(C_0(X)\rtimes_\alpha \Z) \ar[r] \ar[d]  \ar[r] & K^{*+1}(X)^\Z \ar[r]  \ar[d] & 0  \\
0 \ar[r] & K^*(X)_\Z  \ar[r] \ar[d] & K_{*+1}(M_\alpha) \ar[r] \ar[d]  \ar[r] & K^{*+1}(X)^\Z \ar[r]  \ar[d] & 0 \\
 0 \ar[r] & H^{**}_c(X;\Q)_\Z  \ar[r] \ar[d] & H^{**+1}_c(M_\alpha) \ar[r] \ar[d] & H^{**+1}_c(X;\Q)^\Z \ar[r] \ar[d] & 0 \\
 0 \ar[r] & H^{**}_c(X;\Q)_\Z  \ar[r] & H_{**}(\Z\ltimes X;\Q) \ar[r] & H^{**+1}_c(X;\Q)^\Z \ar[r] & 0  }
$$
where the top two rows are as in Proposition \ref{mc basics}, part \eqref{pv ct} (which shows that part of the diagram commutes), the maps between the middle two rows are the classical Chern characters (so that part of the diagram commutes by naturality of the Chern character), and the maps between the bottom two rows are as in Lemma \ref{exp low} in degree one and as in Lemma \ref{mc seq gpd} for all other degrees (and so that part of the diagram commutes by those lemmas).

To check compatibility of the pairings, consider the diagram
$$
\xymatrix{ K_0(C(X)\rtimes_{\alpha} \Z) \ar[r] \ar[d]^-\tau & K^1(M_\alpha) \ar[r] \ar[d]^-{\Lambda_\tau}& H^{od}_c(M_\alpha;\Q) \ar[r] \ar[d]^{\Lambda^H_\tau} & H_{ev}(\Z\ltimes X;\Q) \ar[d]^\tau \\
\R \ar@{=}[r] & \R \ar@{=}[r] & \R \ar@{=}[r] & \R },
$$
where the top composition defines the Chern character, and the last two vertical arrows are defined by first projecting onto $H^1(M_\alpha;\Q)$ and $H^0(\Z\ltimes X;\Q)$ and then taking the respective pairings.  The first square commutes by Proposition \ref{trace agree}, and the last square by Lemma \ref{trace agree 2}, so it remains to check commutativity of the central square.

For this, recall that for a compact space $Y$, the component of the Chern character $K^1(Y) \to H^1(Y;\Q)$ with values in $H^1_c$ is defined by taking a unitary $u\in M_n(C(Y))$ to the continuous function $\text{det}(u):Y\to S^1$ (this is even defined integrally, but we do not need it).  It thus suffices to check that if $u\in M_n(C(M_\alpha))$ is a unitary that is continuously differentiable in the $t$ variable (and equal to one outside a compact set), then 
$$
\tau(\text{det}(u)'\text{det}(u^*))=\tau(u'u^*)
$$
(where the prime denotes the derivative in the $t$-direction).  This follows as if we locally write $u=e^{2\pi i x}$ with $x$ self-adjoint, then $\text{det}(u)=e^{2\pi i \text{tr}(x)}$ and so $\text{det}(u)'=\text{tr}(x')\text{det}(u)$, while on the other hand $u'=x'u$.  The result follows from these local computations.
\end{proof}

\begin{corollary}\label{int cor}
A transformation groupoid $\Z\ltimes X$ with compact base space is HK-good under either of the following conditions:
\begin{enumerate}[(i)]
\item \label{ch4} the cohomology of $X$ vanishes above dimension three, and the Chern class map $c:K^*(X)\to H^{**}(X)$ of Definition \ref{chb2} is an isomorphism; or
\item \label{chfree} there is an integral Chern isomorphism for $X$ in the sense of Definition \ref{ch int}, and the groups $K^i(X)$ are free. 
\end{enumerate}
\end{corollary}

\begin{proof}
For part \eqref{ch4}, we will show that there is an integral Chern isomorphism for $M_\alpha$ in the sense of definition \ref{ch int}.  We first note that the short exact sequences
\begin{equation}\label{ma ses}
0\to H^i_c(SX)_\Z\to H^i(M_\alpha)\to H^i(X)^\Z\to 0
\end{equation}
from Proposition \ref{free gp ses} imply that $H^i(M_\alpha)=0$ for $i>4$, so we only need to work in degrees four or less.  Moreover, this short exact sequence gives an isomorphism $H^4_c(SX)_\Z\to H^4(M_\alpha)$; this implies in particular that the map $H^4_c(SX)\to H^4(M_\alpha)$ is surjective, and as cup products are zero on a suspension (see Lemma \ref{no cup}) we are in the situation of Lemma \ref{chb hom} where the map $c$ of Definition \ref{ch basic} is a homomorphism for $M_\alpha$. 

Hence we have a commutative diagram
$$
\xymatrix{0 \ar[r] & K^*(X)_\Z  \ar[r] \ar[d]^-c & K^*(M_\alpha) \ar[r] \ar[d]^-c  \ar[r] & K^{*+1}(X)^\Z \ar[r]  \ar[d]^-c & 0  \\
 0 \ar[r] & H^{**}_c(X)_\Z  \ar[r] & H^{**}(M_\alpha) \ar[r] & H^{**+1}_c(X)^\Z \ar[r] & 0  }
$$   
of such Chern class maps.  The left and right vertical maps are induced by the isomorphism $c:K^*(X)\cong H^{**}_c(X)$; by naturality of $c$, this is also an isomorphism of $\Z$-modules whence it induce isomorphisms on invariants and coinvariants.  The result follows from the five lemma.

For part \eqref{chfree}, if $K^i(X)$ is free, then so are $K^i(X)^\Z$ (as subgroups of free groups are free) and $H^i(X)^\Z$ (as the assumed integral Chern isomorphism must also induce an isomorphism on fixed point subgroups by naturality).  Hence the Pimnser-Voiculescu type sequences all split (unnaturally), and the result follows by choosing the splittings appropriately (this is possible as everything in sight is free: we leave the algebra to the reader).    
\end{proof}

\section{Examples of crossed products by the integers} \label{Sec:crossed}

In this section, we discuss a number of crossed product groupoids and their associated $C^*$-algebras. Each is obtained by a free action of the integers on a compact metric space. First, as promised in Section \ref{Sec:Irrational}, we give a result covering the irrational rotation algebras.

\begin{example}\label{tori ex}
Let $X$ be the $d$-torus or the $d$-sphere for some $d$, and take any free $\Z$-action.  Then by Corollary \ref{int cor} part \eqref{chfree} and Example \ref{sphere}, the groupoid $\Z\ltimes X$ is HK-good.   This covers the case of the irrational rotation algebras in particular.
\end{example}

Our next example provides a negative result: we give an example of a minimal action on a manifold so that the transformation groupoid is not HK-good.

\begin{example} \label{Ex:CrossNotGood}
Let $M=S^3 \times \mathbb{R}P^4$. Then, by \cite[Theorem 1]{Fathi:1977aa}, there exists $\varphi : M \rightarrow M$ a minimal diffeomorphism that is homotopic to the identity and uniquely ergodic. The Pimsner--Voiculescu sequence and the fact that $\varphi$ is homotopic to the identity imply that 
\[ K_0( C(M) \rtimes \Z) \cong K_1(C(M) \rtimes \Z) \cong K^0(M) \oplus K^1(M). \]
Since 
\[ K^0(S^3\times \mathbb{R}P^4) \cong K^1(S^3\times \mathbb{R}P^4) \cong K^0(\mathbb{R}P^4)\oplus K^1(\mathbb{R}P^4)\cong \Z\oplus (\Z/4) \]
(see for example \cite[IV.6.47]{Karoubi:1978ai} for the $K$-theory of real projective space), we get 
\[
K_0( C(M) \rtimes \Z)  \cong K_1(C(M) \rtimes \Z) \cong \Z^2 \oplus (\Z/4)^2.
\]
On the other hand, the homology of the transformation groupoid can be computed using Proposition \ref{free gp ses} and the well-known cohomology groups of real projective space (see for example \cite[Theorem 3.12]{Hatcher:2002ud}). One gets
\[
H_{ev}(\Z\ltimes M) \cong H_{od}(\Z\ltimes M) \cong \Z^2 \oplus (\Z/2)^4.
\]
Hence $\Z\ltimes M$ is not HK-good.
\end{example}

Returning to positive results, there are other (and in fact many) transformation groupoids associated to integer actions that are HK-good. A number of the constructions build on work in \cite{Deeley:2018aa, Deeley:2023ab}. In particular, we will use a specific dynamical system constructed in \cite{Deeley:2018aa}: we summarize the key points in the next theorem.

\begin{theorem} \label{ThmAboutZ}
Let $S^d$  be a sphere with odd dimension $d\geq 3$, and let $\varphi : S^d \to S^d$  be a minimal diffeomorphism. Then there exist an infinite compact metric space $Z$ with finite covering dimension and a minimal homeomorphism \mbox{$\zeta : Z \to Z$} satisfying the following:
\begin{enumerate}
\item $Z$ is compact, connected, and homeomorphic to an inverse limit of contractible metric spaces $(Z_n, d_n)_{n \in \mathbb{N}}$.
\item For any continuous generalized cohomology theory, $\hat{H}$, we have an isomorphism $\hat{H}^*(Z) \cong \hat{H}^*(\{\mathrm{pt}\})$. In particular this holds for sheaf cohomology and $K$-theory. \label{sameCohomAsPoint}
\item There is an almost one-to-one factor map $q : Z \to S^d$ which induces a bijection between $\zeta$-invariant Borel probability measures on $Z$ and $\varphi$-invariant Borel probability measures on $S^d$. \label{FactorMap} \qed
\end{enumerate} 
\end{theorem}

\begin{example} \label{Ex:CrossedGoodSpecialCase}
Let $G$ be a finitely generated abelian group and take $Y$ be a connected finite CW-complex with dimension at most three, with 
\[
H^0(Y)\cong \Z \hbox{ and }H^2(Y) \cong G, 
\]
and other cohomology groups trivial. Using Example \ref{low ch int}, the K-theory of $Y$ is  
\[ K^0(Y) \cong \Z \oplus G \hbox{ and }K^1(Y)\cong \{ 0 \}.
\]
Let $X=Z\times Y\times Q$ where $Z$ as in Theorem \ref{ThmAboutZ} and $Q$ the Hilbert cube. In particular, 
\[
K^*(X) \cong K^*(Y) \hbox{ and }H^*(X) \cong H^*(Y).
\] 
Applying \cite[Theorem 1]{Fathi:1977aa} there exists a minimal, uniquely ergodic homeomorphism $\varphi: X \rightarrow X$ that is homotopic to the identity. The Pimsner--Voiculescu exact sequence and the fact that $\varphi$ is homotopic to the identity imply that
\[
K_0(C(X)\rtimes \Z) \cong K_1(C(X)\rtimes \Z) \cong \Z \oplus G.
\]
The groupoid homology can be computed similarly from Proposition \ref{free gp ses}, giving
\[
H_{ev}(\Z \ltimes X)\cong H_{od}(\Z\ltimes X) \cong \Z \oplus G.
\]
These groupoids are HK-good by Corollary \ref{int cor}, part \eqref{ch4}.

It is worth noting that if we take $G=\Z \oplus (\Z/4)$, then the $C^*$-algebra $C(X) \rtimes \Z$ is isomorphic to the crossed product $C^*$-algebra in Example \ref{Ex:CrossNotGood}: thus we see that this $C^*$-algebra admits both an HK-good and a non HK-good model.
\end{example}

The next class of examples we consider are crossed products where the underlying space has finitely generated $K$-theory. The starting point is a slight generalization of \cite[Theorem 4.5]{Deeley:2024ab}.
\begin{theorem} \label{thm:fgKtheory}
Suppose $d\in \N \setminus \{0\}$ and $F_0$, $F_1$ are finite abelian groups. Then there exists a connected finite CW-complex $Y$ with dimension at most three and $\varphi: X\rightarrow X$ a minimal uniquely ergodic homeomorphism such that
\[
K_0(C(X)\rtimes \Z) \cong \Z^d \oplus F_0 \hbox{ and }K_1(C(X)\rtimes \Z)\cong \Z^d\oplus F_1
\]
where $X=Z\times Y\times Q$ (with $Z$ as in Theorem \ref{ThmAboutZ} and \cite{Deeley:2018aa} and $Q$ the Hilbert cube).
\end{theorem}
\begin{proof}
The only difference between the present theorem and the statement of \cite[Theorem 4.5]{Deeley:2024ab} is the explicit statement that the space, $Y$, can be taken to be a finite CW-complex of dimension at most three.  However, this follows from the construction: see \cite[Theorem 5.8]{Deeley:2023ab} and \cite[Lemmas 4.3 and 4.4]{Deeley:2024ab}.
\end{proof}

The next result is immediate from Corollary \ref{int cor}, part \eqref{ch4}.

\begin{proposition}
Using the notation of the previous theorem, the groupoid $\Z\ltimes X$ is HK-good.  \qed
\end{proposition}

Next, a refinement of the previous two results is discussed. We state it because of the importance of projectionless unital $C^*$-algebras and as in this case the Elliott invariant of the relevant $C^*$-algebra can be completely determined (rather than just the $K$-theory as in Theorem \ref{thm:fgKtheory}).

\begin{corollary}
For any $d \in \mathbb{N} \setminus \{0\}$ and any pair of finite abelian groups $F_0, F_1$, there exists a minimal dynamical system $(X, \varphi)$ such for the crossed product $A := C(X) \rtimes_{\varphi} \mathbb{Z}$, we have
\begin{enumerate}
\item $A$ is a classifiable $C^*$-algebra;,
\item the pointed ordered $K_0$-group $(K_0(A),  K_0(A)_+,  [1])$ is isomorphic to  $( \mathbb{Z}^d \oplus F_0, \,  \mathbb{Z}_{> 0} \oplus \mathbb{Z}^{d-1} \oplus F_0 \cup (0_{\mathbb{Z}^d},0_{F_0}),\, (1, 0_{\mathbb{Z}^{d-1}}, 0_{F_0}))$;
\item $K_1(A) \cong \mathbb{Z}^d \oplus F_1$;
\item $A$ has a unique tracial state;
\item $r : T(A) \to SK_0(A)$ satisfies $r(\tau)((n_1, \dots, n_d), g)) = n_1$;
\item $A$ has no non-trivial projections; and
\item the groupoid $\Z\ltimes X$ is HK-good.
\end{enumerate}
\end{corollary}
\begin{proof}
Without the HK-good statement this is \cite[Corollary 4.6]{Deeley:2024ab}. That we can take the groupoid to be HK-good follows from the previous proposition.
\end{proof}

\begin{example}
In this example, we give a self-contained development of the construction used in Theorem \ref{thm:fgKtheory} in the special case when $d=1$, $F_0=\Z/3$, and $F_1=\Z/2$. We hope this example is useful to readers less familiar with \cite{Deeley:2024ab}.

Take $S^1\vee S^1\vee S^1$ with base point the wedge point and the finite order self-homomorphism given by cyclically permuting the copies of $S^1$. Form the reduced mapping cone, $C_f$ of the map $f: S^1\vee S^1\vee S^1 \rightarrow S^1$ defined by mapping each circle in $S^1\vee S^1\vee S^1$ to $S^1$ identically. The reduced mapping cone construction is equivariant, so we get a finite order self-homomorphism, $\beta_1: C_f \rightarrow C_f$. Using the long exact sequence in cohomology, 
\[
H^0(C_f)\cong \Z\quad  \hbox{and}\quad H^2(C_f)\cong \Z \oplus \Z
\]
where the part of the long exact sequence that is relevant for later use is 
\[
0 \rightarrow \Z \rightarrow \Z^3 \rightarrow H^2(C_f) \rightarrow 0
\]
with map $\Z \rightarrow \Z^3$ given by $n \mapsto (n,n,n)$. 

Let $e_1$, $e_2$, and $e_3$ denote respectively the images of $(1,0,0)$, $(0,1,0)$, and $(0,0,1)$ in $H^2(C_f)$. Equivariance implies that $\beta^*(e_1) = e_2$ and that $\beta^*(e_2)=e_3=-e_1-e_2$ where the second equality follows from the specific nature of the map (i.e., $n \mapsto (n,n,n)$). Therefore, we have that the map $\beta^*: \Z^2 \cong H^2(C_f)\rightarrow H^2(C_f)\cong \Z^2$ is given by the matrix
\[
A=\left( \begin{array}{cc} 0 & -1 \\ 1 & -1 \end{array}\right).
\] 

Let $Y$ be $S^3 \vee C_f$ with base point the wedge point. Define $\beta : S^3\vee C_f \rightarrow S^3\vee C_f$ via a reflection in a plane containing the base point on the $S^3$ factor and $\beta_1$ on $C_f$. By construction, 
\[ 
H^0(Y) \cong \Z,~H^2(Y) \cong \Z^2,~H^3(Y)\cong \Z,
\]
and all other cohomology groups are trivial. Moreover, $\beta^*$ is given by the identity on $H^0(Y)$, the negation of the identity on $H^3(Y)$, and the matrix $A$ on $H^2(Y)$.

A similar argument (or one can use the Chern character) implies that for $K$-theory, we have
\[
K^0(Y) \cong \Z \oplus \Z^2 \quad \hbox{and}\quad K^1(Y)\cong \Z
\]
and $\beta^*: K^0(Y) \rightarrow K^0(Y)$ is given by $id \oplus A$ and $\beta^*: K^1(Y) \rightarrow K^1(Y)$ is given by the negation of the identity. 

Let $X=Z\times Y\times Q$  where $Z$ as in Theorem \ref{ThmAboutZ} (also see \cite{Deeley:2018aa}) and $Q$ the Hilbert cube. By \cite[Theorem 5.11]{Deeley:2023ab}, there exists a minimal uniquely ergodic homeomorphism $\varphi: X \rightarrow X$ such that $\varphi^*$ is given by $\beta^*$ on both $K$-theory and cohomology (here we are identifying $K^*(Y)$ and $K^*(X)$ and likewise for cohomology). Hence the map $id - \varphi^*$ is given by the direct sum of the zero map on $\Z$ with the matrix
\[
B=\left( \begin{array}{cc} 1 & 1 \\ -1 & 2 \end{array}\right)
\]
in even degree and multiplication by $2$ in odd degree. The Smith normal form of $B$ is
\[
\left( \begin{array}{cc} 1 & 0 \\ 0 & 3 \end{array}\right).
\]
Using the Pimsner--Voiculescu exact sequence, we get that 
\[
K_0(C(X)\rtimes \Z) \cong \Z \oplus \Z/3 \quad \hbox{and}\quad K_1(C(X)\rtimes \Z) \cong \Z \oplus \Z/2.
\]
A similar computation in groupoid homology using Proposition \ref{free gp ses} gives that 
\[
H_{ev}(\Z\ltimes X)  \cong \Z \oplus \Z/3 \quad  \hbox{and}\quad H_{od}(\Z\ltimes X) \cong \Z \oplus \Z/2.
\]
That the pairings are compatible can be seen directly, or follows from Corollary \ref{int cor}, part \eqref{ch4}. Moreover, the crossed product $C(X) \rtimes \Z$ has a unique tracial state and fits within classification because the minimal system is uniquely ergodic. 
\end{example}

\section{Matui's long exact sequence for an open inclusion}\label{incl sec}

We now start working towards showing that the examples constructed by the first author, Putnam, and Strung in \cite{Deeley:2024ab} are HK-good; this will take the rest of the main body of the paper.  The first step, carried out in this section, will be to study a long-exact sequence for an open inclusion of groupoids developed by Matui in  \cite{Matui:2022aa}.  Matui works primarily with ample groupoids, i.e.\ those with zero-dimensional base spaces.  Here we generalize Matui's work to other base spaces.

Throughout this section, $\gpd$ denotes a locally compact, Hausdorff, \'{e}tale groupoid.  Let also $\gpd'$ be an open subgroupoid of $\gpd$ with the same base space.  Then $\gpd'$ is \'{e}tale in its own right, and for each $n$, $(\gpd')^{(n)}$ is an open subspace of $\gpd^{(n)}$.  Given a $c$-soft $\gpd$-sheaf $\mathcal{S}$ on $\gpd^{(n)}$, we write $\mathcal{S}_{(\gpd')^{(n)}}$ for the sheaf on $\gpd^{(n)}$ defined by setting all stalks outside $(\gpd')^{(n)}$ to be zero (compare \cite[page 11]{Bredon:1997aa}).  We write moreover $F^{(n)}:=\gpd^{(n)}\setminus (\gpd')^{(n)}$ for the closed `difference' space, and write $\mathcal{S}_{F^{(n)}}:=\mathcal{S}/\mathcal{S}_{(\gpd')^{(n)}}$ for the corresponding quotient sheaf.  Then for each $n$, using $c$-softness we have a short exact sequence 
\begin{equation}\label{diff ses}
0\to \Gamma_c(\gpd^{(n)};\mathcal{S}_{(G')^{(n)}}) \to \Gamma_c(\gpd^{(n)};\mathcal{S}) \to \Gamma_c(\gpd^{(n)};\mathcal{S}_{F^{(n)}}) \to 0. 
\end{equation}
(this follows for example from \cite[Theorem II.9.9]{Bredon:1997aa} and on noting that $\mathcal{S}_{(\gpd')^{(n)}}$ is $c$-soft, either by applying \cite[Corollary II.9.13]{Bredon:1997aa}, or arguing directly).  

Let now $\mathcal{A}$ be a $\gpd$-sheaf on $\gpd^{(0)}$, and consider a resolution
\begin{equation}\label{resolve}
0\to \mathcal{A}\to \mathcal{S}^0\to \mathcal{S}^1\to \cdots 
\end{equation}
of $\mathcal{A}$ by $c$-soft $\gpd$-sheaves as in line \eqref{reso} above.  Then the short exact sequences in line \eqref{diff ses} give rise to a short exact sequence of double complexes as in line \eqref{tot com}, and therefore to a short exact sequence of total complexes which in the $n^\text{th}$ entry looks like
{\small \begin{equation}\label{ses tc}
0\to \bigoplus_{p-q=n}\Gamma_c(\gpd^{(p)};(\mathcal{S}_p^q)_{(G')^{(p)}}) \to \bigoplus_{p-q=n}\Gamma_c(\gpd^{(p)};\mathcal{S}_p^q) \to \bigoplus_{p-q=n}\Gamma_c(\gpd^{(p)};(\mathcal{S}_p^q)_{F^{(p)}}) \to 0. 
\end{equation}}
Recall moreover that for any sheaf $\mathcal{S}$ on $G^{(n)}$ there is a canonical identification
$$
\Gamma_c((\gpd')^{(n)};\mathcal{S}|_{(\gpd')^{(n)}})=\Gamma_c(\gpd^{(n)};\mathcal{S}_{(\gpd')^{(n)}})
$$
(compare \cite[Proposition I.6.6]{Bredon:1997aa}).  It follows that regarding the resolution of line \eqref{resolve} as a resolution of $\mathcal{A}$ considered as a $\gpd'$ sheaf, we have that the left hand complex appearing in line \eqref{ses tc} computes the homology $H_*(\gpd';\mathcal{A})$.  

\begin{definition}\label{rel comp}
With notation as above, let $\mathcal{A}$ be a $\gpd$-sheaf on $\gpd^{(0)}$, which we also consider as a $\gpd'$ sheaf on $(\gpd')^{(0)}=\gpd^{(0)}$.   We define $H_n(\gpd/\gpd';\mathcal{A})$ to be the homology of the double complex 
$$
\Big(\Gamma_c(\gpd^{(p)};(\mathcal{S}_p^{-q})_{F^{(p)}})\Big)_{q=,...-1,0;~p=0,1,2,...} 
$$ 
\end{definition}

Summarizing the discussion above and using that short exact sequences of complexes induce long-exact sequences in homology, we deduce the following result.

\begin{proposition}\label{hom les}
With notation as in Definition \ref{rel comp}, there is a long exact sequence of homology groups 
$$
\cdots H_{n-1}(\gpd';\mathcal{A}) \leftarrow H_n(\gpd/\gpd';\mathcal{A})\leftarrow  H_n(\gpd;\mathcal{A})\leftarrow H_n(\gpd';\mathcal{A})\leftarrow H_{n+1}(\gpd/\gpd';\mathcal{A}) \leftarrow \cdots \eqno\qed
$$
\end{proposition}

The long exact sequence from Proposition \ref{hom les}  is not immediately useful without a way to compute the groups $H_n(\gpd/\gpd';\mathcal{A})$.  In the remainder of this section, we will adapt an `excision' type result of Matui \cite[Section 3]{Matui:2022aa} to our context; this will allow us to compute $H_n(\gpd/\gpd';\mathcal{A})$ for the orbit breaking groupoids of interest to us. Matui's result (along with our result) are the homology version of Putnam's $K$-theory result, \cite{Putnam:2021aa}.

Keeping in the situation of a pair $(\gpd,\gpd')$ of an \'{e}tale groupoid and open subgroupoid with the same base space (and following Putnam \cite[Section 6]{Putnam:2021aa} and Matui \cite[Section 3]{Matui:2022aa}), we define $\Delta:= \gpd\setminus \gpd'$.

\begin{definition}\label{reg incl}(Compare \cite[Definition 6.3]{Putnam:2021aa}).
With notation as above, equip $r(\Delta)\subseteq \gpd^{(0)}$ with the quotient topology it inherits from the surjection $r:\Delta\to r(\Delta)$.  The inclusion $\gpd'\subseteq \gpd$ is \emph{regular} if the map $r:\Delta\to r(\Delta)$ is open.
\end{definition}

\begin{definition}\label{hpd def}(Compare \cite[Definition 6.5 and Theorem 6.8]{Putnam:2021aa})
With notation as above, let $m:\gpd^{(2)}\to \gpd$ be the multiplication map, define $\hpd':=m((\Delta\times \Delta)\cap m^{-1}(\gpd'))$ equipped with the quotient topology it inherits from the surjection $m:(\Delta\times \Delta)\cap m^{-1}(\gpd')\to \hpd'$, and define $\hpd:=\hpd'\sqcup r(\Delta)$ where $r(\Delta)$ is equipped with the quotient topology as in Definition \ref{reg incl} and $\hpd$ is equipped with the disjoint union topology.
\end{definition}

In \cite[Theorems 6.7 and 6.8]{Putnam:2021aa}, Putnam shows that we have identifications of sets $\hpd=\gpd|_{r(\Delta)}$ and $\hpd'=\gpd'|_{r(\Delta)}$ (but not of  topological spaces); we use these identifications to equip $\hpd$ and $\hpd'$ with groupoid structures.  The next theorem essentially summarizes \cite[Theorems 6.7 and 6.8]{Putnam:2021aa} mentioned above, and also results of Matui \cite[Theorem 3.7 and Proposition 3.9]{Matui:2022aa}.  

\begin{theorem}\label{mat the}
With conventions as above, both $\hpd$ and $\hpd'$ are locally compact, Hausdorff, \'{e}tale groupoids, with $\hpd'$ clopen in $\hpd$.  Moreover, for any $n\in \N\cup\{0\}$, the natural inclusion 
$$
\hpd^{(n)}\setminus (\hpd')^{(n)}\to \gpd^{(n)}\setminus (\gpd')^{(n)}
$$
is a homeomorphism.
\end{theorem}

\begin{proof}
The results \cite[Theorems 6.7 and 6.8]{Putnam:2021aa} and \cite[Theorem 3.7 and Proposition 3.9]{Matui:2022aa} of Matui and Putnam show everything in the statement under the additional assumptions that $\gpd$ is second countable and totally disconnected.  Inspection of the arguments reveals that those assumptions are not necessary.
\end{proof}

\begin{proposition}\label{excision}
With conventions as above, let $\mathcal{A}$ be a $\gpd$-sheaf on $\gpd^{(0)}$, let $\iota:\hpd^{(0)}\to \gpd^{(0)}$ denote the canonical continuous map\footnote{We say `continuous map' rather than `inclusion' as while this injective continuous map is set-theoretically an inclusion, it is typically not a homeomorphism onto its image.}, and write $\iota^*\mathcal{A}$ for the pullback $\hpd$-sheaf.  Then the natural map of pairs $(\hpd,\hpd')\to (\gpd,\gpd')$ induces an isomorphism 
$$
H_*(\hpd/\hpd';\iota^*\mathcal{A})\cong H_*(\gpd/\gpd';\mathcal{A}) .
$$
\end{proposition}

\begin{proof}
By definition, $H_*(\gpd/\gpd';\mathcal{A})$ is the homology of the double complex 
$$
\Big(\Gamma_c(\gpd^{(p)};(\mathcal{S}_p^{-q})_{F^{(p)}})\Big)_{q=,...-1,0;~p=0,1,2,...} 
$$
as in the statement of Proposition \ref{hom les}, where $\mathcal{S}^\bullet$ is a resolution of $\mathcal{A}$ by $c$-soft sheaves.  Using the natural identification 
$$
\Gamma_c(\gpd^{(p)};(\mathcal{S}_p^{-q})_{F^{(p)}})=\Gamma_c(F^{(p)};(\mathcal{S}_p^{-q})|_{F^{(p)}})
$$
of \cite[Proposition II.6.6]{Bredon:1997aa} (compare also \cite[Section II.10]{Bredon:1997aa}), $H_*(\gpd/\gpd';\mathcal{A})$ is the same as the homology of the double complex 
\begin{equation}\label{g dc}
\Big(\Gamma_c(F^{(p)};(\mathcal{S}_p^{-q})|_{F^{(p)}})\Big)_{q=,...-1,0;~p=0,1,2,...} .
\end{equation}
On the other hand, note that as pullbacks are exact functors (this is straightforward from the definition: compare \cite[line (2.3.3)]{Kashiwara:1990aa}), and as pullbacks of $c$-soft sheaves by $\iota$ are $c$-soft (this follows as $\iota$ is injective and takes compact sets to compact sets) we have that $\iota^*\mathcal{S}^\bullet$ is a $c$-soft resolution of $\iota^*\mathcal{A}$.  Hence $H_*(\hpd/\hpd';\iota^*\mathcal{A})$ is the homology of the double complex 
\begin{equation}\label{h dc}
\Big(\Gamma_c(\hpd^{(p)}\setminus (\hpd')^{(p)};(\iota^*\mathcal{S}_p^{-q})|_{\hpd^{(p)}\setminus (\hpd')^{(p)}})\Big)_{q=,...-1,0;~p=0,1,2,...}.
\end{equation}
However, the last statement of Theorem \ref{mat the} implies that for each $p$, the (set-theoretic) inclusion $\hpd^{(p)}\to \gpd^{(p)}$ induces an identification of topological spaces $F_{(p)}=\hpd^{(p)}\setminus (\hpd')^{(p)}$.  Hence the double complexes in lines \eqref{g dc} and \eqref{h dc} are the same.
\end{proof}

\section{Orbit breaking subgroupoids}\label{obsec}

In this section we specialize the results of Section \ref{incl sec} to so-called orbit-breaking groupoids associated to actions of the integers.   We finish with a general result comparing homology of orbit breaking groupoids and $K$-theory of the associated $C^*$-algebras.   Throughout, we will fix a locally compact Hausdorff space $X$ and a homeomorphism $\varphi : X \to X$ inducing a free action of $\Z$ on $X$.

From the dynamical system $(X, \varphi)$, one can construct the transformation groupoid $\Z \ltimes X$ as we discussed in Section \ref{secPrelim} above.  To match the existing literature better, we instead work with the equivalence relation 
\[
\mathcal{R}_{\varphi} := \{ (\varphi^{n}(x),x ) \mid x \in X, n \in \mathbb{Z} \},
\]
which is a groupoid with the operations restricted from the pair groupoid $X\times X$.  Since the dynamical system is free, the map 
\[ \Z \ltimes X \to \mathcal{R}_{\varphi}, \quad  (n,x) \mapsto (\varphi^{n}(x),x) \] is 
 a bijection.  We equip $\mathcal{R}_{\varphi}$ with the topology such that this bijection is a homeomorphism; this makes $\mathcal{R}_\varphi$ into a topological groupoid, isomorphic to $\Z\ltimes X$.

Let now $Y$ be a closed non-empty subset of $X$ , and assume for simplicity that $Y$ has finite covering dimension.  We say that $Y$ meets every orbit at most once if $\varphi^n(Y) \cap Y = \varnothing$ for each $n \neq 0$.   Following Putnam (compare \cite{Putnam:1989hi}, \cite[ Example 2.6]{Putnam:1998aa}, and \cite{Putnam:2021aa}) we define $\mathcal{R}_Y \subseteq \mathcal{R}_{\varphi}$ to be the subequivalence relation of $\mathcal{R}_\varphi$ obtained from splitting every orbit that passes through $Y$ into two equivalence classes: precisely, define 
\begin{equation}\label{Delta def}
\Delta:=\{(\phi^k(y),\phi^l(y))\in  \mathcal{R}_{\varphi}\mid y\in Y,~l<1\leq k \text{ or } k<1\leq l\}.
\end{equation}
 and then define
$$
\mathcal{R}_{Y}: = \mathcal{R}_{\varphi} \setminus\Delta.
$$
 One checks $\mathcal{R}_Y$ is an open subequivalence relation of $\mathcal{R}_{\varphi}$. In particular, it is itself an \'{e}tale groupoid when equipped with the subspace topology.   Notice that if a given orbit does not meet $Y$, then it is an equivalence class in both $\mathcal{R}_{\varphi}$ and $\mathcal{R}_{Y}$. However, if a given orbit does meet $Y$, say at the point $y$, then (because it meets $Y$ exactly once) its orbit becomes two distinct equivalence classes in $\mathcal{R}_{Y}$; they are $\{ \varphi^{n}(y) \mid n \geq 1 \}$ and   $\{ \varphi^{n}(y) \mid n \leq 0 \}$. In this sense, the orbit has been `broken' in two at the point $y$.
 
Having introduced these preliminaries, we now apply the machinery of Section \ref{incl sec}; to match notation from that section, we write $\gpd:=\mathcal{R}_\varphi$ and $\gpd':=\mathcal{R}_Y$ for the remainder of this section.   

\begin{lemma}\label{reg}
The inclusion $\gpd'\subseteq \gpd$ is regular in the sense of Definition \ref{reg incl}.
\end{lemma}

\begin{proof}
With $\Delta=\gpd\setminus \gpd'$ as in line \eqref{Delta def}, the map 
\begin{equation}\label{delta homeo}
\Delta\to \Z\times Y\times \N,\quad (\phi^k(y),k-l,\phi^l(y))\mapsto \left\{\begin{array}{ll} (k,y,-l) & k\geq 1 \\ (k,y,l-1) & k<1 \end{array}\right.
\end{equation}
is a homeomorphism.  With respect to the homeomorphism in line \eqref{delta homeo}, we have that $r:\Delta\to r(\Delta)$ is described by 
\begin{equation}\label{r desc}
(k,y,l)\mapsto \phi^k(y).
\end{equation}
Note moreover that by assumption each orbit intersects $Y$ at most once, we have that as sets
$$
r(\Delta)=\Z\cdot Y =\bigsqcup_{n\in \Z} \phi^n(Y);
$$  
it follows from this and the description of the quotient map $:\Delta\to r(\Delta)$ in line \eqref{r desc} that the quotient topology on $r(\Delta)$ identifies with the disjoint union topology on $r(\Delta)=\bigsqcup_{n\in \Z} \phi^n(Y)$, i.e.\ with the product topology on $\Z\times Y$.  In particular, the map $r:\Delta\to r(\Delta)$ is open as required.   
\end{proof}

Let now $\hpd$ and $\hpd'$ be the \'{e}tale groupoids built from the regular inclusion $\gpd'\subseteq \gpd$ as in Definition \ref{hpd def} and Theorem \ref{mat the} above.  We note the following structural results.

\begin{proposition}\label{h str}
For a set $S$, let $P_S$ be the pair groupoid on $S$ equipped with the discrete topology.  Then the inclusion $\hpd'\to \hpd$ fits into a commuting square of groupoid homomorphisms
$$
\xymatrix{ \hpd'\ar[d]^\cong \ar[r] & \hpd \ar[d]^\cong \\
(Y\times P_{\Z_{\geq 1}})\sqcup (Y\times P_{\Z_{\leq 0}}) \ar[r] & Y\times P_\Z }
$$
where the bottom arrow is the obvious inclusion, and the vertical arrows are isomorphisms of topological groupoids.  Moreover the map
$$
\iota_Y:Y\to Y\times P_\Z,\quad y\mapsto (y,(0,0))
$$
and the map 
$$
\iota_Y:Y\sqcup Y \to (Y\times P_{\Z_{\geq 1}})\sqcup (Y\times P_{\Z_{\leq 0}})
$$
induced by $y\mapsto (y,(1,1))$ in the first copy of $Y$ and $y\mapsto (y,(0,0))$ in the second copy are Morita equivalences.  Finally these maps induce isomorphisms
$$
H^{-n}_c(Y)\stackrel{\cong}{\to} H_n(\hpd)\quad \text{and}\quad H_c^{-n}(Y)\oplus H^{-n}_c(Y)\stackrel{\cong}{\to} H_n(\hpd')
$$
for all $n\in \Z$, where the left hand sides are the usual sheaf cohomology groups.
\end{proposition}

\begin{proof}
We leave the construction of the commutative square and checking the isomorphisms and Morita equivalences to the reader.  Once one knows the Morita equivalences, the homology isomorphisms follow from the fact that Morita equivalences induce isomorphisms on groupoid homology (see \cite[Corollary 3.6]{Crainic:2000aa}\footnote{We assumed finite covering dimension for $Y$ in order to be able to apply this result: we do not think it is really necessary, but reproving it in general would require significant extra work.}) and the fact that the groupoid homology of a trivial groupoid (i.e.\ a space) identifies with the compactly supported sheaf cohomology, up to replacing degrees by their negative (see Example \ref{spex} above, or \cite[3.5 (3)]{Crainic:2000aa}).
\end{proof}

\begin{corollary}\label{y hom}
With notation as in Definition \ref{rel comp}, for each $n$ there is a canonical isomorphism
$$
H_n(\gpd/\gpd')\cong  H_n(\hpd/\hpd') \cong H^{1-n}_c(Y).
$$
\end{corollary}

\begin{proof}
The first isomorphism $H_n(\gpd/\gpd')\cong  H_n(\hpd/\hpd')$ follows from Proposition \ref{excision}.

For the second isomorphism, note that the homology isomorphisms from Proposition \ref{h str} give rise to isomorphisms
$$
H_n(\hpd)\cong H_c^{-n}(Y)\quad \text{and}\quad H_n(\hpd')\cong H_c^{-n}(Y)\oplus H_c^{-n}(Y).
$$
Combining these Morita equivalence isomorphisms and the long exact sequence of Proposition \ref{hom les} we get a long exact sequence
$$
\cdots  \leftarrow H_n(\hpd/\hpd')\leftarrow  H^{-n}_c(Y)\leftarrow H_c^{-n}(Y)\oplus H_c^{-n}(Y)\leftarrow H_{n+1}(\hpd/\hpd';\mathcal{A}) \leftarrow\cdots
$$
where the map $H^{-n}_c(Y)\oplus H^{-n}_c(Y)\to  H^{-n}_c(Y)$ is given by $(x,y)\mapsto x+y$ and is in particular surjective.  The long exact sequence above thus reduces to a collection of short exact sequences
$$
 0 \leftarrow H^{-n}_c(Y) \leftarrow H_c^{-n}(Y)\oplus H_c^{-n}(Y) \leftarrow H_{n+1}(\hpd/\hpd') \leftarrow 0 .
$$
It follows that $H_{n+1}(\hpd/\hpd')$ identifies with the subgroup of $H^{-n}_c(Y)\oplus H^{-n}_c(Y)$ consisting of elements of the form $(x,-x)$, and we are done.
\end{proof}

Assume from now on that $X$ (and hence also $Y$) is compact.  Let $\mathcal{Z}$ be the sheaf of locally constant $\Z$-valued functions on $G^{(0)}$, and fix a resolution 
\begin{equation}\label{res2}
0\to \mathcal{Z} \to \mathcal{S}^0  \to \mathcal{S}^1\to \cdots
\end{equation}
by $c$-soft $\gpd$-sheaves as in line \eqref{resolve} above. 

Recall the notation $\mathcal{S}_p$ from line \eqref{s sub n} above.  Let now $u\in \Gamma_c(\gpd^{(1)};\mathcal{S}^0_1)$ be the image of the characteristic function of $\{(\phi(x),x)\mid x\in X\}$, which is an element of $\Gamma_c(\gpd^{(1)};\mathcal{Z}_1)$, under the map induced by the resolution in line \eqref{res2}.  We have the following lemma.

\begin{lemma}\label{u cocycle}
The element $u$ defined above is a cocycle for $H_1(\gpd)$, and in particular defines a class $[u]$ in this group.
\end{lemma}

\begin{proof}
Looking at the double complex in line \eqref{tot com} we have differentials as pictured
$$
\xymatrix{ \Gamma_c(\gpd^{(0)};\mathcal{S}_0^0) & \ar[l]^-{\delta} \Gamma_c(\gpd^{(1)};\mathcal{S}_1^0) \ar[d]^- {\partial} \\
& \Gamma_c(\gpd^{(1)};\mathcal{S}_1^1)},
$$
where $\partial$ comes from the resolution in line \eqref{res2}, and $\delta$ is the differential from the bar complex defined in \cite[3.1]{Crainic:2000aa}.  We must show that $\partial(u)=\delta(u)=0$.  For $\partial$ this follows as $u$ comes via the inclusion $\Gamma_c(\gpd^{(1)};\mathcal{Z}_1)\to \Gamma_c(\gpd^{(1)};\mathcal{S}^0_1)$, and as the composition of any two arrows in line \eqref{res2} is zero.  To show $\delta(u)=0$ is suffices to show that the corresponding map
$$
\delta:\Gamma_c(\gpd^{(1)};\mathcal{Z}_1)\to \Gamma_c(\gpd^{(0)};\mathcal{Z})
$$
sends $u$ to zero.  Indeed, comparing to line \eqref{first bound} above, it sends $u$ to $\chi_{\phi(X)}-\chi_X$.  As $\phi$ is a homeomorphism, $\phi(X)=X$, so we are done.
\end{proof}

\begin{lemma}\label{id quot}
With notation as above, the image of the class $[u]\in H_1(\gpd)$ under the map 
$$
H^0(Y) \cong H_1(\gpd/\gpd') \leftarrow H_1(\gpd)
$$
from the long exact sequence in Proposition \ref{hom les} and the isomorphism in Corollary \ref{y hom} is the class $[\chi_Y]$ of the characteristic function of $Y$ (identified as an element of $H^0_c(Y)$ via the inclusion $\Gamma_c(Y;\mathcal{Z}|_Y)\to \Gamma_c(Y;\mathcal{S}^0|_Y)$). 
\end{lemma}

\begin{proof}
According to the proof of Proposition \ref{excision}, we have an identification 
$$
H_1(\gpd/ \gpd')\cong H_1(\hpd/ \hpd');
$$
abusing notation, also write $[u]$ for the class of the image of $u$ in $H_1(\hpd/\hpd')$.  According to the proof of Proposition \ref{y hom} it suffices to show that the image of $[u]$ under the composition
$$
H_1(\hpd/\hpd')\to H_0(\hpd')\cong H^0_c(Y)\oplus H^0_c(Y)
$$
of the boundary map from Proposition \ref{hom les} and Morita equivalence isomorphism from Proposition \ref{h str} identifies with the class $([\chi_Y],-[\chi_Y])$.  According to the definition of the boundary map in a long exact sequence of homology groups associated to a short exact sequence of chain complexes, we can compute the image of $[u]$ in $H_0(\hpd')$ by: (1) restricting $u$ to $\hpd\setminus \hpd'$, getting an element of $\Gamma_c(\hpd\setminus \hpd';\mathcal{Z})$; then (2) lifting it to $\hpd$, getting an element of $\Gamma_c(\hpd;\mathcal{Z})$; and then (3) taking the image under the boundary map $\Gamma_c(\hpd;\mathcal{Z})\to \Gamma_c(\hpd^{(0)};\mathcal{Z})$ (which identifies with $\Gamma_c((\hpd')^{(0)};\mathcal{Z})$).  Passing through this process, steps (1) and (2) give us the characteristic function of $\{(\phi(y),y)\mid y\in Y\}$ in $\Gamma_c(\hpd;\mathcal{Z})$ and the boundary map takes this to $\chi_{\phi(Y)} - \chi_Y$. Under the explicit Morita equivalence isomorphism 
$$
H_0(\hpd')\cong H^0_c(Y)\oplus H^0_c(Y)
$$
of Proposition \ref{h str} this is exactly what we wanted, so we are done.
\end{proof}

The following result summarizes our main result on the homology of orbit breaking subgroupoids.  For ease of reading, we now pass back from the $\gpd$ and $\gpd'$ notation to the more specific orbit breaking groupoid notation.  For the statement of the result, recall that the action of $\Z$ on $X$ is \emph{minimal} if all orbits are dense.

\begin{proposition} \label{lem:long-Z-action-orbit-break}
Suppose $(X, \varphi)$ is a free dynamical system, $Y \subseteq X$ is closed and meets every orbit at most once, $\mathcal{R}_{\varphi}$ is the transformation groupoid associated to $\varphi$, and $\mathcal{R}_Y$ is the orbit breaking groupoid obtained from $Y$. Then there is a long exact sequence in groupoid homology given by
$$
\cdots \leftarrow H_{n-1}(\mathcal{R}_Y) \leftarrow H^{1-n}_c(Y)\leftarrow  H_n(\mathcal{R}_{\varphi})\leftarrow H_n(\mathcal{R}_Y)\leftarrow H^{-n}_c(Y) \leftarrow \cdots.
$$
Moreover, the following hold:
\begin{enumerate}[(i)]
\item \label{0 surj} If $\mathcal{R}_{\varphi}$ is such that the natural map $H^0(X)\to H_0(\mathcal{R}_{\varphi})$ coming from the inclusion of the first column in line \eqref{tot com} is surjective, then the map $H_0(\mathcal{R}_{\varphi})\leftarrow H_0(\mathcal{R}_Y)$ from the long exact sequence above is also surjective.
\item \label{z iso} If $X$ is compact, the action is minimal, and $Y$ is connected, then both groups $H^0_c(Y)$ and $H_1(\mathcal{R}_{\varphi})$ appearing in this long exact sequence are isomorphic to $\Z$, and the map $H^0_c(Y)\leftarrow  H_1(\mathcal{R}_{\varphi})$ is an isomorphism.
\end{enumerate}
\end{proposition}

\begin{proof}
The long exact sequence obtained from the inclusion $\mathcal{R}_Y\subseteq \mathcal{R}_{\varphi}$ as in Proposition \ref{hom les} is given by
$$
\cdots \leftarrow H_{n-1}(\mathcal{R}_Y) \leftarrow H_n(\mathcal{R}_{\varphi}/\mathcal{R}_Y)\leftarrow  H_n(\mathcal{R}_{\varphi})\leftarrow H_n(\mathcal{R}_Y)\leftarrow H_{n+1}(\mathcal{R}_{\varphi}/\mathcal{R}_Y) \leftarrow \cdots.
$$
On the other hand, Corollary \ref{y hom} gives us isomorphisms 
$$
H_n(\mathcal{R}_{\varphi}/\mathcal{R}_Y)\cong H^{1-n}_c(Y)
$$
for each $n$, so we get the long exact sequence in the statement.

For part \eqref{0 surj}, we note that as $\mathcal{R}_{\varphi}$ and $\mathcal{R}_Y$ both have base space $X$, the map $H_0(\mathcal{R}_{\varphi})\leftarrow H_0(\mathcal{R}_Y)$ from the long exact sequence (which is just induced by inclusion) fits into a commutative diagram
$$
\xymatrix{ H^0(X) \ar[d] & \ar@{=}[l] H^0(X) \ar[d] \\ 
 H_n(\mathcal{R}_{\varphi}) &  H_n(\mathcal{R}_Y) \ar[l] }.
$$
Surjectivity of the bottom horizontal arrow therefore follows from surjectivity of the left vertical arrow.

For part \eqref{z iso}, note that as $Y$ is connected and compact, $H^0(Y)$ is isomorphic to $\Z$ and generated by $[\chi_Y]$.  Lemma \ref{id quot} thus implies that the map $H^0(Y)\leftarrow  H_1(\mathcal{R}_{\varphi})$ is surjective.  On the other hand, $H_1(\mathcal{R}_{\varphi})\cong H^0(X)^\Z$ by Proposition \ref{free gp ses} (whether or not $Y$ is connected).  As $H^0(X)$ consists of locally constant functions from $X$ to $\Z$ with the $\Z$-action induced from the action on $X$, the only classes in $H^0(X)$ that are invariant for the $\Z$-action are constant by minimality.  Hence $H_1(\mathcal{R}_{\varphi})\cong \Z$.  As the map $H^0(Y)\leftarrow  H_1(\mathcal{R}_{\varphi})$ is surjective and both groups are copies of $\Z$, it is an isomorphism as claimed.
\end{proof}

In the remainder of this section, we derive some sufficient conditions for an orbit-breaking groupoid to be HK-good.  First, we record a lemma about $K$-theory which is essentially contained in \cite{Deeley:2023ab}.

\begin{lemma}\label{dps struct}
Suppose $(X, \varphi)$ is a free and minimal dynamical system, $Y \subseteq X$ is closed and meets every orbit at most once, $\mathcal{R}_{\varphi}$ is the transformation groupoid associated to $\varphi$, and $\mathcal{R}_Y$ is the orbit breaking groupoid obtained from $Y$.  Assume moreover that $K^1(X)=0$, and that the canonical map $C(X,\Z)\to K^0(X)$ is an isomorphism.

Then there are a short exact sequence 
$$
0\to \widetilde{K}^0(Y) \to K_0(C^*_r(\mathcal{R}_Y)) \to  K_0(C^*_r(\mathcal{R}_\phi)) \to 0
$$
and an isomorphism $K^1(Y)\to K_1(C^*_r(\mathcal{R}_Y))$.  
\end{lemma}

\begin{proof}
Putnam \cite[Theorem 2.5 and Example 2.4]{Putnam:1998aa} constructs a six-term exact sequence 
\begin{equation}\label{base 6tex}
\xymatrix{ K^0(Y) \ar[r] & K_0(C^*_r(\mathcal{R}_Y)) \ar[r] & K_0(C^*_r(\mathcal{R}_\phi)) \ar[d] \\
K_1(C^*_r(\mathcal{R}_\phi)) \ar[u] & K_1(C^*_r(\mathcal{R}_Y)) \ar[l] & K^1(Y) \ar[l] }.
\end{equation}
where the maps $K_i(C^*_r(\mathcal{R}_Y))\to K_i(C^*_r(\mathcal{R}_\phi))$ are induced by inclusion (this works without the assumptions on $K^0(X)$).  As $K^0(X)=C(X,\Z)$ and as the action on $X$ is minimal, we therefore have that 
$$
K^0(X)^\Z\cong \Z,
$$
generated by the constant function with value one.  Hence in particular, from this and the Pimsner-Voiculescu exact sequence (plus the fact that $K^1(X)=0$)  $K_1(C^*_r(\mathcal{R}_\phi))\cong \Z$, generated by the canonical unitary implementing the $\Z$-action.  On the other hand, \cite[Corollary 5.2]{Deeley:2023ab} implies that the map $K_1(C^*_r(\mathcal{R}_\phi))\to K^0(Y)$ appearing in line \eqref{base 6tex} takes the generator of this copy of $K_1(C^*_r(\mathcal{R}_\phi))$ to the class $[1]$ of the constant function on $Y$, and so the exact sequence in line \eqref{base 6tex} implies the existence of an exact sequence
\begin{equation}\label{second 6tex}
\xymatrix{ \widetilde{K}^0(Y) \ar[r] & K_0(C^*_r(\mathcal{R}_Y)) \ar[r] & K_0(C^*_r(\mathcal{R}_\phi)) \ar[d] \\
0 \ar[u] & K_1(C^*_r(\mathcal{R}_Y)) \ar[l] & K^1(Y) \ar[l] }.
\end{equation}
Continuing, the Pimsner-Voiculescu exact sequence for $\mathcal{R}_\phi$ again implies that the canonical inclusion $C(X)\to C^*_r(\mathcal{R}_\phi)$ induces a surjection on $K_0$.  As the inclusion $C(X)\to C^*_r(\mathcal{R}_\phi)$ factors through the inclusion $C(X)\to C^*_r(\mathcal{R}_Y)$ (this follows as $X$ is the base space of both $\mathcal{R}_\phi$ and $\mathcal{R}_Y$) we get that the map $K_0(C^*_r(\mathcal{R}_Y)) \to K_0(C^*_r(\mathcal{R}_\phi))$ is surjective, and so the exact sequence in line \eqref{second 6tex} simplifies to the short exact sequence and isomorphism in the statement.
\end{proof}

\begin{proposition}\label{orbit break main}
Suppose $(X, \varphi)$ is a free dynamical system, $Y \subseteq X$ is closed and meets every orbit at most once, $\mathcal{R}_{\varphi}$ is the transformation groupoid associated to $\varphi$, and $\mathcal{R}_Y$ is the orbit breaking groupoid obtained from $Y$.  We make the following further assumptions:
\begin{enumerate}[(i)]
\item $K^0(X)=C(X,\Z)=H^0(X)$, $K^1(X)=0$, and $H^i(X)=0$ for $i\geq 1$;
\item $Y$ is connected, and has covering dimension at most three;
\item the quotient map $K_0(C^*_r(\mathcal{R}_Y)) \to K_0(C^*_r(\mathcal{R}_\phi))$ from the short exact sequence from Lemma \ref{dps struct} splits.\footnote{We do not know if this last condition is necessary: it is possible it follows from the other assumptions.}
\end{enumerate}
Then $\mathcal{R}_Y$ is HK-good.
\end{proposition}

\begin{proof}
Proposition \ref{free gp ses} plus the homological assumptions on $X$ imply that 
$$
H_i(\mathcal{R}_\phi)\cong \left\{\begin{array}{ll} H^0(X)_\Z & i=0 \\ \Z & i=1 \\ 0 & \text{otherwise} \end{array}\right..
$$
Proposition \ref{lem:long-Z-action-orbit-break} therefore implies that the long exact sequence given there induces isomorphisms 
\begin{equation}\label{hom isos}
 H^0(X)_\Z\cong  H_0(\mathcal{R}_Y)\cong H_0(\mathcal{R}_\phi) \quad \text{and}\quad H_i(\mathcal{R}_Y) \cong H^{-i}(Y) \text{ for } i<0
\end{equation}
(the first two are induced by the open inclusions $X\to \mathcal{R}_Y\to \mathcal{R}_\phi$), and that the groupoid homology is zero in other dimensions.  

On the other hand, Lemma \ref{dps struct} and the assumptions implies that we have isomorphisms
\begin{equation}\label{k isos}
K_0(C^*_r(\mathcal{R}_Y))\cong K_0(C_r^*(\mathcal{R}_\phi))\oplus \widetilde{K}^0(Y) \quad \text{and}\quad K_1(C^*_r(\mathcal{R}_Y))\cong K^1(Y)
\end{equation}
(the component of the first isomorphism giving the map $K_0(C^*_r(\mathcal{R}_Y))\cong K_0(C_r^*(\mathcal{R}_\phi))$ is induced by the open inclusion).  Moreover, from the Pimsner-Voiculescu sequence for $\phi$, we have a canonical isomorphism
\begin{equation}
K^0(X)_\Z\cong K_0(C_r^*(\mathcal{R}_\phi))
\end{equation}
induced by the open inclusion $X\to \mathcal{R}_\phi$.  

Now, the assumption that $Y$ has covering dimension at most three implies that there is an integral Chern isomorphism for $Y$ in the sense of Definition \ref{ch int} (see Example \ref{low ch int 2}), and the fact that $Y$ is connected therefore implies that $\widetilde{K}^0(Y)\cong H^2(Y)$.  Combining lines \eqref{hom isos} and \eqref{k isos} implies therefore that we have isomorphisms
$$
K_0(C^*_r(\mathcal{R}_Y))\cong K_0(C_r^*(\mathcal{R}_\phi))\oplus \widetilde{K}^0(Y)\cong H_0(\mathcal{R}_\phi)\oplus H^2(Y)$$
and 
$$
K_1(C^*_r(\mathcal{R}_\phi))\cong H^1(Y)\oplus H^3(Y)
$$
with the first isomorphism respecting the direct sum decomposition, and moreover taking $[1]\in K_0(C^*_r(\mathcal{R}_Y))$ to $[1]\in H_0(\mathcal{R}_\phi)$.  This is the statement about homology and $K$-theory we need for $\mathcal{R}_Y$ to be HK-good, so it remains to check the conditions on traces.  

Now, by Lemma \ref{meas to trace}, we have that the canonical map $T(\mathcal{R}_Y)\to T(C^*_r(\mathcal{R}_Y))$ is an isomorphism, so it remains to check the pairings.  Note first that the canonical map $T(\mathcal{R}_\phi)\to T(C^*_r(\mathcal{R}_\phi))$ is also an isomorphism by Lemma \ref{meas to trace} again.  Note moreover that for any $\tau \in T(\mathcal{R}_\phi)\cong T(C^*_r(\mathcal{R}_\phi))$, the pairings with $H_0$ and $K_0$ are compatible with the isomorphism $H_0(\mathcal{R}_\phi)\cong K_0(C^*_r(\mathcal{R}_\phi))$ using that $\mathcal{R}_\phi$ is HK-good (this follows from Corollary \ref{int cor}, part \eqref{ch4}).  On the other hand, using \cite[Theorem 12.3.12]{Giordano:2018aa}, the canonical restriction map $T(C^*_r(\mathcal{R}_\phi))\to T(C^*_r(\mathcal{R}_Y))$ is a homeomorphism, and therefore this is true for invariant measures also.  It follows that the pairings of elements of $K_0(C^*_r(\mathcal{R}_Y))$ (respectively, $H_0(\mathcal{R}_Y)$) with a trace agree with the corresponding pairings of their images in $K_0(C^*_r(\mathcal{R}_\phi))$ (respectively $H_0(\mathcal{R}_\phi)$), and these agree as we have already noted.
\end{proof}

\section{Homology and point-like systems} \label{Sec:point}
In this section, we refine one of the main constructions in \cite{Deeley:2024ab}.  Here is a statement of the result we are interested in: see \cite[Corollary 6.4]{Deeley:2024ab} for more details.

\begin{theorem}[Deeley-Putnam-Strung]\label{pt like}
Let $G_0$ and $G_1$ be countable abelian groups, let $\Delta$ be a finite-dimensional simplex, and let $\rho:\Delta\to \text{Hom}(\Z\oplus G_0,\R)$ be the constant map sending all $\delta\in \Delta$ to the map $(n,g)\mapsto n$.  

Then there is a locally compact, Hausdorff, \'{e}tale, principal, amenable groupoid $\gpd$ with compact base space of finite covering dimension and such that the Elliott invariant of $C^*_r(\mathcal{G})$ agrees with $(G_0,G_1,(1,0)\footnote{One can also achieve that the class of the unit is $(k,0)$ for any $k\geq 1$ on adjusting $\rho$ correspondingly; we could also deal with that case, but avoid it for simplicity.},\Delta,\rho)$. \qed
\end{theorem}

Our goal in this section is to show that the groupoid in this construction can also be chosen to be HK-good.

The starting point for the construction of $\gpd$ in \cite{Deeley:2024ab} is the dynamical system constructed in \cite{Deeley:2018aa}. The main features of this system are summarized in Theorem \ref{ThmAboutZ} above. In particular, $(Z, \zeta)$ will denote this dynamical system throughout our discussion.

\begin{lemma} \label{ThmAboutZhomology}
Suppose $(Z, \zeta)$ is as in Theorem \ref{ThmAboutZ}. Then the homology of the transformation groupoid associated to $(Z, \zeta)$, $\mathcal{R}_{\zeta}$ satisfies
$$
H_0(\mathcal{R}_{\zeta})\cong H_1(\mathcal{R}_{\zeta}) \cong \Z
$$
and is trivial otherwise.
\end{lemma}

\begin{proof}
Theorem \ref{ThmAboutZ} implies that $H^0(Z) \cong \Z$ and $H^i(Z)$ is trivial when $i\neq 0$. Since $\zeta$ is a homeomorphism and $Z$ is connected, $\zeta^*$ is the identity map on $H^0(Z) \cong \Z$.  The result follows from Proposition \ref{free gp ses}.
\end{proof}

Now, let $G_0$ and $G_1$ be countable abelian groups as in Theorem \ref{pt like}. Standard results (see for example \cite[Exercise 13.2]{Rordam:2000mz}) imply that we can take a compact connected metric space $Y$ of covering dimension at most three such that
\[ K^0(Y) \cong H^{ev}(Y) \cong \mathbb{Z} \oplus G_0, \qquad K^1(Y) \cong H^{od}(Y)  \cong G_1.\]
We now consider $Y$ fixed for the rest of the construction below. Let $d$ be an odd number large enough such that there exists an embedding $Y \hookrightarrow S^{d-2}$: for example, $d=9$ would be good enough by \cite[Theorem V.2]{Hurewicz:1948aa}. 

Let $(Z, \zeta)$ be a minimal dynamical system constructed from a minimal diffeomorphism $\varphi : S^d \rightarrow S^d$ as given by Theorem~\ref{ThmAboutZ}.  Using the main result of \cite{Windsor:2001aa}, we may assume that the simplex of invariant measures on $S^d$ agrees with the given simplex $\Delta$ from Theorem \ref{pt like}.  We will use the following lemma from \cite{Deeley:2024ab}; note that the embedding of $Y$ into $S^{d-2}$ is used in the proof of this lemma, see \cite[Lemma 6.2]{Deeley:2024ab} for details.

\begin{lemma}\label{disj emb}
There exists an embedding $\iota : Y \to Z$ such that $\varphi^n(\iota(Y)) \cap \iota(Y) = \emptyset$ for every $n \in \mathbb{N} \setminus \{0\}$.  \qed
\end{lemma}

Using Lemma \ref{disj emb}, we can consider the orbit breaking subgroupoid $\mathcal{R}_Y \subseteq \mathcal{R}_{\zeta}$.  The groupoid $\gpd=\mathcal{R}_Y$ is shown in  \cite[Corollary 6.4]{Deeley:2024ab} to have the properties in Theorem \ref{pt like}.  The following result is now immediate from Proposition \ref{orbit break main}.

\begin{theorem} \label{K(A_Y)}
With notation as above, the groupoid $\mathcal{R}_Y$ is HK-good.  \qed
\end{theorem}

\begin{remark} \label{rem:HKbad}
If we take $Y$ as above, but without the assumption that $K^*(Y) \cong H^*(Y)$, then we get groupoids that are not HK-good. For an explicit an example, the reader can consider the case when $Y=\mathbb{R}P^4$.
\end{remark}

\section{Homology and Cantor-like systems} \label{Sec:Cantor}

In this section, we refine another construction from \cite{Deeley:2024ab} so that homology is taken into account.  Here is a statement of the result we are interested in: see \cite[Corollary 7.4]{Deeley:2024ab} for more details.

\begin{theorem}[Deeley-Putnam-Strung]\label{cant like}
Let $G_1$ and $T$ be countable abelian groups, and let $G_0$ be a simple dimension group\footnote{See for example \cite[Definitions 5.1.6 and 7.2.4]{Rordam:2000mz}} with a specified order unit $u\in G_0$.  Let $\Delta$ be the Choquet simplex of order preserving homomorphisms $G_0\to \R$ that take the given order unit $u$ to $1$\footnote{i.e.\ the states on the ordered scaled abelian group $(G_0,u)$ in the sense of \cite[Section 5.2*]{Rordam:2000mz}.}, and let $\rho:\Delta\to \text{Hom}(T\oplus G_0,\R)$ be the map satisfying $\rho(\delta)(t,g)=\delta(g)$.  

Then there is a locally compact, Hausdorff, \'{e}tale, principal, amenable groupoid $\gpd$ with compact base space of finite covering dimension, such that the groupoid $C^*$-algebra is classifiable and real rank zero, and such that the Elliott invariant of $C^*_r(\mathcal{G})$ agrees with $(T\oplus G_0,G_1,(0,u),\Delta,\rho)$.  \qed
\end{theorem}

The starting point for establishing this at the dynamical system level is the following result.  It comes from \cite{Deeley:2024ab}, but is very much based on work of Floyd \cite{Floyd:1949aa} and its generalization by Gjerde and Johansen \cite{Gjerde:1999aa}.

\begin{theorem} \label{FloGjeJohSysThm}
Let $(K, \varphi)$ be a minimal dynamical system with $K$ the Cantor set and let $d \geq 1$ be a natural number. Then for any $d\geq 1$ there exists a minimal system 
$(\tilde{K}, \tilde{\varphi})$ with a factor map 
$$
\pi: (\tilde{K}, \tilde{\varphi}) \rightarrow (K, \varphi)
$$
such that for each $x\in K$, $\pi^{-1}(x)$ is either $[0,1]^d$, or a single point, and both cases occur.

Moreover, the map $\pi$ induces isomorphisms $\pi^{*}:K^{*}(K) \rightarrow K^{*}(\tilde{K})$ and $\pi_*:T(\Z\ltimes \tilde{K})\to T(\Z\ltimes K)$. \qed
\end{theorem}

\begin{lemma}\label{cant lem} 
Suppose that $(K, \varphi)$ is a Cantor minimal system. Then $H_0(\mathcal{R}_{\varphi}) \cong H^0(K)_\Z$, $H_1(\mathcal{R}_{\varphi}) \cong \Z$, and all other homology groups are trivial.
\end{lemma}

\begin{proof}
The isomorphism $H_0(\mathcal{R}_{\varphi}) \cong H^0(K)_\Z$ follows from the short exact sequence
$$
0\to H^0(K)_\Z\to H_0(\mathcal{R}_\phi)\to H^1(K)^\Z \to 0
$$ 
of Proposition \ref{free gp ses} and the fact that $H^1(K)=0$.  For the isomorphism $H_1(\mathcal{R}_{\varphi}) \cong \Z$ note that $H_1(\mathcal{R}_{\varphi})\cong H^0(K)^\Z$ by Proposition \ref{free gp ses}.  On the other hand, $H^0(K)$ is the space $C(K,\Z)$ of continuous functions from $K$ to $\Z$, and minimality thus forces $H^0(K)^\Z\cong \Z$ as the only invariant functions in $C(K,\Z)$ are constant.  The vanishing statements follow from the short exact sequences of Proposition \ref{free gp ses} and the fact that $H^i(K)=0$ for $i\neq 0$.
\end{proof}

\begin{lemma} \label{lem:Cantor-fat-Cantor-same-homology}
The map $\pi: (\tilde{K}, \tilde{\varphi}) \rightarrow (K, \varphi)$ induces isomorphisms $\pi^*:H^*(K)\to H^*(\tilde K)$ and $\pi^*:H_*(\mathcal{R}_{\varphi}) \rightarrow H_*(\mathcal{R}_{\tilde{\varphi}})$.  
\end{lemma}

\begin{proof}
The map $\pi:\tilde{K}\to K$ induces an isomorphism on cohomology by the Vietoris mapping theorem (see for example \cite[Theorem II.11.7]{Bredon:1997aa}).  As it is equivariant, it therefore also induces isomorphisms $\pi^*:H^i(K)^\Z\to H^i(\tilde{K})^\Z$ and $\pi^*: H^i(K)_\Z\to H^i(\tilde{K})_\Z$ for all $i$.  The result follows from naturality of the short exact sequences in Proposition \ref{free gp ses} and the five lemma.
\end{proof}

We now consider an orbit breaking groupoid again.  Let $G_0$, $T$, and $G_1$ be as in Theorem \ref{cant like}.  Let $Y$ be a compact connected metric space of covering dimension at most three such that $H^2(Y)\cong T$ and $H^1(Y)\oplus H^3(Y)\cong G_1$.  Let also $K$ be a Cantor minimal system such that $K_0(C(K)\rtimes \Z)$ is isomorphic to $G_0$ as an ordered, scaled group (such a system exists by \cite[Corollary 8.7]{Herman:1992aa}).  Let $d$ be large enough so that $Y$ embeds into $[0,1]^d$ ($d\geq 7$ is good enough by \cite[Theorem V.2]{Hurewicz:1948aa}).   Suppose that $(K, \varphi)$, $(\tilde{K}, \tilde{\varphi})$ and $\pi:\tilde{K}\to K$ are as in Theorem \ref{lem:Cantor-fat-Cantor-same-homology} and let $x\in K$ be such that $\pi^{-1}(x)\cong [0,1]^d$.  Embed $Y$ in $\pi^{-1}(x)$, and let $\mathcal{R}_Y$ be the associated orbit breaking subgroupoid of $\mathcal{R}_{\tilde\varphi}$. 

Then the groupoid $\mathcal{R}_Y$ is shown to satisfy all the conditions of Theorem \ref{cant like} in \cite[Corollary 7.4]{Deeley:2024ab}.  Our last result in this section is that $\mathcal{R}_Y$ as above is HK-good.  It is almost immediate from Proposition \ref{orbit break main}: only thing that still needs to be checked is the splitting of the short exact sequence from Lemma \ref{dps struct}, which is contained in the proof of \cite[Theorem 7.3]{Putnam:1989hi}.

\begin{theorem}
With notation as above, the groupoid $\mathcal{R}_Y$ is $HK$-good. \qed
\end{theorem}

\begin{example} \label{Ex:ManyHKgood}
In this example, we consider a special case of the above construction. Let $\theta\in \R$ be irrational, and define $G_0 \cong \Z + \theta \Z \subseteq \R$ (a simple dimension group), $G_1\cong \Z^2$, and $T$ trivial. Then the construction in this section gives a groupoid $\gpd$ that is HK-good and has associated $C^*$-algebra the irrational rotation $C^*$-algebra. However, the groupoid is a different groupoid than the standard transformation groupoid that gives the irrational rotation $C^*$-algebra. In particular, this example shows that a single $C^*$-algebra can have more than one groupoid that realizes it and is HK-good.  

It is maybe also interesting to note that the homologies of the two groupoid models are not even the same!  Indeed, the model above has 
$$
H_i(\gpd)\cong \left\{\begin{array}{ll} \Z\oplus \Z & i\in \{0,-1\} \\ 0 & \text{otherwise} \end{array}\right.,
$$
while the standard model $\Z\ltimes S^1$ has 
$$
H_i(\Z\ltimes S^1)\cong \left\{\begin{array}{ll} \Z\oplus \Z & i=0 \\ \Z& i\in \{1,-1\} \\ 0 & \text{otherwise} \end{array}\right..
$$
Nonetheless, the homologies agree once one reduces the $\Z$-grading on groupoid homology to a $\Z/2$-grading in the standard way (as they must if both are HK-good).

There is another HK-good groupoid that gives the irrational rotation $C^*$-algebra, and is moreover ample. The relevant groupoid was constructed by Putnam as part of the class of examples in \cite[Theorem 1.1]{Putnam:2018aa}; the proof that it is HK-good is due to Reardon \cite{Reardon:PhDthesis}.  The homology of this groupoid is again different: we get
$$
H_i(\gpd)\cong \left\{\begin{array}{ll} \Z\oplus \Z & i\in \{0,1\} \\ 0 & \text{otherwise} \end{array}\right..
$$
Thus we have three different HK-good groupoid models of the irrational rotation algebra, all with different homology.
\end{example}

\begin{remark} \label{Rem:NotHKgoodCantor}
As was discussed in Remark \ref{rem:HKbad} in the case of the point-like space situation, if we take $Y$ as above but without the assumption that $K^*(Y) \cong H^*(Y)$, then we get groupoids that are not HK-good. Again, for an explicit an example, the reader can consider the case when $Y=\mathbb{R}P^4$.
\end{remark}

\appendix

\section{Groupoid homology and infinite-dimensional base spaces} 

When they defined groupoid homology, Crainic and Moerdijk \cite{Crainic:2000aa} made a finite-dimensionality assumption on the base space of the relevant groupoid.  This assumption is needed so that the machinery of derived functors (see for example \cite[Chapter 10]{Weibel:1995ty} or \cite[Chapter 1]{Kashiwara:1990aa}) works in the expected way: the issue is that the machinery works best for bounded below complexes.  In this appendix we show how to make sense of the definition for groupoids with possibly infinite-dimensional base space (for simplicity, however, and unlike Crainic and Moerdijk, we keep our standing assumption that all groupoids are Hausdorff).

The following lemma is well-known.  See for example \cite[Section I.3]{Bredon:1997aa} for the notion of pullback, or inverse image, sheaf $\tau^*\mathcal{S}$ used in the statement.

\begin{lemma}\label{c soft no dep}
Let $X$ be a locally compact Hausdorff space, and let $\tau:Y\to X$ be an \'{e}tale\footnote{An \emph{\'{e}tale} function is a local homeomorphism.} map.  Let $\mathcal{S}^\bullet$ and $\mathcal{T}^\bullet$ be bounded below cochain complexes of $c$-soft sheaves on $X$, and let 
$$
f:\mathcal{S}^\bullet\to \mathcal{T}^\bullet
$$ 
be a quasi-isomorphism in the category of sheaves\footnote{We mean that $f$ induces isomorphisms on the homology of the complexes as taken in the category of sheaves, not that it induces an isomorphism on sheaf cohomology.}.  Then the map
$$
f:\Gamma_c(Y;\tau^*\mathcal{S}^\bullet)\to \Gamma_c(Y;\tau^*\mathcal{T}^\bullet)
$$
functorially induced by $f$ on complexes of compactly supported sections is also a quasi-isomorphism.
\end{lemma}

\begin{proof}
We first prove the result in the case that $X=Y$ and $\tau:X\to X$ is the identity map.  Let $\mathcal{C}(f)^\bullet$ be the usual mapping cone cochain complex (see for example \cite[1.5.1]{Weibel:1995ty}), so for each $i$, 
\begin{equation}\label{direct sum}
\mathcal{C}(f)^i=\mathcal{S}^{i+1}\oplus \mathcal{T}^i,
\end{equation}
and there is a short exact sequence
\begin{equation}\label{sheaf ses}
0\to \mathcal{T}^\bullet \to \mathcal{C}(f)^\bullet\to  \mathcal{S}^{\bullet+1} \to 0 
\end{equation}
of cochain complexes of sheaves.  Taking sections and applying \cite[Theorem II.9.9]{Bredon:1997aa}, this induces a short exact sequence
\begin{equation}\label{sect ses}
0\to \Gamma_c(X;\mathcal{T}^\bullet) \to \Gamma_c(X;\mathcal{C}(f)^\bullet)\to  \Gamma_c(X;\mathcal{S}^{\bullet+1}) \to 0 
\end{equation}
of cochain complexes of sections.  On the other hand, as $f$ is a quasi-isomorphism, the long exact sequence in cohomology (we mean taken in the category of sheaves, not sheaf cohomology) associated to the short exact sequence in line \eqref{sheaf ses} implies that the complex $\mathcal{C}(f)^\bullet$ is exact.  Noting that each $\mathcal{C}(f)^i$ is $c$-soft (by the formula in line \eqref{direct sum}), \cite[Theorem II.9.11]{Bredon:1997aa} implies that the complex $\Gamma_c(X;\mathcal{C}(f)^\bullet)$ is also exact.  As $\Gamma_c(X;\mathcal{C}(f)^\bullet)$ is the same as the mapping cone of the map induced by $f$ on sections, this last mapping cone complex is also exact.  Hence the boundary map in the long exact sequence on cohomology associated to the short exact sequence in line \eqref{sect ses} is a quasi-isomorphism.  This is exactly the map induced by $f$ on sections, however, so we are done with the case that $\tau:X\to X$ is the identity.

In general, note that the functor $\tau^*$ from sheaves on $X$ to sheaves on $Y$ preserves exact sequences of sheaves (see for example \cite[pages 12-13]{Bredon:1997aa}), and takes mapping cones to mapping cones.   A similar mapping cone argument to the first part therefore shows that the map induced by $f$ from $\tau^*\mathcal{S}^\bullet$ to $\tau^*\mathcal{T}^\bullet$ is a quasi-isomorphism.  Moreover, as $\tau$ is \'{e}tale direct checks show that $\tau^*$ takes $c$-soft sheaves to $c$-soft sheaves. We are thus in the situation of the first part again.
\end{proof}

For the next lemma, we recall that for an \'{e}tale groupoid $\gpd$ and a $\gpd$-sheaf $\mathcal{S}$ on $\gpd^{(0)}$, $\mathcal{S}_p$ denotes the sheaf $\tau^*_p\mathcal{S}$ on $\gpd^{(p)}$, where
\begin{equation}\label{taup}
\tau_p:\gpd^{(p)}\to \gpd^{(0)}, \quad (g_1,...,g_n)\mapsto r(g_1)
\end{equation}
is the \'{e}tale map from \cite[3.1]{Crainic:2000aa} (see also line \eqref{taun map} above).

\begin{lemma}\label{c soft no dep 2}
Let $\gpd$ be an \'{e}tale groupoid, and let $f:(\mathcal{S}^q)_{q\geq 0}\to (\mathcal{T}^q)_{q\geq 0}$ be a quasi-isomorphism of bounded below cochain complexes of c-soft $\gpd$-sheaves on $\gpd^{(0)}$.  Then the map 
$$
f:\Gamma_c(\gpd^{(p)};\mathcal{S}^{-q}_p)_{p,q\geq 0}\to \Gamma_c(\gpd^{(p)};\mathcal{T}^{-q}_p)_{p,q\geq 0}
$$
induced by $f$ on the associated fourth quadrant double complexes induces an isomorphism on the homology of the associated direct sum total complex\footnote{Compare for example \cite[1.2.6]{Weibel:1995ty} for the direct sum total complex associated to a double complex.}.
\end{lemma}

\begin{proof}
Let $E_{pq}^r$ be the double complex spectral sequence associated to the filtration of the double complex by columns (see for example \cite[Sectiuon 5.6]{Weibel:1995ty}).  As we are working with the direct sum total complex and a fourth quadrant double complex, this converges (see for example \cite[page 142]{Weibel:1995ty}).  Hence it suffices to show that the map induced by $f$ on the $r^\text{th}$ page induces an isomorphism for some $r$.  The first page of the spectral sequence has entries $E_{pq}^1=H^{-q}(\Gamma_c(\mathcal{G}^{(p)};\mathcal{S}^\bullet_p))$, i.e.\ in the $p^\text{th}$ column we have the homology of the complex $\Gamma_c(\gpd^{(p)};\mathcal{S}^{-\bullet}_p)$, and the map between spectral sequences induced by $f$ is the map functorially induced by $f$ on these homology groups.  Directly applying Lemma \ref{c soft no dep}, however, the map $f$ induces an isomorphism on these groups, so we are done.
\end{proof}

\begin{lemma}\label{inj ginj}
Let $\gpd$ be an \'{e}tale groupoid, and let $0\to \mathcal{A}^0\to \cdots \to \mathcal{A}^d$ be an exact sequence of $\gpd$-sheaves\footnote{We will mainly be interested in the case of a single sheaf $\mathcal{A}=\mathcal{A}^0$, but will need the more general statement once.} on the base space $\gpd^{(0)}$.  Then there exists an exact sequence 
$$
0\to \mathcal{A}^0\to \cdots \mathcal{A}^d\to  \mathcal{I}^0\to \mathcal{I}^1\to \mathcal{I}^2\to \cdots
$$
of $\gpd$-sheaves such that each $\mathcal{I}^i$ is simultaneously injective in the category of sheaves (whence $c$-soft) and in the category of $\gpd$-sheaves.
\end{lemma}

\begin{remark}
Crainic and Moerdijk \cite[2.1]{Crainic:2000aa} deduce that the category of $\gpd$-sheaves has enough injectives (i.e.\ that resolutions by injective $\gpd$-sheaves always exist) from the fact that the category of $\gpd$-sheaves of sets is a topos.  It may also be possible to deduce Lemma \ref{inj ginj} in topos language, but we did not pursue that, and instead give an elementary proof.  We also remark that we do not know an example of a $\gpd$-sheaf that is injective as a $\gpd$-sheaf, but not as a sheaf: it could therefore be that the existence of enough injectives in the category of $\gpd$-sheaves suffices for Lemma \ref{inj ginj}, but we do not know this.
\end{remark}

\begin{proof}[Proof of Lemma \ref{inj ginj}]
We follow \cite[page 197]{Grothendieck:1957aa}, which (essentially) treats the case of transformation groupoids.  The statement that an injective sheaf is $c$-soft is well-known: see for example from \cite[Proposition II.5.3 and Corollary II.9.6]{Bredon:1997aa}.  Using a standard argument from homological algebra (see for example \cite[2.2.5, and 2.3.6]{Weibel:1995ty}), it suffices to show that any $\gpd$-sheaf $\mathcal{A}$ embeds in a $\gpd$-sheaf $\mathcal{I}$ that is both $\gpd$-injective, and injective .

First, fix a point $x\in \gpd^{(0)}$, and let 
$$
\gpd x:=\{y\in \gpd^{(0)}\mid r(g)=y \text{ for some } g\in \gpd_x\}
$$
be the associated orbit.  Let $\gpd_x^x$ denote the isotropy group at $x$.  Note that the stalk $\mathcal{A}_x$ of $\mathcal{A}$ at $x$ is a module over the group algebra $\Z\gpd_x^x$ of $\gpd_x^x$.  Let $M_x$ denote any choice of injective $\Z\gpd_x^x$-module that contains $\mathcal{A}_x$ (see for example \cite[page 42]{Weibel:1995ty} for existence of such an injective module).  For any other point $y=x$ of the orbit $\gpd x$, choose $g=g_{yx}\in \gpd$ with $y=r(g)$ and $s(g)=x$.  Define $M_y$ to have the same underlying abelian group as $M_x$, and with action of $\Z\gpd_y^y$ determined for $h\in \gpd_y^y$ and $m\in M_x$ by 
$$
h\cdot m:=(g^{-1}hg)m.
$$
Note that $M_y$ is an injective $\Z\gpd_y^y$-module.  If $y,z\in \gpd x$ and $k\in \gpd$ satisfies $r(k)=y$ and $s(k)=z$, we define a corresponding map 
$$
k:M_z\to M_y,\quad m\mapsto g_{yx}^{-1}kg_{zx}m.
$$

Now, carry out the above procedure for every orbit in $\gpd^{(0)}$ to get a family $(M_x)_{x\in \gpd^{(0)}}$ indexed by $x\in \gpd^{(0)}$ such that each $M_x$ is an injective $\Z[\gpd^x_x]$-module, and such that for each $g\in \gpd_x^y$, there is a corresponding isomorphism $g:M_x\to M_y$; moreover, the family of such isomorphisms is compatible with multiplication and identities in $\gpd$.  We then define $\mathcal{I}$ to be the sheaf on $\gpd^{(0)}$ whose sections over an open set $U$ are given by 
$$
\mathcal{I}(U):=\prod_{x\in U}M_x.
$$
We leave it to the reader to check that this is a $\gpd$-sheaf in the natural way, and that $\mathcal{A}$ is a subsheaf.  To show that $\mathcal{I}$ is injective as a $\gpd$-sheaf, note that if $\mathcal{B}$ is a $\gpd$-sheaf, then giving a homomorphism $\mathcal{B}\to\mathcal{I}$ is the same thing as giving a family of homomorphisms $\{\mathcal{B}_x\to M_x\}_{x\in \gpd^{(0)}}$ that is compatible with the groupoid action in the appropriate sense.  Given then an inclusion $\mathcal{B}\subseteq \mathcal{C}$ and a homomorphism $\mathcal{B}\to\mathcal{I}$, we can build a homomorphism $\mathcal{C}\to \mathcal{I}$ one orbit at a time.  Indeed, given $x\in \gpd^{(0)}$, the dashed arrow 
$$
\xymatrix{ \mathcal{C}_x \ar@{-->}[dr]^{f_x}  & \\ \mathcal{B}_x\ar[u] \ar[r] & M_x}
$$
can be filled in with a $\gpd_x^x$-equivariant map using that $M_x$ is injective in the category of $\Z\gpd_x^x$-modules.  Having chosen such a map for $x$, extend it to the whole orbit $\gpd x$ equivariantly, i.e.\ given $y\in \gpd x$ and $g\in \gpd$ with $s(g)=x$ and $r(g)=y$, we define $f_y:\mathcal{C}_y\to M_y$ by $f_y:=g\circ f_x\circ g^{-1}$; this does not depend on the choice of $g$ as $f_x$ is $\gpd_x^x$ equivariant.  This completes the proof that $\mathcal{I}$ is injective as a $\gpd$-sheaf.  

The argument that $\mathcal{I}$ is injective as a sheaf is similar (and simpler, as there is no need to worry about equivariance), once we have observed that each $M_x$ is also injective as a $\Z$-module.  To see this, note that for fixed $m\in M$ and $n\geq 1$ there is a diagram
$$
\xymatrix{ \Z \ar@{-->}[dr]  & \\ \Z\ar[u]^-{\times n} \ar[r]_{1\mapsto m} & M_x }
$$
where the two copies of $\Z$ are considered as trivial $\Z\gpd_x^x$-modules, the vertical map is $z\mapsto nz$ and the horizontal map sends $z$ to $zm$.  Injectivity allows us to fill in the diagonal map, whence there is $d\in M_x$ with $nd=m$.  Hence $M_x$ is divisible as an abelian group, so injective as a $\Z$-module (see for example \cite[Corollary 2.3.2]{Weibel:1995ty}) as required.
\end{proof}

We now come to our main goal in this subsection.

\begin{proposition}\label{cm id wd}
Let $\gpd$ be an \'{e}tale groupoid and let $\mathcal{A}$ be a $\gpd$-sheaf on the base space $\gpd^{(0)}$.  Let 
$$
0\to \mathcal{A}\to \mathcal{S}^0\to \mathcal{S}^1\to \mathcal{S}^2\to \cdots
$$
be a resolution of $\mathcal{A}$ by $c$-soft $\gpd$-sheaves.  Then the homology of the direct sum total complex associated to the double complex $(\Gamma_c(\gpd^{(p)},\mathcal{S}^{-q}_p))_{p,q\geq 0}$ depends only on the quasi-isomorphism class of $\mathcal{S}^\bullet$.

In particular, the Crainic-Moerdijk homology groups are well-defined for all locally compact, Hausdorff, \'{e}tale groupoids.
\end{proposition}

\begin{proof}
Let $\mathcal{I}^\bullet$ be a resolution of $\mathcal{A}$ with the properties in Lemma \ref{inj ginj}.  As $\mathcal{I}^\bullet$ consists of $\gpd$-injective sheaves, a standard argument in homological algebra (see for example \cite[2.7.1]{Weibel:1995ty}) gives a quasi-isomorphism $f:\mathcal{S}^\bullet\to \mathcal{I}^\bullet$.  As the sheaves in $\mathcal{I}^\bullet$ are injective, they are $c$-soft (see for example \cite[Proposition II.5.3 and Corollary II.9.6]{Bredon:1997aa}), whence Lemma \ref{c soft no dep 2} implies that the homologies of the direct sum total complex associated to the double complexes $(\Gamma_c(\gpd^{(p)},\mathcal{S}^{-q}_p))_{p,q\geq 0}$ and $(\Gamma_c(\gpd^{(p)},\mathcal{I}^{-q}_p))_{p,q\geq 0}$ are the same.  As $\mathcal{I}^\bullet$ is independent of $\mathcal{S}^\bullet$, we are done.

The last statement is an immediate consequence of the definition of the Crainic-Moerdijk groups.
\end{proof}

\section{Group hyperhomology}\label{hyp hom sec}

In this appendix we recall two definitions of group hyperhomology, and relate them to each other, and to groupoid homology of the associated transformation groupoid.  

We will use the following as our `official' definition of group hyperhomology; we compare it with another, perhaps more standard, definition below.  The first definition is based on  \cite[Section VII.5]{Brow:1982rt}, although we drop the assumption stated there that the chain complex $C_\bullet$ is bounded below.  For now the definition is provisional; we show that it does not depend on the choices involved in Lemma \ref{h1 wd} below.

\begin{definition}\label{hyper1}
Let $G$ be a discrete group, and let $C_\bullet$ be a chain complex of $\Z G$-modules.  Let $(P_p)_{p\geq 0}$ be a resolution of the trivial $\Z G$-module $\Z$ by projective $\Z G$-modules, and let $(P_p\otimes_{\Z G}C_q)_{p\geq 0,q \in \Z}$ be the associated tensor product\footnote{Compare for example \cite[2.7.1]{Weibel:1995ty}.} (right-half-plane) double complex.  

The \emph{group hyperhomology of $G$ with coefficients in $C_\bullet$}, denoted $\mathbb{H}_*(G;C_\bullet)$ is defined to be the homology of the direct sum total complex associated to this double complex.
\end{definition}

\begin{lemma}\label{h1 wd}
With notation as in Definition \ref{hyper1}, $\mathbb{H}_*(G;C_\bullet)$ does not depend on the choice of projective resolution $P_\bullet$ used to define it up to canonical isomorphism, and depends on $C_\bullet$ only up to quasi-isomorphism.
\end{lemma}

\begin{proof}
Any two projective resolutions of $\Z$, say $(P_p)$ and $(Q_p)$, are chain homotopy equivalent: see for example \cite[2.2.6]{Weibel:1995ty}.  Such a chain homotopy equivalence gives rise to a chain homotopy equivalence of the total complexes associated to the tensor product double complexes; this gives independence of the projective resolution.  

To see that $\mathbb{H}_*(G;C_\bullet)$ only depends on $C_\bullet$ up to quasi-isomorphism\footnote{For an alternative argument avoiding spectral sequences, see the proof of \cite[Theorem I.8.6]{Brow:1982rt}.} , as in \cite[Section 5.6]{Weibel:1995ty} there is a spectral sequence 
$$
E^1_{pq}=H^{(v)}_q(P_p\otimes_{\Z G} C_\bullet)~\Rightarrow~ \mathbb{H}_{p+q}(G;C_\bullet)
$$
where $H^{(v)}$ denotes `homology in the vertical direction'; convergence is justified as on \cite[page 142]{Weibel:1995ty}, noting that our double complex avoids the second quadrant.  As the functor $M\mapsto P_p\otimes_{\Z G}M$ is exact by projectivity of $P_p$, we have canonical isomorphisms $H^{(v)}_q(P_p\otimes_{\Z G} C_\bullet)\cong P_p \otimes_{\Z G} H_q(C_\bullet)$.  Hence a quasi-isomorphism of $C_\bullet$ with another complex $D_\bullet$ induces an isomorphism on the $E^1$-page of our spectral sequence; as the spectral sequence converges to the hyperhomology groups $\mathbb{H}_*(G;C_\bullet)$, this suffices.
\end{proof}

For completeness, let us compare the definition of hyperhomology in Definiton \ref{hyper1} above with another definition which is perhaps more standard; it is more in the spirit of derived categories as opposed to the direct definition above.  Again, the definition is provisional for now: we justify it in Lemma \ref{h2 wd} below.

\begin{definition}\label{hyper2}
Let $G$ be a discrete group, and let $C_\bullet$ be a bounded below\footnote{It is not clear to us that this definition makes sense in the absence of the `bounded below' assumption; this is our reason for preferring the version in Definition \ref{hyper1} above.} chain complex of $\Z G$-modules.  Choose a quasi-isomorphism $f:Q_\bullet\to C_\bullet$ where $Q_\bullet$ is a bounded below complex of projective modules.

We define the \emph{second hyperhomology groups of $G$ with coefficients in $C_\bullet$}, denoted $\mathbb{H}^{(2)}_*(G;C_\bullet)$, to be the homology of the associated complex $(\Z\otimes_{\Z G} Q_q)_{q\in \Z}$ of coinvariants.
\end{definition}

\begin{lemma}\label{h2 wd}
Let $G$ be a discrete group, and let $C_\bullet$ be a bounded below chain complex of $\Z G$-modules.  With notation as in Definition \ref{hyper2}, the second hyperhomology $\mathbb{H}^{(2)}_*(G;C_\bullet)$ is well-defined, depends on $C_\bullet$ only up to quasi-isomorphism, and is canonically isomorphic to the hyperhomology $\mathbb{H}_*(G;C_\bullet)$ of Definition \ref{hyper1}.
\end{lemma}

\begin{proof}
We first need to show that there is a quasi-isomorphism $f:Q_\bullet\to C_\bullet$ with $Q_\bullet$ consisting of projective modules and bounded below.  This can be proved in exactly the same way as \cite[Proposition 1.7.7 (i)]{Kashiwara:1990aa}: that reference is for bounded below \emph{co}chain complexes of \emph{injective} modules, but the same proof works mutatis mutandis on reversing the arrows\footnote{One could also use a Cartan-Eilenberg resolution of $C_\bullet$ as in \cite[Section 5.7]{Weibel:1995ty}.  In either case, the bounded below assumption seems important.}.  If moreover $P_\bullet$ and $Q_\bullet$ are bounded below chain complexes of projective modules equipped with quasi-isomorphisms to $C_\bullet$, then they are chain-homotopy equivalent to each other by the proof of \cite[Lemma 10.4.6 and Theorem 10.4.8]{Weibel:1995ty} (again, this is stated for cochain complexes of injective modules, but the same idea works).  Such a chain homotopy equivalence induces a chain homotopy equivalence of chain complexes of coinvariants, so we are done with showing that $\mathbb{H}^{(2)}_*(G;C_\bullet)$ is well-defined.  It is immediate that it only depends on $C_\bullet$ up to quasi-isomorphism.

Let us now show that $\mathbb{H}^{(2)}_*(G;C_\bullet)$ and $\mathbb{H}_*(G;C_\bullet)$ are isomorphic.  Using the first part of the proof and the fact that both sorts of hyperhomology only depend on $C_\bullet$ up to quasi-isomorphism, we may assume that $C_\bullet$ consists of projective modules.  As $C_\bullet$ is bounded below, we have moreover that there is $d\in \Z$ such that $C_q=0$ for all $q<d$.  Consider the augmented double complex 
$$
\xymatrix{ \ar[d] & \ar[d] & \ar[d] & \\ \Z\otimes_{\Z G} C_{d+2} \ar[d] & P_0\otimes_{\Z G} C_{d+2} \ar[l] \ar[d] &  P_1\otimes_{\Z G} C_{d+2} \ar[l] \ar[d] & \ar[l] \\
\Z\otimes_{\Z G} C_{d+1} \ar[d] & P_0\otimes_{\Z G} C_{d+1} \ar[l] \ar[d] &  P_1\otimes_{\Z G} C_{d+1} \ar[l] \ar[d] & \ar[l]\\
\Z\otimes_{\Z G} C_{d}  & P_0\otimes_{\Z G} C_{d} \ar[l]  &  P_1\otimes_{\Z G} C_{d} \ar[l] &\ar[l] }.
$$ 
The first column computes $\mathbb{H}^{(2)}_*(G;C_\bullet)$ and the double complex to the right of the first complex computes  $\mathbb{H}_*(G;C_\bullet)$.  As each $C_q$ is projective, the rows are exact.  Hence the acyclic assembly lemma (see for example \cite[2.7.3]{Weibel:1995ty}) implies that the augmented double complex has trivial homology.  On the other hand, as in the proof of \cite[Theorem 2.7.2]{Weibel:1995ty}, the total complex of the augmented double complex is the mapping cone of the augmentation map from the total complex of the right hand double complex to the complex in the first column (up to translation).  Acyclicity of this mapping cone implies that the augmentation map is an isomorphism, giving the result.
\end{proof}

We now relate group hyperhomology with appropriate coefficients to Crainic-Moerdijk groupoid homology.  We start with the following definition; it can be usefully compared with that of equivariant homology (see for example \cite[Section VII.7]{Brow:1982rt}).  As usual, it is provisional: see Lemma \ref{equi well def} for a proof that it does not depend on the choices involved.

\begin{definition}\label{hyper3}
Let $X$ be a locally compact Hausdorff space equipped with an action of a discrete group $G$.  Let 
\begin{equation}\label{long sheaf res}
0\to \mathcal{Z}\to \mathcal{S}^0\to \mathcal{S}^1\to \cdots 
\end{equation}
be a resolution of the sheaf $\mathcal{Z}$ of locally constant $\Z$-valued functions on $X$ by $c$-soft $G$-sheaves.   We define $\mathbb{H}_*(G;X)$ to be the hyperhomology $\mathbb{H}_*(G;\Gamma_c(X;\mathcal{S}^{-\bullet}))$.
\end{definition}

\begin{lemma}\label{equi well def}
The homology groups $\mathbb{H}_*(G;X)$ do not depend on the choice of resolution as in line \eqref{long sheaf res} up to canonical isomorphism.
\end{lemma}

\begin{proof}
Lemma \ref{inj ginj} gives us a resolution of $\mathcal{Z}$ by a complex $\mathcal{I}^\bullet$ of injective $G$-sheaves on $X$ that are also injective (whence $c$-soft) as sheaves.  As in the proof of Proposition \ref{cm id wd}, there is a quasi-isomorphism $f:\mathcal{S}^\bullet\to \mathcal{I}^\bullet$ which induces a quasi-isomorphism $f:\Gamma_c(X;\mathcal{S}^\bullet)\to \Gamma_c(X;\mathcal{I}^\bullet)$ by Lemma \ref{c soft no dep 2}.  As $\mathcal{I}^\bullet$ is independent of the choice of $\mathcal{S}$, the result thus follows from Lemma \ref{h1 wd}.
\end{proof}

If $X$ is compact, we now define a pairing between $\mathbb{H}_0(G;X)$ as in Definition \ref{hyper3} and the simplex $T(G\ltimes X)$ of invariant probability  measures on $X$.  Assume that a resolution as in line \eqref{long sheaf res} is a Borel resolution in the sense of Definition \ref{top res}.  Let $\tau$ be an invariant probability measure on $X$.  Choose a projective resolution $(P_p)_{p\geq 0}$ of $\Z$ by projective $\Z G$-modules with $P_0=\Z G$, and the map $P_0\to \Z$ the standard augmentation map defined by summing coefficients.  Then cycles for $\mathbb{H}_0(G;C_\bullet)$ are represented by tuples $(x_0,x_1,...)$ with $x_i\in P_i\otimes_{\Z G} \Gamma_c(X;\mathcal{S}^i)$, such that only finitely many $x_i$ are non-zero, and such that $(x_0,x_1,...)$ goes to zero under the boundary map for the total complex.  In particular, $x_0\in P_0\otimes_{\Z G} \Gamma_c(X;\mathcal{S}^0)=\Gamma_c(X;\mathcal{S}^0)$ is a complex-valued and compactly supported Borel function on $X$.  We define 
$$
\tau:\mathbb{H}_0(G;C_\bullet)\to \C,\quad [x_0,x_1,...]\mapsto \int_X x_0d\tau.
$$
The following lemma follows from almost the same argument as Proposition \ref{pair good} and is left to the reader.

\begin{lemma}\label{good hh pair}
The pairing defined above between the simplex $T(G\ltimes X)$ of invariant probability measures on $X$ and $\mathbb{H}_0(G;C_\bullet)$ is well-defined and does not depend on the choice of Borel resolution. \qed
\end{lemma}

Our final goal in this appendix is to show that if $\gpd=G\ltimes X$ is a transformation groupoid associated to a locally compact Hausdorff space $X$ equipped with an action of a discrete group $G$, then there is a canonical isomorphism $H_*(G\ltimes X)\cong \mathbb{H}_*(G;X)$ between the Crainic-Moerdijk homology and the group hyperhomology, and moreover that this respects the pairing with $T(G\ltimes X)$ in case $X$ is compact.  To set this up, write $P_n$ for the free $\Z G$-module consisting of all finitely supported functions $c:G^{n+1}\to \Z$, where the $\Z G$-action is determined by the diagonal left translation action of $G$ on $G^{n+1}$.  There is then a resolution of the trivial $\Z G$-module $\Z$
$$
0\leftarrow \Z\leftarrow P_0 \leftarrow P_1 \leftarrow\cdots
$$
where the map $P_0\to \Z$ is given by $c\mapsto \sum_{g\in G} c(g)$, and the map $\partial: P_{n+1}\to P_{n}$ is given by $\partial=\sum_{i=0}^{n+1} (-1)^{i}\partial_i$, where 
$$
(\partial_i c)(g_0,...,g_n):=\sum_{g\in G} c(g_0~,...,\underbrace{g}_{i^\text{th}\text{ place}},...,~g_n)
$$ 
(compare for example \cite[Section I.5]{Brow:1982rt}).

\begin{lemma}\label{iso}
Let $\gpd=G\ltimes X$ be a transformation groupoid associated to an action of a discrete group $G$ on a locally compact Hausdorff space $X$.  Let $\mathcal{S}$ denote a $\gpd$-sheaf on $\gpd^{(0)}$, and let $\mathcal{S}_p:=\tau_p^*\mathcal{S}$ be the associated pullback sheaf on $\gpd^{(p)}$ (see line \eqref{taun map} above, or \cite[3.1]{Crainic:2000aa}).  Then the Crainic-Moerdijk `bar complex'
$$
0\leftarrow \Gamma_c(\gpd^{(0)};\mathcal{S}_0) \leftarrow \Gamma_c(\gpd^{(1)};\mathcal{S}_1)  \leftarrow \Gamma_c(\gpd^{(2)};\mathcal{S}_2) \leftarrow \cdots 
$$
of \cite[3.1]{Crainic:2000aa} is isomorphic to the complex 
\begin{equation}\label{stan res}
0\leftarrow P_0\otimes_{\Z G} \Gamma_c(X;\mathcal{S})\leftarrow  P_1\otimes_{\Z G} \Gamma_c(X;\mathcal{S})\leftarrow  P_2\otimes_{\Z G} \Gamma_c(X;\mathcal{S}) \leftarrow \cdots .
\end{equation}
Moreover, this isomorphism can be chosen natural in $\mathcal{S}$.
\end{lemma}

\begin{proof}
For notational simplicity, let us write $M$ for the $G$-module $M:=\Gamma_c(X;\mathcal{S})$ and $\alpha$ for the action on this module.  We first note that for any $n$, $\Gamma_c(\gpd^{(n)};\mathcal{S}_n)$ identifies with the space of finitely supported functions $f:G^n\to \Gamma_c(X;\mathcal{S}_n)$, and that having made this identification, the Crainic-Moerdijk face map $d_i:\Gamma_c(\gpd^{(n+1)};\mathcal{S}_{n+1})\to \Gamma_c(\gpd^{(n)};\mathcal{S}_{n})$ is given by 
$$
(d_i f)(g_1,...,g_n)=\left\{\begin{array}{ll}  \sum_{g\in G} \alpha_gf(g,g_1,...,g_n) & i=0  \\ \sum_{g\in G} f(g_1,...,g,g^{-1}g_i,...,g_n) & 0<i<n \\ \sum_{g\in G} f(g_1,...,g_n,g) & i=n+1 \end{array}\right.
$$
(in the case $0<i<n$, the $g$ occurs in the $i^{\text{th}}$ place, and the $g^{-1}g_i$ in the $(i+1)^\text{th}$ place).  

On the other hand, for each $n$, $P_n\otimes_{\Z G} M$ identifies with the space of finitely supported functions $f:G^n\to M$ by the map sending the elementary tensor $c\otimes m$ to the function 
$$
f(g_1,...,g_n):=\sum_{h\in G} c(h,hg_1,hg_1g_2,hg_1g_2g_3,...,hg_1g_2...g_n)\alpha_h(m)
$$
Having made these identifications, direct computations show that the face maps $\partial_i$ and $d_i$ match up, completing the proof.
\end{proof}

The following corollary is now almost immediate from the definitions of Crainic-Moerdijk homology, and of $\mathbb{H}_*(G;C_\bullet)$ as in Definition \ref{hyper1} above.

\begin{corollary}\label{cm hh}
Let $\gpd=G\ltimes X$ be a transformation groupoid associated to an action of a discrete group $G$ on a locally compact Hausdorff space $X$.  Then there is a canonical isomorphism 
$$
H_*(\gpd)\cong \mathbb{H}_*(G;X)
$$
between the Crainic-Moerdijk homology and the group hyperhomology of Definition \ref{hyper3}.  Moreover, the isomorphism is compatible with the pairings of $H_0(\gpd)$ and of $\mathbb{H}_0(G;X)$ with invariant probability measures.
\end{corollary}

\begin{proof}
Let $(\mathcal{S}^q)_{q\geq 0}$ be a resolution of the $\gpd$-sheaf $\mathcal{Z}$ of locally constant $\Z$-valued functions on $X$ by $c$-soft $\gpd$-sheaves.  The Crainic-Moerdijk homology $H_*(\gpd)$ is by definition the homology of the double complex
$$
\xymatrix{ \Gamma_c(\gpd^{(0)};\mathcal{S}_0^0) \ar[d]& \ar[l] \ar[d] \Gamma_c(\gpd^{(1)};\mathcal{S}_1^0)  & \ar[l] \ar[d] \Gamma_c(\gpd^{(2)};\mathcal{S}_2^0) & \ar[l]   \\
\Gamma_c(\gpd^{(0)};\mathcal{S}_0^1)\ar[d]& \ar[l] \ar[d] \Gamma_c(\gpd^{(1)};\mathcal{S}_1^1)  & \ar[l] \ar[d] \Gamma_c(\gpd^{(2)};\mathcal{S}_2^1) & \ar[l]   \\
\Gamma_c(\gpd^{(0)};\mathcal{S}_0^2) \ar[d]& \ar[l] \ar[d] \Gamma_c(\gpd^{(1)};\mathcal{S}_1^2)  & \ar[l] \ar[d] \Gamma_c(\gpd^{(2)};\mathcal{S}_2^2) & \ar[l]   \\ & & & }
$$
where the vertical differentials are induced from the resolution and the horizontal differentials are as in the Crainic-Moerdijk bar resolution.  On the other hand, the hyperhomology $\mathbb{H}_*(G;X)$ is by definition the homology of the double complex
$$
\xymatrix{ P_0\otimes_{\Z G} \Gamma_c(X;\mathcal{S}^0) \ar[d] &\ar[l] \ar[d]  P_1\otimes_{\Z G} \Gamma_c(X;\mathcal{S}^0) &\ar[l] \ar[d]  P_2\otimes_{\Z G} \Gamma_c(X;\mathcal{S}^0)  &\ar[l]  \\
P_0\otimes_{\Z G} \Gamma_c(X;\mathcal{S}^1)  \ar[d]&\ar[l] \ar[d]  P_1\otimes_{\Z G} \Gamma_c(X;\mathcal{S}^1) &\ar[l] \ar[d]  P_2\otimes_{\Z G} \Gamma_c(X;\mathcal{S}^1)  &\ar[l]  \\
P_0\otimes_{\Z G} \Gamma_c(X;\mathcal{S}^2) \ar[d] &\ar[l] \ar[d]  P_1\otimes_{\Z G} \Gamma_c(X;\mathcal{S}^2) &\ar[l] \ar[d]  P_2\otimes_{\Z G} \Gamma_c(X;\mathcal{S}^2)  & \ar[l]  \\ & & & }
$$
where the horizontal differentials are induced from those in line \eqref{stan res}, and the vertical differentials are induced from the resolution.  Thanks to Lemma \ref{iso}, these double complexes are canonically isomorphic, so we are done with the isomorphisms.  Compatibility with the pairings follows directly.
\end{proof}

\section{Integral Chern characters}

In this appendix we discuss some examples where the Chern character is compatible with an integral isomorphism between $K$-theory and cohomology, as opposed to `just' being a rational isomorphism.  The material in this section is classical and no doubt well-known; we could not find appropriate references in the literature, however, so provide brief arguments.  We have made no attempt to get optimal results.  For background material on Chern classes as used here, we recommend \cite{Milnor:1974zf}.

We start with an ad-hoc definition.

\begin{definition}\label{ch int}
Let $X$ be a locally compact Hausdorff space.  We define an \emph{integral Chern isomorphism} for $X$ to be any (graded) isomorphism $\text{ch}_\Z:K^*(X)\to H^{**}_c(X)$ that is natural for self-homeomorphisms of $X$ and is compatible with the usual Chern character in the following sense.  The diagram below
$$
\xymatrix{ K^*(X)\ar[r]^-{\text{ch}_\Z} \ar[dr]_-{\text{ch}} & H^{**}_c(X) \ar[d] \\ & H^{**}_c(X;\Q) },
$$
where the diagonal map is the usual Chern character, and the vertical map is the comparison map from integral to rational cohomology, should commute.
\end{definition}  

\begin{example}\label{sphere}
The usual Chern character canonically induces an integral Chern isomorphism for any sphere: this follows for example from \cite[Theorem V.3.25]{Karoubi:1978ai}.  It follows from this and the K\"{u}nneth formulas for $K$-theory and cohomology that there is an integral Chern isomorphism for any (finite) product of spheres, and in particular for any torus.
\end{example}

The following definition starts to build a candidate for an integral Chern isomorphism.  

\begin{definition}\label{ch basic}
Let $X$ be a compact Hausdorff space, and let $V(X)$ denote the monoid of isomorphism classes of vector bundles over $X$.  We define the \emph{Chern class map} $c:V(X)\to H^{ev}(X)$ as follows.  Let $V$ be a vector bundle over $X$.  Then the component of $c(V)$ in $H^0(X)$ is defined to be the class of the $\Z$-valued function $\text{rank}(V)$ associating to each $x\in X$ the dimension of the fiber $V_x$.  For $n\geq 1$, the component of $c(V)$ in $H^{2n}(X)$ is $(-1)^{n-1}c_n(V)$, where $c_n$ is the $n^\text{th}$ Chern class\footnote{The class $c(V)$ is closely related to the \emph{total Chern class} $1+c_1(V)+c_2(V)+\cdots$ as in \cite[page 158]{Milnor:1974zf}, but it differs in dimension zero, and in the choice of signs.}.
\end{definition}

We need another ad-hoc definition.

\begin{definition}
Let $X$ be a locally compact Hausdorff space.  We say that $X$ has \emph{trivial (even) cup products} if all the cup product maps 
$$
\smile:H^i_c(X)\times H^j_c(X)\to H_c^{i+j}(X)
$$
are zero whenever $i,j>0$ (and both are even).
\end{definition}

We record a basic fact from algebraic topology: see for example \cite[Exercise 2 on page 228]{Hatcher:2002ud} (and the hints given there) for a proof.

\begin{lemma}\label{no cup}
If $X=SY$ is a suspension, then $X$ has trivial cup products.  \qed
\end{lemma}

The following lemma is immediate from the Whitney sum formula \cite[page 167]{Milnor:1974zf}.

\begin{lemma}\label{chb hom}
Let $X$ be a compact Hausdorff space with trivial even cup products.  Then the Chern class map $c:V(X)\to H^{ev}(X)$ of Definition \ref{ch basic} is a monoid homomorphism.   \qed
\end{lemma}

Hence if $X$ has trivial even cup products, then by the universal property of the Grothendieck group, $c$ uniquely determines a homomorphism $c:K^0(X)\to H^{ev}(X)$.  Considering one-point compactifications, it also induces a homomorphism $c:K^0(X)\to H^{ev}_c(X)$ for a locally compact Hausdorff space $X$.  Finally, we  define the Chern class map for odd parity to be the bottom horizontal map in the diagram below
$$
\xymatrix{ K^0(SX) \ar[r]^-c  &  H^{ev}_c(SX) \ar[d] \\
K^1(X) \ar[u]  \ar@{-->}[r] &  H^{od}_c(X) }
$$
where the two vertical maps are suspension isomorphisms (note that the top horizontal map exists by Lemma \ref{no cup}). 

\begin{definition}\label{chb2}
For a locally compact Hausdorff space $X$ with trivial even cup products, the \emph{Chern class map} is the homomorphism $c:K^*(X)\to H_c^{**}(X)$ just defined.
\end{definition}

\begin{example}\label{low ch int}
Let $X$ be a finite (possibly disconnected) CW complex with dimension at most three.  Then the map $c$ of Definition \ref{chb2} is an integral Chern isomorphism for $X$ in the sense of Definition \ref{ch int}.  Indeed, in low degrees the rational even Chern character is given by 
$$
\text{ch}(V)=\text{rank}(V)+c_1(V)+\frac{1}{2}(c_1(V)^2-2c_2(V))+\text{higher order terms}.
$$
The assumptions imply that the higher order terms are zero (they live in $H^k(X)$ for $k>4$), and that $c_1(V)^2=0$ whence the diagram
$$
\xymatrix{ K^0(X)\ar[r]^-{c} \ar[dr]_-{\text{ch}} & H^{ev}(X) \ar[d] \\ & H^{ev}(X;\Q) }
$$
commutes.  The analogous diagram in odd degrees commutes similarly, now using Lemma \ref{no cup}.  One checks that it is an integral isomorphism by induction on the number of cells, a Mayer-Vietoris argument, and the case of spheres as in Example \ref{sphere}.  
\end{example}

\begin{example}\label{low ch int 2}
Let $X$ be a compact Hausdorff space of covering dimension at most three.  Then $X$ can be written as the limit of an inverse system $(Y_i)$ of finite CW complexes of dimension at most three.  The Chern class map $c$ is an integral Chern isomorphism for each $Y_i$ by Example \ref{low ch int}, and therefore it is an integral Chern isomorphism for $X$ by continuity of $K$-theory, continuity of sheaf cohomology, and continuity of the maps $c$ and $\text{ch}$.
\end{example}

\begin{example}
Let $X=S^6$.  Then $X$ has trivial (even) cup products, so the Chern class map $c:K^*(X)\to H_c^{**}(X)$ is well-defined.  However, it is not an integral isomorphism, and is not compatible with the (rational) Chern character.  In fact, the map $[V]\mapsto \text{rank}(V)+ \frac{1}{2}c_3(V)$ defines an integral Chern isomorphism in that case (compare Example \ref{sphere} and the references given there).
\end{example}


\end{document}